\documentclass[11pt]{article}
\usepackage[applemac]{inputenc}
\usepackage[T1]{fontenc}
\usepackage{lmodern}
\usepackage[a4paper]{geometry}
\usepackage[english]{babel}
\usepackage{graphicx}
\usepackage{onlyamsmath}
\usepackage{amsfonts}
\usepackage{amsthm}
\usepackage{amssymb}

\usepackage{color}

\newcommand{\be}{\begin{equation}}
\newcommand{\ee}{\end{equation}}
\newcommand{\bea}{\begin{eqnarray}}
\newcommand{\eea}{\end{eqnarray}}
\newcommand{\bee}{\begin{eqnarray*}}
\newcommand{\eee}{\end{eqnarray*}}

\newcommand{\BB}{\mathcal{B}}
\newcommand{\E}{\mathbb{E}}
\newcommand{\Rd}{{\mathbb{R}^2}}
\newcommand{\Rm}{{\mathbb{R}}}
\newcommand{\Pro}{\mathbb{P}}
\newcommand{\e}{\varepsilon}
\newcommand{\eps}{\varepsilon}

\def\fref#1{{\rm (\ref{#1})}}

\newcommand{\calE}{\mathcal{E}}
\newcommand{\calC}{\mathcal{C}}
\newcommand{\calH}{\mathcal{H}}
\newcommand{\calI}{\mathcal{I}}
\newcommand{\calJ}{\mathcal{J}}
\newcommand{\calL}{\mathcal{L}}
\newcommand{\calP}{\mathcal{P}}
\newcommand{\calQ}{\mathcal{Q}}

\newcommand{\calS}{\mathcal{S}}
\newcommand{\calX}{\mathcal{X}}
\newcommand{\calT}{\mathcal{T}}

\newcommand{\calK}{\mathcal{K}}

\newcommand{\bx}{\mathbf{x}}

\newtheorem{defi}{Definition}[section]

\newtheorem{thm}{Theorem}[section]
\newtheorem{prop}{Proposition}[section]
\newtheorem{cor}{Corollary}[section]
\newtheorem{lem}{Lemma}[section]

\title{Fractional White\,-\,Noise Limit and Paraxial Approximation for Waves in Random Media}

\author{Christophe Gomez\thanks{Aix Marseille Universit{\'e}, CNRS, Centrale Marseille, I2M, UMR 7373, 13453 Marseille, France, christophe.gomez@univ-amu.fr.} \quad and \quad Olivier Pinaud\thanks{Department of Mathematics, Colorado State University, Fort Collins, CO, USA, pinaud@math.colostate.edu} }
\date{}

\begin{document}

\maketitle

\begin{abstract}
This work is devoted to the asymptotic analysis of high frequency wave propagation in random media with long-range dependence. We are interested in two asymptotic regimes, that we investigate simultaneously: the paraxial approximation, where the wave is collimated and propagates along a privileged direction of propagation, and the white-noise limit, where random fluctuations in the background are well approximated in a statistical sense by a fractional white noise. The fractional nature of the fluctuations is reminiscent of the long-range correlations in the underlying random medium. A typical physical setting is laser beam propagation in turbulent atmosphere. Starting from the high frequency wave equation with fast non-Gaussian random oscillations in the velocity field, we derive the fractional It\^o-Schr\"odinger equation, that is a Schr\"odinger equation with potential equal to a fractional white noise. The proof involves a fine analysis of the backscattering and of the coupling between the propagating and evanescent modes. Because of the long-range dependence, classical diffusion-approximation theorems for equations with random coefficients do not apply, and we therefore use moment techniques to study the convergence. 
\end{abstract}

\section{Introduction} \label{intro}

Problems related to wave propagation in random media are encountered in many applications that range from imaging the earth's crust in geophysics \cite{claerbout}, to communication in underwater acoustics \cite{trappert} or laser beam propagation in the atmosphere \cite{fannjiang,strohbehn}. The random medium often models a complex medium for which only partial information is known. Typically, the large-scale variations of the medium (i.e. the background) are known, while the small-scale fluctuations (i.e. the heterogeneities) might be too difficult to estimate and are considered as random. 

In these applications, waves are generally in a high frequency regime, with frequencies sufficiently high so that the interaction of the wave with the fine structures of the medium cannot be ignored. From both the theoretical and numerical perspectives, describing the cumulative effects of this interaction is a very challenging task. There is therefore a need for an approximate, but still accurate, description of the wave propagation. The common strategy to attack the problem is based on the high frequency assumption and on asymptotic theories of random ODEs or PDEs. There is now a vast literature on this matter, and we refer to \cite{book} and the references therein for more details.

In this work, we are interested in two particular asymptotic limits, that we intend to perform at once. The first one is the paraxial (parabolic) approximation, which is valid when the wave has a privileged direction of propagation and is sufficiently collimated. In the frequency picture, the $d-$dimensional Helmholtz equation is reduced to the $(d-1)$-dimensional Schr\"odinger equation where the time variable plays the role of the variable along the axis of propagation. There is a significant gain since a boundary value problem is replaced by an evolution problem with lower spatial dimensions. In homogenous media, the derivation of the paraxial wave equation is relatively straightforward, and is based on asymptotics of the principal symbol of the operator describing the propagation (here that of the scalar wave equation). The situation is much more complex when the medium is heterogeneous since the interaction with the medium generates some backscattering. One has then to resort to some particular features of the medium, for instance small amplitude of the fluctuations \cite{bamberger}, or oscillatory behavior \cite{bailly,garnier0}, to justify the approximation. 


The second type of limit is of probabilistic nature, and depends on the correlation structure of the fluctuations. After the high frequency wave has propagated over sufficiently large distances in the random medium, it is natural to expect some sort of universal statistical behavior to describe the multiple scattering on the wavefield. We are naturally thinking here of applications of the (non-)central limit theorem. There is also a vast literature on this subject, see for instance \cite{book,taqqu1,taqqu}. In our context of the paraxial approximation, the random medium fluctuations are then asymptotically statistically equivalent to a white noise in the main direction of propagation (say $z$). This holds when the medium has sufficiently fast decaying correlations. The limiting model, known as the  It\^o-Schrodinger equation, is studied mathematically in \cite{dawson}.

When the starting point is the wave equation, or equivalently the Helmholtz equation, there are, to the best of our knowledge, only two references on the coupled paraxial-white noise limit: in \cite{bailly}, the authors consider the random Helmholtz equation in layered media and derive the It\^o-Schrodinger equation. Layered media are a nice setting since the dynamics is essentially one-dimensional and the transverse variables play little role. In this latter work, fluctuations of the medium in the transverse direction are too slow to have a significant effect, and the resulting white noise only depends on $z$. The cumulative effect of the random fluctuations on the wave is then a random phase shift driven by a Brownian motion. In \cite{garnier0}, the medium is much more general, and sufficiently complex to lead to a white noise in $z$ with transverse dependence. The cumulative effect is then more complicated and not just a phase shift.

These two references assume that the medium has short-range correlations. It is not always the case in practice, as is pointed out in \cite{dolan,holm,sidi} for geophysical problems, wave propagation in turbulent atmosphere, or medical imaging. This has then stimulated recent mathematical works on wave propagation in random media with long-range dependence \cite{bal2,garnier,gomez,gomez3,gomez4,marty,marty2}. It is shown there that the wave dynamics in such media can be in great contrast with that of waves in media with rapidly decaying correlations. For instance, anomalous diffusion phenomena were exhibited in \cite{garnier,gomez,gomez3}. 

The goal of this paper is to derive rigorously and simultaneously the paraxial and the white-noise approximations in the context of random media with slowly decaying correlations in the $z$ direction. Heuristically, the limiting classical white noise is replaced by a fractional white noise, leading to the \emph{fractional It\^o-Schr\"odinger equation}. From the mathematical viewpoint, this is a significantly more difficult problem than the ones addressed in \cite{bailly,garnier0}. Indeed, in the long-range case, the martingale techniques of \cite{bailly,garnier0} and standard diffusion-approximation theorems for ODEs with random coefficients do not apply. There is essentially no general theory in this long-range setting, and we are thus restricted to the use of moments techniques which are fairly involved analytically. Note as well that the existence theory for the fractional It\^o-Schr\"odinger equation is not direct, which leads to some additional difficulties in the asymptotic theory.

Let us be more specific now and introduce the scalar wave equation in the physical space $\mathbb{R}^{3}$ (the setting could be extended to $\Rm^d$, $d \geq 2$, since the techniques used in the paper are dimension independent),
\begin{equation}\label{waveeq}
\Delta P-\frac{1}{c^2(z,x)}\partial^2_t P = \nabla \cdot \mathbf{F}\qquad (t,z,x)\in (0,+\infty)\times\mathbb{R}\times\Rd,
\end{equation}
equipped with initial conditions 
\[P(t=0,z,x)=\partial_t P(t=0,z,x)=0\qquad \forall (z,x)\in \Rm\times \Rm^2.\] 
Above, the $z$-direction will play the role of the main propagation axis, $\Delta=\partial^2_z+\Delta_x$ is the Laplacian, and $\Delta_x$ the Laplacian with respect to the transverse variable $x$. Here, the forcing term $\mathbf{F} (t,z,x)$ has the form
\[  \textbf{F}(t,z,x):=f_0\Big(\frac{t}{\lambda_0},\frac{x}{r_0}\Big)\delta(z-L_S)\textbf{e}_{z},\]
where $\delta$ is the Dirac measure, $\textbf{e}_{z}$ is the unit vector pointing in the $z$-direction, and then models a source located in the plane $z=L_S<0$, emitting a wave in the $z$-direction with profile $f_0$ (see Figure \ref{dessin1}), central wavelength $\lambda_0$, and transverse width $r_0$. The divergence form of the source term is standard in linear acoustics where $P$ would represent the pressure wave, see \cite{garnier0} for instance. Other types of sources could be considered with minor modifications. 
In \eqref{waveeq}, the velocity field is assumed to be given by
\begin{figure}
\begin{center}
\includegraphics*[scale=0.3]{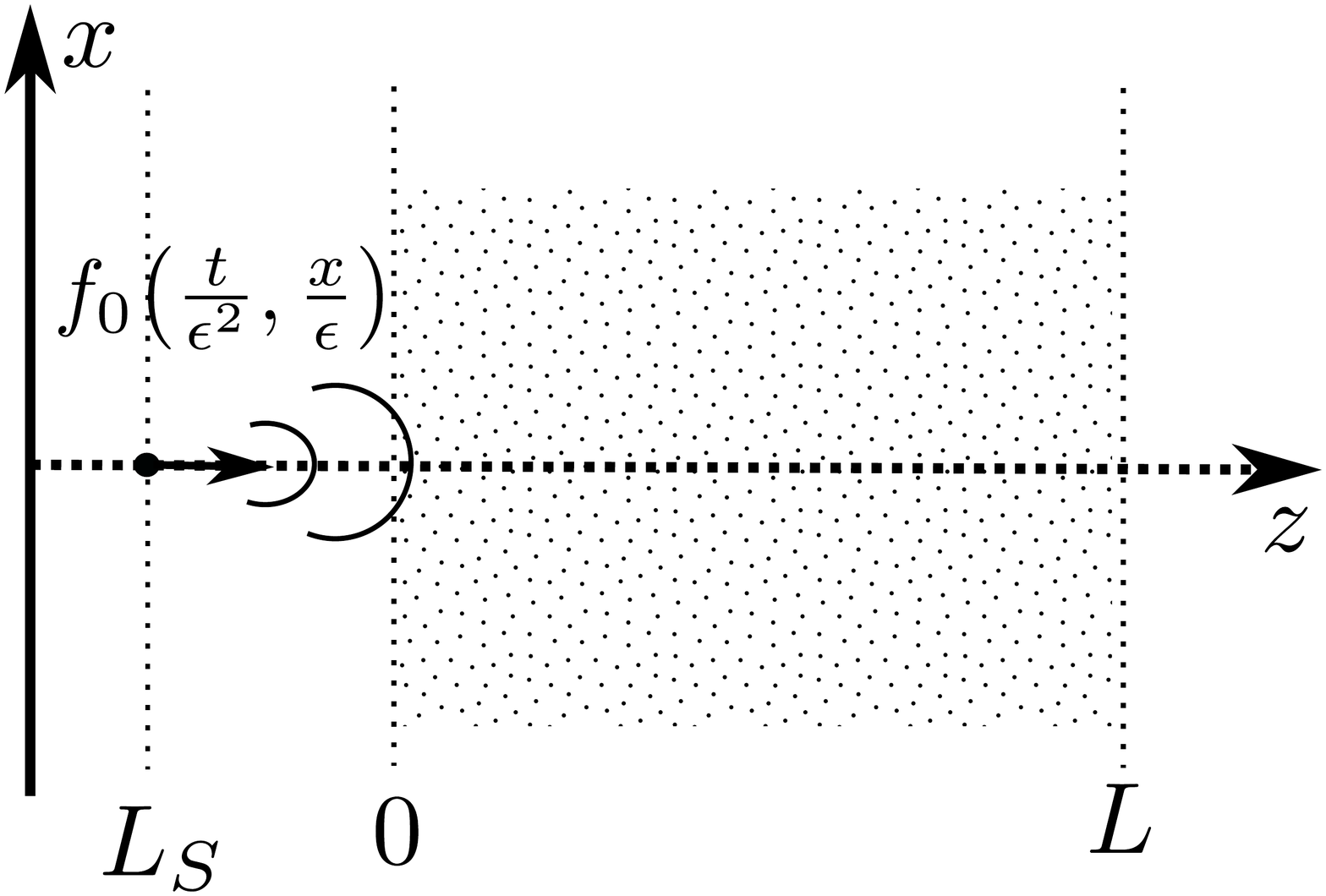}
\end{center}
\caption{\label{dessin1} Illustration of the wave propagation model.}
\end{figure}
\[\begin{split}
\frac{1}{c^2(z,x)} & :=  \left\{ \begin{array}{ccl} 
                                            \frac{1}{c_0^2}\big[1+ \sigma V\big(\frac{z}{l_c},\frac{x}{l_c}\big)\big]& \text{ if }  & z\in[0, L_z] \\
                                             \frac{1}{c_0^2} & \text{ if }  &z\in (-\infty,0)\cup(L_z,+\infty)\\
                                          \end{array} \right.  \quad\text{and}\qquad x\in\Rd,
\end{split}\]
where $c_0$ is the background velocity (constant for simplicity), and the random field $V(z,x)$, with a stationary covariance, models fluctuations around $c_0$ in the slab $(0,L_z) \times \Rd$. The parameters $\sigma$ and $l_c$ represent the amplitude and the correlation length of the fluctuations. The main assumption on $V$ is that it satisfies the long-range property in the $z$-direction, which is translated mathematically into a bounded non-integrable autocorrelation function which decreases at infinity only as  
\[\E[V(z+s,x)V(s,y)] \underset{z\to+ \infty}{\sim} \frac{c_{\mathfrak{H}}}{z^{\mathfrak{H}}}\qquad\text{with} \qquad \mathfrak{H} \in (0,1),\]
and then
\be  \label{longRV}  \int_0^{+\infty}\big\vert \E[V(z+s,x)V(s,y)] \big\vert dz=+\infty.\ee
We introduce now the scalings, which are similar to these of \cite{garnier0} where fluctuations with rapidly decaying correlations are considered. We assume first that the correlation length $l_c$ is small compare to the overall distance of propagation in the random medium  $L_z$, and we denote their ratio by
\[\e:=\frac{l_c}{L_z}\ll1.\]
Second, we assume that the transverse width $r_0$ of the source and the correlation length $l_c$ are of the same order,
\[r_0\sim l_c.\]
This assumption allows for a full interaction of the wave with the transverse fluctuations of the medium, leading to a non-trivial transverse behavior. Third, we assume that the central wavelength $\lambda_0$ is small compare to $L_z$ by taking
\[\e^2\sim\frac{\lambda_0}{L_z}.\]
This corresponds to a high frequency regime. With these choices, the Rayleigh length of the beam is of order of  the propagation distance $L_z$. The Rayleigh length is defined as the distance from the beam waist to the place where its cross-section is doubled by diffraction. Hence, the beam is still collimated at the exit of the random slab, which is a crucial assumption for the validity of the paraxial approximation. In homogeneous media, the Rayleigh length is of order $r^2_0/\lambda_0$. Therefore, we have for our problem
\[\frac{\lambda_0}{L_z}\sim \frac{\lambda_0}{r_0^2} \frac{r^2_0}{L_z}\sim \frac{r_0^2}{L^2_z}\sim\e^2.\]
This is a parabolic scaling, where the wave oscillations in the $z$ direction are much faster than in the transverse direction, which then leads to the paraxial wave equation. From now on, we consider the propagation distance $L_z$ as our reference scale of order 1, and rescale parameters as
\[ L_z=L,\qquad\lambda_0=\e^2 \qquad l_c=\e,\qquad\text{and}\qquad r_0=\e.\] 
Finally, we consider
\[\sigma=\e^{s}\qquad\text{with}\qquad s=2-\mathfrak{H}/2, \quad \mathfrak{H} \in (0,1),\]
where $\mathfrak{H}$ is related to the decay of the correlation function of $V$ in the variable $z$ as defined before. This specific choice of $s$ leads to a nontrivial asymptotic regime in the limit $\e$ goes to $0$. As a result, the wave equation \eqref{waveeq} becomes
\begin{equation}\label{waveeq2}
\Delta P-\frac{1}{c^2_0}\Big(1+\e^s V\Big(\frac{z}{\e},\frac{x}{\e}\Big)\mathbf{1}_{(0,L)}(z)\Big)\partial^2_t P =f_0\Big(\frac{t}{\e^2},\frac{x}{\e}\Big)\delta'(z-L_S).
\end{equation}
The main result of the paper is the asymptotic description of the pulse front exiting from the random section at $z=L$ and around the expected arrival time, which is defined by
\begin{equation}\label{pulse}
 P^\e_L(t,x)=P\Big(\lambda_0 t+\frac{L_z-L_S}{c_0},L_z,r_0 x \Big)= P\Big( \e^2 t+ \frac{L-L_S}{c_0}, L , \e x \Big).
\end{equation}
 Here, the solution is rescaled around the arrival time, and at the transverse scale of the source profile. We will show in this work that the process $P^\e_L$ converges in law in $\calC^0((0,+\infty),L^2(\Rd))$ to a process 
\[p^0_L(t,x):= \int    e^{-i\omega t} \Psi_\omega(L,x) d\omega,\]
where $\Psi_\omega$ satisfies the following fractional It{\^o}-Schrödinger equation
\begin{equation}\label{ItoSch0}
d\Psi_\omega(z,x)=\frac{i}{2k_\omega}\Delta_x \Psi_\omega+ik_\omega \Psi_\omega(z,x) dW_H (z,x)=0,
\end{equation}
where $k_\omega:=\omega/c_0$ is the wavenumber,
 \[\Psi_\omega(0,x):=\frac{1}{2}e^{-i\Delta_x L_S/(2k_\omega)}\check{f}_0(\omega,x),\] 
 with $e^{-iL_S \Delta_x /(2k_\omega)}$ the semigroup of the free Schr\"odinger equation, and with the convention 
\be \label{TFT} \check{f}(\omega)=  \frac{1}{2\pi} \int f(t)e^{i\omega t} dt\qquad\text{and}\qquad f(t)=\int \check{f}(\omega)e^{-i\omega t} d\omega.\ee 
Above, $W_H$ is a fractional field in $z$ with Hurst index 
\[H:=1-\mathfrak{H}/2\in (1/2,1)\]
and will be defined properly further, along with the nature of the stochastic integral. This latter integral is of pathwise type, and can be seen as a fractional equivalent to the It\^o-Stratonovich integral for standard Brownian motions. The function $\Psi_\omega$ describes the pulse deformation, in the paraxial approximation, due to the interaction of the wave with the random medium in the section $(0,L)$. The initial condition $\Psi_\omega(0)$ is simply the free propagation of the source from $z=L_S$ to $z=0$ in the paraxial approximation. In \fref{ItoSch0}, backscattering is neglected, leading to an initial value problem. As was already observed in different contexts in \cite{bal2,marty} for instance, the long-range nature thus leads to a different statistical description of the wave than in the classical mixing case of \cite{garnier0}. In the latter reference, waves are in the regime of the central limit theorem, and the resulting Schr\"odinger equation is driven by a standard Brownian field. Here, we are in a different regime where $\eps^{-\mathfrak{H}/2} \int_0^zV(u/\eps,x)du$ converges in law to a fractional field in $z$ with Hurst index $H=1-\mathfrak{H}/2 \in (1/2,1)$. A important difficulty in this work is then to justify that a similar type of limit holds for solutions to \fref{waveeq2}. There are in addition two other main technical points: showing that the coupling with the evanescent modes is negligible; these modes exist because of the non-trivial transverse frequency content of the random medium; and showing that backscattering can be ignored.

The paper is organized as follows. In Section \ref{main}, we introduce the various assumptions, define the stochastic integral and the notion of solution for \eqref{ItoSch0}, and state our main results. In Section \ref{outline}, we give an outline of the proof. The proof is then broken down into the subsequent sections. Section \ref{exist} concerns the derivation of some important estimates. Section \ref{RP} is devoted to central technical results about expectation and limits of iterated integrals. Section \ref{secevaprop} addresses the evanescent modes and Section \ref{proofnoback} the backscattering. Section \ref{proofItoSc} is devoted to the convergence to the fractional It\^o-Schrodinger equation. Section \ref{proofpulse} finalizes the proof of the main theorems, and Section \ref{proofL2p} addresses an estimate introduced further.

\textbf{Acknowledgment.} O. Pinaud acknowledges support from NSF CAREER grant DMS-1452349.

\section{Preliminaries and main results} \label{main}
Throughout this work, we will use the following conventions for the Fourier transform: $\check{f}$ denotes the Fourier transform w.r.t. the variable $t$ as in \fref{TFT}, and $\hat{f}$ that w.r.t. $t$ and $x$,
\[
\hat{f}(\omega,\kappa)=\frac{1}{(2\pi)^3} \int f(t,x)e^{i(\omega t+\kappa \cdot x)} dtdx\qquad\text{with}\qquad f(t,x)=\int  \hat{f} (\omega,\kappa)e^{-i(\omega t+\kappa \cdot x)}d\omega d\kappa.
\]

\subsection{Assumptions}

\paragraph{The source term.} We suppose that $\hat{f}_0(\omega,\kappa)$ is a bounded function with compact support in both variables, and even in the variable $\omega$. We assume moreover that it is supported away from zero w.r.t. $\omega$, that is there exists $\omega_c>0$ such that
\be \label{hypsup}
(-\omega_c,\omega_c) \cap supp_\omega \hat f_0(\omega,\kappa) = \emptyset, \qquad \forall \kappa \in \Rm^2.
\ee
The latter assumption essentially means that the source is shortband. Larger bandwidths could be included by direct modifications of the proofs.

\paragraph{The random field.} 
We construct the random field on a probability space $(\Omega,\mathcal{T},\Pro)$ and in the Fourier space as follows: the field $V$ is the Fourier transform of a random measure $\tilde V(z,dq)$, i.e.
\[
V(z,x)=\int_{\Rm^2} e^{- i q\cdot x} \tilde V(z,dq).
\]
We define $\tilde V(z,dq)$ sufficiently explicitly in order to be able to carry on the calculations. Let then $\calS \subset \Rm^2$ be a bounded domain, symmetric around the origin ($\calS=-\calS$), included in a ball $B(0,r_\calS)$. The domain $\calS$ will be the support of $\tilde V$ since we want the largest transverse frequency to be at most of order $\eps^{-1}$ (after rescaling). Let also $\BB_\mathfrak{H}$ be a real-valued  mean-zero Gaussian random field on $[0,+\infty)\times \calS$, continuous and stationary with respect to the variable $z$, and such that $\BB_\mathfrak{H}(z,q)=\BB_\mathfrak{H}(z,-q)$. Its covariance function is given by
\[ 
\E[\BB_\mathfrak{H}(z+z_0,q_1)\BB_\mathfrak{H}(z_0,q_2)]:=r_\mathfrak{H}(z)\hat{R}(q_1,q_2),
\] 
where $\hat{R}$ is assumed to be a continuous positive symmetric and bounded function such that 
 \[0< \hat{R}(q_1,q_2) \leq \hat{R}(q,q)=1,\qquad \forall(q,q_1,q_2)\in\calS\times\calS\times\calS.
\] 
 Besides, $r_\mathfrak{H}$ is a continuous even function bounded by $r_\mathfrak{H}(0)=1$ and 
\begin{equation}\label{LRr}
r_\mathfrak{H}(z)\underset{z\to +\infty}{\sim} \frac{c_\mathfrak{H}}{z^{\mathfrak{H}}}\qquad\text{with}\qquad \mathfrak{H}\in(0,1).
\end{equation}
Hence, $r_\mathfrak{H}$ is not integrable at the infinity. Let then $\Theta$ be a smooth odd function satisfying for all $l\in \mathbb{N}$,
\be\label{hyptheta}\sup_{u \in \Rm} \vert \Theta^{(l)}(u)\vert \leq C_{\Theta}^l,\ee
where $ \Theta^{(l)}$ stands for the $l$-th derivative of $\Theta$, and consider $\Theta(\mathcal{B}_\mathfrak{H}(z,q))$, which is not a Gaussian variable. Introducing a random measure $m(dq)$, supported on $\calS$, independent of the random field $\BB_\mathfrak{H}$, and whose properties are defined below, we write $\tilde V$ as $\tilde V(z,dq):=m(dq) \Theta(\mathcal{B}_\mathfrak{H}(z,q))$, so that
\begin{equation}\label{Vtheta}
V(z,x)=\int_{\calS}m(dq)e^{-i q \cdot x} \Theta\big(\BB_ \mathfrak{H}(z,q)\big).
\end{equation}
We suppose that $m^\ast(dq)=m(-dq)$, with bounded associated total variation measure $\vert m\vert$, that is, almost surely,
\begin{equation}\label{hypmomentm}
\vert m\vert (\calS) \leq C_m,
\end{equation}
for some deterministic constant $C_m>0$. This yields in particular that $V$ is real and bounded, and therefore that the velocity field cannot take negative values for $\eps$ sufficiently small. We suppose moreover that $m$ is stationary,
\[
\E[m(\varphi_1)m(\varphi_2)]=\int_\calS \mathfrak{m}(dq) \varphi_1(q)\varphi_2(q),
\]
where $\mathfrak{m}$ is a positive measure on $\calS$ with finite mass and the $\varphi_i$ are smooth functions. This yields that $V$ has a stationary covariance in both $z$ and $x$, which is a common assumption in applications. An example of such a measure is the following:
\be \label{expm}m(dq)=\sum_{j\geq 0} a_j (U_j \delta_{q_j}+\overline{U_j}\delta_{-q_j}),\ee
where $(a_j)_{j\geq0}\in l^1(\mathbb{N},\Rm)$ is deterministic, $(U_j)_{j\geq0 }\in\mathbb{C}^{\mathbb{N}}$ and $(q_j)_{j\geq 0}\in\calS^{\mathbb{N}}$ are independent iid sequences of random variables with appropriate distributions, and the $U_j$ have a zero mean.   
 
Note that $\E[V(z,x)]=0$ by symmetry, and in the same spirit as \cite[Lemma 1]{marty}, we show in Proposition \ref{LRV} of Section \ref{RP} that $V$ itself satisfies the long-range property
 \[
 \E[V(z+z_0,x)V(z_0,y)]\underset{z\to +\infty}{\sim} \frac{C_\mathfrak{H}}{z^{\mathfrak{H}}}  R_0(x-y)\quad\text{with}\quad C_\mathfrak{H}:=\frac{c_\mathfrak{H}}{2\pi}\Big(\int_{-\infty}^{+\infty}  \Theta(u) u e^{-u^2/2}du\Big)^2,
 \]
 and
 \begin{equation}\label{defR}
 R_0(x):= \int_{\calS}  \mathfrak{m}(dq) \hat R (q,q)e^{-iq\cdot x}= \int_{\calS}  \mathfrak{m}(dq) e^{-iq\cdot x}.
\end{equation}
This implies that \fref{longRV} is satisfied. Examples of realizations of $V$ are given in Figure \ref{plotV}.
\begin{figure}
\includegraphics*[scale=0.25]{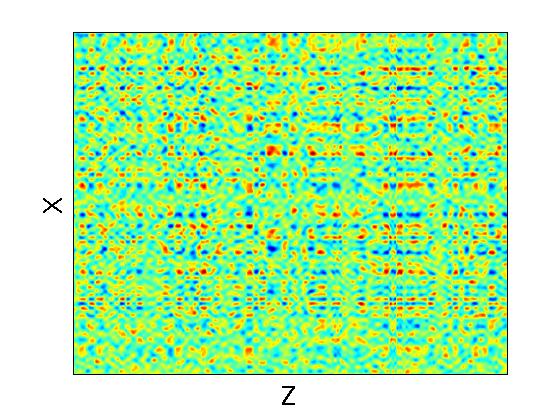}
\includegraphics*[scale=0.25]{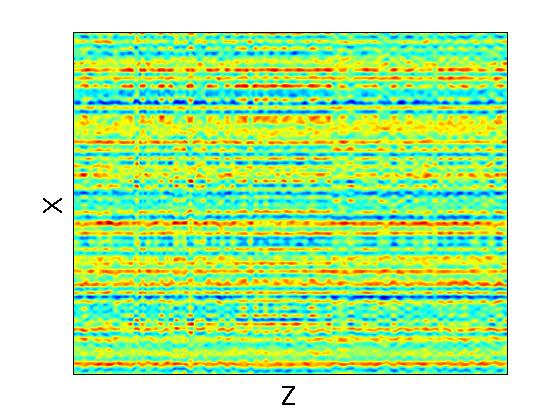}
\includegraphics*[scale=0.25]{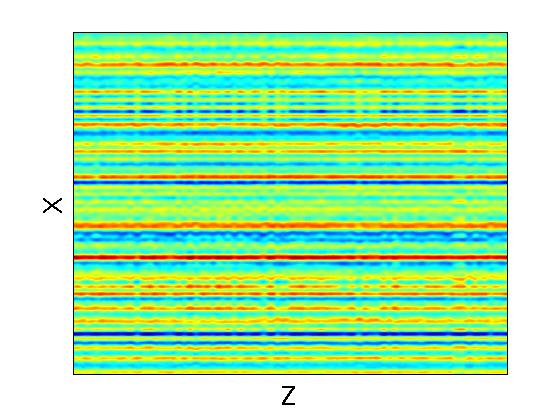}
\caption{\label{plotV} Examples of realizations of $V$. Here, $V$ is obtained via \fref{Vtheta} with $\Theta(x)= \sin(10x)$ and $\calS=[-15,15]$. The measure $m$ is as in \fref{expm}, where the $U_j$ are uniform in $[-1,1]$, the $q_j$ are chosen with a discrete uniform distribution among the points of a uniform discretization of $\calS$, and $a_j=1$ for $j \leq 100$ and zero otherwise. The field $\mathcal{B}_{\mathfrak{H}}$ is obtained via a similar formula as \fref{FBMB} where the $e_n$ are cosines, the $\beta_n$ behave like $n^{-2}$ and $W_{H,n}$ is replaced by a Gaussian process with autocorrelation $r_\mathfrak{H}(z)$ defined  as the Fourier transform of $\mathbf{1}_{(-15,15)}(k)/|k|^{1-\mathfrak{H}}$. From left to right, $\mathfrak{H}=0.9,0.5,0.1$. Observe the arising of long-range correlations in the $z$ direction as $\mathfrak{H}$ decreases, that is as $r_\mathfrak{H}$ decreases slower at the infinity.}
\end{figure}

The limiting field $W_H$ of \fref{ItoSch0} is heuristically obtained as follows: the scalings in $V$ and the long-range behavior act in a such a way that only the linear part in $\Theta$ is not negligible, and such that $\e^{-\mathfrak{H}/2}\int_0^z \mathcal{B}_\mathfrak{H}(u/\e,q)du$ is well approximated (in distribution) by a fractional Brownian field in $z$. Hence, $W_H$ is a random field with covariance operator given by 
\begin{equation}\label{correlation}
\E[ W_H(z_1,x) W_H(z_2,y)]= \frac{C_\mathfrak{H}}{2H(2H-1)}\big( z^{2H}_1+ z_2^{2H}-\vert z_1-z_2\vert^{2H}\big) R_0(x-y)\end{equation}
for all $(z_1,z_2)\in[0,+\infty)\times[0,+\infty)$ and $(x,y)\in \Rd\times \Rd$. The construction of $W_H$ and the definition of the stochastic integral are given in the next section.

\subsection{Stochastic integral and fractional It{\^o}-Schrödinger equation}\label{IntSto}

The stochastic integral with respect to a fractional Brownian motion obtained  here in the limiting process is of pathwise type, and is defined according to the work of Z\"{a}hle \cite{zahle}.  We start this section with the construction of the fractional field $W_H$ with covariance operator \eqref{correlation}, which is used thereafter to define the stochastic integral. Finally, we give the definition of a solution of \eqref{ItoSch0} before stating the main results of the paper.

\paragraph{Fractional field.}

A one dimensional standard fractional Brownian motion with Hurst index $H$, on a probability space $(\tilde{\Omega},\tilde{\mathcal{T}},\tilde{\Pro})$ is a centered Gaussian process $b_H$ with covariance
\[\E[b_H(u)b_H(v)]=\frac{1}{2}(u^{2H}+v^{2H}-\vert u-v \vert^{2H}),\qquad \forall (u,v)\in[0,+\infty)\times[0,+\infty).\] 
Moreover, such a process admits the following spectral representation
\begin{equation}\label{spectralB}
b_H(u)=C^{1/2}_{H}\int \frac{e^{iru}-1}{ir\vert r\vert^{H-1/2}}w(dr),
\end{equation}
with $C_H=H\Gamma(2H)\sin(\pi H)/\pi$, and where $w(dr)$ is a complex Gaussian random measure such that $w^\ast(dr)=w(-dr)$ and
\[\E[w(dr)w^\ast(ds)]=\delta(r-s)drds.\]
The construction of the fractional field with covariance operator \eqref{correlation} is done in the Fourier domain. Let $(W_{H,n})_{n\geq 1}$ be a sequence of independent standard fractional Brownian motions on the probability space $(\tilde{\Omega},\tilde{\mathcal{T}},\tilde{\Pro})$ with Hurst index $H$. Using the fact that 
\[\begin{array}{rcl} 
Q : L^2(\calS)&\longrightarrow&L^2(\calS)\\
\varphi &\longmapsto & \int_\calS dp \hat{R}(p,q)\varphi(q)
\end{array} \]
 is a positive self-adjoint trace class operator \cite[Corollary 4.4]{brislawn}, their exist a sequence $(e_n)_{n\geq1}$ of orthonormal eigenvectors and  a sequence $(\beta_n)_{n\geq 1}$ of positive eigenvalues for $Q$. Therefore, the Gaussian random field  
 \begin{equation}\label{FBMB}
 B_H(z,q):=\sum_{n\geq 1} \sqrt{\beta_n } e_n(q) W_{H,n}(z)
 \end{equation}
 defines a infinite-dimensional standard fractional Brownian motion on $L^2(\calS)$, and then
\begin{equation}\label{FBM}
W_H(z,x):=\sigma_H\int_\calS m(dq) e^{-iq\cdot x}  B_H(z,q)\qquad\text{with}\qquad \sigma^2_H=\frac{C_\mathfrak{H}}{H(2H-1)}
\end{equation} 
defines a random mixture of fractional Brownian fields with Hurst index $H$ and covariance operator given by \eqref{correlation}. It is interesting to note that our asymptotic noise model is not Gaussian. In fact, for medium perturbations with mixing properties the asymptotic noise is always Gaussian \cite{bailly,book,garnier0}. However, for medium perturbations with slowly decaying correlations, it is not necessarily the case. In a one-dimensional wave propagation setting, it has been observed in \cite{marty2} that the asymptotic noise model is not necessarily Gaussian if the initial medium fluctuations have non-Gaussian statistics. In our context, the non-Gaussianity of $W_H$ is due to the random measure $m$, which has been introduced for the covariance stationarity of $V$. Nevertheless, the form of $W_H$ allows the use of the Gaussian properties, and then the use of \cite{zahle} to define the stochastic integral in \eqref{ItoSch0}.

\paragraph{Stochastic integral.}\label{intsto}

 We follow here the approach of \cite{maslowski,nualart}. Let us consider the Banach space
 \[W^{\alpha}(0,L,B):=\Big\{\psi\in\mathcal{C}^0([0,L],B) \quad\text{such that}\quad  \| \psi \|_{\alpha,B}<+\infty \Big\},\]
with
\[  \| \phi \|_{\alpha,B}:= \sup_{z\in[0,L]} \Big[ \|\phi(z) \|_{B} +\int_0^z \frac{\| \phi(z)-\phi(u)  \|_{B}}{(z-u)^{\alpha+1}}du\Big],\]
and where $B$ is a given Banach space. Moreover, for $\beta\in(0,1)$, let us denote by $\mathcal{C}^{\beta}([0,L],B)$ the Banach space of $\beta$-H\"older functions on $[0,L]$ with values in $B$,
equipped with
\[\|\phi\|_{\beta,\mathcal{C},B} := \sup_{z\in[0,L]} \|\phi(z) \|_{B} + \sup_{0\leq v<u\leq L}\frac{\| \phi(u)-\phi(v)  \|_{B}}{(u-v)^{\beta}}.\]
One can remark that for $\alpha\in (0,1/2)$ 
\begin{equation}\label{inclusion}\|\phi\|_{\alpha,B} \leq L^{1-2\alpha} \|\phi\|_{1-\alpha,\mathcal{C},B} \quad\text{so that}\quad\mathcal{C}^{1-\alpha}([0,L],B)\subset W^{\alpha}(0,L,B). \end{equation}
Now, for $\alpha\in(0,1)$ and $z\in(0,L)$, let us introduce, for a real-valued function $f$, the so-called Weyl's derivative given by
\[\begin{split}
D^\alpha_{0^+}f(z)&:=\frac{1}{\Gamma(1-\alpha)}\Big[\frac{f(z)}{z^\alpha}+\alpha\int_0^z\frac{f(z)-f(u)}{(z-u)^{\alpha+1}}du\Big],\\
D^\alpha_{L^-}f(z)&:=\frac{(-1)^\alpha}{\Gamma(1-\alpha)}\Big[\frac{f(z)}{(L-z)^\alpha}+\alpha\int_z^L\frac{f(z)-f(u)}{(u-z)^{\alpha+1}}du\Big],
\end{split}\]  
whenever these quantities are well-defined, and where $\Gamma(u)=\int_0^{+\infty} r^{u-1}e^{-r}dr$ is the gamma function. Following \cite{zahle}, the generalized Stieljes integral of a function $f\in\mathcal{C}^{\nu}([0,L],\mathbb{R})$ with respect to $g\in\mathcal{C}^{\mu}([0,L],\mathbb{R})$, with $\nu+\mu>1$, $\nu>\alpha$, and $\mu>1-\alpha$ is defined by
\begin{equation}\label{defintsto}
\int_0^L f dg:=(-1)^{\alpha} \int_0^L D^\alpha_{0^+}f(u) D^{1-\alpha}_{L^-}g_{L^-}(u)du,
\end{equation}
where $g_{L^-}(u):=g(u)-g(L^-)$. The definition does not depend on $\alpha$, and we have
\[\int_0^z fdg=\int_0^L f\textbf{1}_{(0,z)}dg.\]
Moreover, according to \cite{nualart}, this integral can be extended to more general classes of functions thanks to the relation
\[\Big\vert  \int_0^L fdg \Big\vert\leq \|f\|_{\alpha,1}\Lambda_\alpha(g),\]
where
\[ \|f\|_{\alpha,1}:=\int_0^L\Big( \frac{\vert f(u)\vert}{u^\alpha}+\int_0^u \frac{\vert f(u)-f(v)\vert}{(u-v)^{\alpha+1}}dv \Big)du,\]
and
\[\Lambda_\alpha(g):=\frac{1}{\Gamma(1-\alpha)\Gamma(\alpha)}\sup_{0<u<z<L}\vert D^{1-\alpha}_{t-}g_{t^-}(u)\vert.\]
Consequently, this integral is well-defined as soon as $f\in W^{\alpha}(0,L,\mathbb{R})$ and $\Lambda_\alpha(g)<+\infty$. 

As a result, for a random function $F\in W^{\alpha}(0,L,L^2(\Rd))$, the stochastic integral with respect to the fractional field $W_H$,
 \[ \int_0^z F(u,x) dW_H(u,x),\]
  is defined by \eqref{defintsto} almost everywhere in $x$ and $\Pro$-almost surely. In fact, we have for $\alpha\in(1-H,1/2)$ and for all $z\in[0,L]$,
 \[
  \Big\| \int_0^z F(u) dW_H(u)\Big\|_{L^2(\Rd)}\leq C\|F\|_{W^{\alpha}(0,L,L^2(\Rd))} \sup_{x\in \Rd} \Lambda_{\alpha}(W_H(x)),
 \]
with
\[
\E\Big[\sup_{x\in \Rd} \Lambda_{\alpha}(W_H(x))\Big] \leq \int_\calS  \E[\vert m\vert(dq)] \E[\Lambda_{\alpha}(B_H(q))]\leq C_m  \sup_{q\in\calS}\E[\Lambda_{\alpha}(B_H(q))]<\infty, \]
as will be proved later in Lemma \ref{techBA}.

\paragraph{Fractional It\^o-Schrödinger equation.}

The notion of solution for the It\^o-Schrödinger equation \eqref{ItoSch0} is made precise in the following definition. First, let us introduce some additional notations. Let $k\in \mathbb{N}^\ast$, and let us denote by $H^k(\Rd)$ the $k$-th Sobolev space on $\Rd$. Consider moreover $W^{\alpha}_k(0,L):=W^{\alpha}(0,L,H^k(\Rd))$, equipped with the norm $ \| \cdot \|_{\alpha,H^k(\Rd)}$, and the complete metric space
\[  W^{\alpha}_\infty(0,L):=\bigcap_{k\in \mathbb{N}^\ast}W^{\alpha}_k(0,L),\]
equipped with
 \[ d_{\alpha,\infty} (\phi,\psi): =\sum_{k\geq 1} \frac{1}{2^k} \left(1\wedge \| \phi -\psi \|_{\alpha,H^k(\Rd)} \right).\] 

\begin{defi}
Let $H\in(1/2,1)$, $\alpha\in(1-H,1/2)$, and $W_H$ be the fractional field defined by \eqref{FBM}. We say that $\Psi_\omega \in W^{\alpha}_{\infty}(0,L)$ is a pathwise solution of \eqref{ItoSch0} if, with probability one, for all $(z,x)\in [0,L]\times \Rd$, we have
\[ \Psi_\omega (z,x) = \Psi_\omega (0,x)+ \frac{i}{2k_\omega}\int_0^z \Delta_x \Psi_\omega(u,x) du +i k_\omega \int_0^z \Psi_\omega(u,x) dW_H(u,x).\]
\end{defi}
In other words, a solution of \eqref{ItoSch0}  is a pointwise solution of this equation for almost all realizations of the randomness. We will see later that a solution to \eqref{ItoSch0} has automatically H\"older regularity
\[\Psi_\omega\in \mathcal{C}^{H-\theta}_{\infty}(0,L):=\bigcap_{k\in \mathbb{N}^\ast} \mathcal{C}^{H-\theta}([0,L],H^k(\Rd))\quad\text{with}\quad \theta=H+\alpha-1.\]
Here, $\mathcal{C}^{H-\theta}_{\infty}(0,L)$ is a complete metric space
equipped with
 \[ d_{H-\theta,\mathcal{C},\infty} (\phi,\psi): =\sum_{k\geq 1} \frac{1}{2^k} \left(1\wedge \| \phi -\psi \|_{H-\theta,\mathcal{C},H^k(\Rd)}\right),\]
so that $\mathcal{C}^{H-\theta}_{\infty}(0,L)\subset W^\alpha_\infty(0,L)$ according to \eqref{inclusion}. The solutions we define here are classical solutions in the standard terminology. It is not completely trivial to construct less regular solutions to \fref{ItoSch0}, which is required for non-linear problems, see \cite{pinaud}. Indeed, the standard technique is to use the mild formulation, and then treat a term of the form 
\[
\int_0^z S(z-u) \Psi_\omega(u,x) dW_H(u,x),
\]
 where $S$ is the Schr\"odinger semigroup. As explained before, some H\"older regularity in $u$ is needed in order to make sense of the integral. Since the semigroup is not sufficiently regularizing, this regularity in $u$ has to be exchanged for some regularity in $x$ on $\Psi_\omega$, and the fixed point procedure cannot be closed. This is not a problem in our linear setting where we can iterate the stochastic integrals and suppose that the initial condition is $\calC^\infty$ in $x$. A different strategy has to be adopted in the non-linear case \cite{pinaud}.

Note that the stochastic integral here is the fractional equivalent to the It{\^o}-Stratonovich integral for standard Brownian motions, and as such satisfies the classical integration by parts formula. This then formally yields the conservation relation, for all $z\in[0,L]$,
\[\|\Psi_\omega(z)\|_{L^2(\Rd)}=\|\Psi_\omega(0)\|_{L^2(\Rd)}=\frac{1}{2}\|\check{f}_0(\omega)\|_{L^2(\Rd)}.\]

\subsection{Main results}

We will actually not work directly with the process $P_L^\eps$ given by \eqref{pulse}, but rather with an approximate process $p_L^\eps$ defined by
\be \label{pulse22}
p_L^\eps(t,x):=p\Big( \e^2 t+ \frac{L-L_S}{c_0}, L , \e x \Big),
\ee
where $p$ solves the wave equation \fref{waveeq2} with artificial absorption $\alpha_\eps$,
\[\Delta p-\frac{1}{c^2_0}\Big(1+\e^s V\Big(\frac{z}{\e},\frac{x}{\e}\Big)\mathbf{1}_{(0,L)}(z)\Big)\partial^2_t p+ i \alpha_\eps p =f_0\Big(\frac{t}{\e^2},\frac{x}{\e}\Big)\delta'(z-L_S),\]
and vanishing initial conditions. Thanks to the estimate below, proved in Section \ref{proofL2p},
\be \label{L2p}
\sup_{(t,z) \in (0,T) \times \Rm}\| P(t,z,\eps \cdot)-p(t,z,\eps \cdot)\|_{L^2(\Rm^2)}\leq  \frac{C_T\alpha_\eps^{1/2}}{ \eps}\qquad\forall T>0,
\ee
it is equivalent, from the viewpoint of convergence in law (see \cite[Theorem 3.1 pp. 27]{billingsley}), to consider $p_L^\eps$ instead of $P_L^\eps$ by choosing $\alpha_\eps=o(\eps^2)$. The main theorem will be hence stated in terms of $p_L^\eps$. The introduction of $p$ is an important point since the absorption term provides us with straightforward estimates in $L^2((0,L) \times \Rm^2)$, that would require much more work with the process $P$. These estimates are not uniform, but sufficiently tamed, and are exploited throughout the paper.

We will mostly work in the frequency domain, and in order to take Fourier transforms in time, we extend $p$ to negative times by setting $p(-t,z,x)=p(t,z,x)$, for all $t>0$.

The main result of this paper is the following theorem, that states the convergence of the pulse \eqref{pulse22}. 
 
\begin{thm}[Convergence result]\label{mainth1}
The family $(p^\e_L )_{\e\in(0,1)}$, defined by \eqref{pulse22}, converges in law in the space $\mathcal{C}^0((-\infty,+\infty), L^2(\Rd))\cap L^2((-\infty,+\infty)\times \Rd)$ to a limit given by 
\[p^0_L(t,x)= \int  e^{-i\omega t}\Psi_\omega(L,x)d\omega,\]
where $\Psi_\omega$ is the unique pathwise solution to the fractional It\^o-Schr\"odinger equation  \eqref{ItoSch0}.
\end{thm}

The second theorem below is a by-product of the proof of the main theorem, and provides us with some interesting properties of the solutions of the fractional It\^o-Schr\"odinger equation: existence and uniqueness, conservation of the energy, approximation by a smooth process which can be expanded in terms of scattering events, and approximate formulas for moments of any order.

\begin{thm}[Properties of the fractional It\^o-Schr\"odinger equation]\label{mainth2}
We have the three following statements:
\begin{enumerate}
\item  The fractional It\^o-Schr\"odinger equation \eqref{ItoSch0} admits a unique pathwise solution $\Psi_\omega$ for all $\alpha\in(1-H,1/2)$, which satisfies
\begin{equation}\label{consITO}
\|\Psi_\omega(z)\|_{L^2(\Rd)}=\frac{1}{2}\|\check{f}_0(\omega)\|_{L^2(\Rd)}\qquad \forall z\in[0,L].
\end{equation}
Moreover,  $\Psi_\omega\in \mathcal{C}^{H-\theta}_{\infty}(0,L)$ for all $\theta\in(0,H-1/2)$.
\item For all $\theta\in(0,H-1/2)$, the process $\Psi_\omega$ can be approximated by 
\[
\Psi_\omega(z)=\lim_{A\to+\infty}\Psi^A_\omega(z)\qquad{with}\qquad \Psi^A_\omega(z)=\check{f}_0(\omega,\cdot)+ \sum_{n\geq 1} \Psi^{A,n}_{\omega}(z),
\]
where the limit holds in $C^{H-\theta}_\infty(0,L)$ in probability.
Here, we have in the Fourier domain
\[\begin{split}
\hat{\Psi}^{n,A}_{\omega}(z,\kappa):= (ik_\omega )^n \int_{\calS^n}\mathbf{m}(d\mathbf{q}^{(n)}) \int_{\Delta_n(z)} d\mathbf{u}^{(n)} & e^{-i\vert \kappa\vert^2z/(2k_\omega)}e^{iG_n(\mathbf{u}^{(n)},\mathbf{Q}^{(n)})}\hat{f}_0(\omega,Q_n)\\
&\times \prod_{m=1}^n \int_{-A}^A \frac{e^{ir_m u_m}}{\vert r_m \vert^{H-1/2}}w(dr_m,q_m),
\end{split}\]
where $\mathbf{u}^{(n)}:=(u_1,\dots,u_n)$, $\mathbf{q}^{(n)}:=(q_1,\dots,q_n)$, $\mathbf{m}(d\mathbf{q}^{(n)})=m(dq_1)\dots m(dq_n)$, $\mathbf{Q}^{(n)}:=(Q_0,\dots,Q_n)$, with $Q_m:=\kappa-q_1-\cdots-q_m$, and
\[
 G_n(\mathbf{u}^{(n)},\mathbf{Q}^{(n)}):=\frac{1}{2k_\omega}\sum_{m=1}^n\big(\vert Q_{m-1}\vert^2- \vert Q_m\vert^2 \big)u_m.
 \]
Moreover, 
\[\Delta_n(z):=\big\{(u_1,\dots,u_n)\in[0,z]^n,\quad\text{s.t.}\quad 0\leq u_j \leq u_{j-1}\quad \forall j\in\{2,\dots,n\}\big\},\]
and $(w(dr,q))_{q\in\calS}$ is a family of complex Gaussian random measure defined by
\begin{equation}\label{defw}
w(dr,q):=\sum_{n\geq 1} \sqrt{\frac{\beta_nC_{\mathfrak{H}}\Gamma(2H-1)\sin(\pi H) }{\pi}} e_n(q)w_n(dr),
\end{equation}
where $(w_n(dr))_{n\geq 1}$ is the family of independent complex Gaussian random measure in the spectral representation \eqref{spectralB} of the family $(W_{n,H})_{n\geq 1}$ introduced in \eqref{FBM}.
\item We have
\[\begin{split}
\E\Big[\Big< \prod_{j_1=1}^{M_1}&\Psi_{\omega_{1,j_1}}(z)\prod_{j_2=1}^{M_2}\overline{\Psi_{\omega_{2,j_2}}(z)}, \varphi\Big>_{L^2(\mathbb{R}^{2(M_1+M_2)})}\Big]\\
&=\lim_{A\to+\infty}\E\Big[\Big< \prod_{j_1=1}^{M_1}\Psi^A_{\omega_{1,j_1}}(z)\prod_{j_2=1}^{M_2}\overline{\Psi^A_{\omega_{2,j_2}}(z)}, \varphi\Big>_{L^2(\mathbb{R}^{2(M_1+M_2)})}\Big],
\end{split}\]
for all $(M_1,M_2)\in\mathbb{N}^2$, frequencies $(\omega_{i,j})_{(i,j)\in\{1,2\}\times\{1,\dots,M_i\}}$ and $\varphi\in L^2(\mathbb{R}^{2(M_1+M_2)})$.
\end{enumerate}
\end{thm}

The conservation relation \eqref{consITO} is a consequence of a negligible backscattering and shows that the energy of the pulse is conserved at the end of the random section, that is
\[\|p^0_L\|_{L^2((-\infty,+\infty)\times \Rd)}=\frac{1}{2}\|f_0\|_{L^2((-\infty,+\infty)\times \Rd)}.\]
Note also that the convergence in the second point holds in $W^\alpha_\infty(0,L)$ with $\alpha=1-H+\theta$ according to \eqref{inclusion}, and that $\Psi_\omega$ becomes smoother in $z$ as $H$ increases, which is expected since the regularity of the fractional brownian motion improves with $H$. Moreover, we will see further that the process $\Psi^A_\omega$ is the solution to a fractional It\^o-Schr\"odinger equation with a regularized fractional white noise, and as such enjoys some regularity properties (w.r.t. $z$) that are convenient in justifying formal computations, in particular the calculation of the moments as in item (3) above. Moments are important for instance in imaging applications, where they help quantify the stability of reconstructions with respect to changes in the random medium, see e.g. \cite{CINT,BP-CPDE,BP-M3AS,garnier2}. The series expansion in item (2) is the classical Born approximation, see e.g. \cite[Section 17.2]{ishimaru}.

Note finally that $\Psi_\omega$ satisfies various formulations of \eqref{ItoSch0}, for instance
  \[ \hat{\Psi}_\omega (z,\kappa) =\hat{\Psi}_\omega (0,\kappa) -\frac{i\vert \kappa\vert^2}{2k_\omega} \int_0^z  \hat{\Psi}_\omega(u,\kappa) du  +i k_\omega \sigma_H \int_\calS m(dq) \int_0^z \hat{\Psi}_\omega(u,\kappa -q) dB_H(u,q)\]
 in the Fourier domain, or the mild formulation
   \[ \hat{\Psi}_\omega (z,\kappa) =  e^{-i\vert \kappa \vert ^2z/(2k_\omega)}\hat{\Psi}_\omega (0,\kappa) + ik_\omega \sigma_H \int_\calS m(dq) \int_0^z e^{-i\vert \kappa \vert ^2(z-u)/(2k_\omega)} \hat{\Psi}_\omega(u,\kappa -q) dB_H(u,q)du,\]
 where the relation between $W_H$ and $B_H$ is given by \eqref{FBM}. 

The rest of the paper is dedicated to the proofs of the theorems. We will focus on Theorem \ref{mainth1} since its proof contains that of Theorem \ref{mainth2}. We begin with an outline in the next section.

\section{Outline of the proof} \label{outline}

The proof starts with recasting the wave equation in the Fourier domain in all variables, except $z$. Defining for this the Fourier transform in time 
\[\check{g}^\e(\omega)=  \frac{1}{2\pi \e^2} \int g(t)e^{i\omega t/\e^2} dt\qquad\text{with}\qquad g(t)=\int \check{g}^\e(\omega)e^{-i\omega t/\e^2} d\omega,\]
that accounts for the high frequencies generated by the source, we obtain from \eqref{waveeq2} the Helmholtz equation
\begin{align}\nonumber \label{helm1}
&\partial^2_z\check{p}^\e_\omega (z,x)+\Delta_x \check{p}^\e_\omega (z,x)+\frac{k^2_\omega}{\e^4}\Big( 1+\e^{s} V\Big(\frac{ z}{\e},\frac{x}{\e} \Big)\mathbf{1}_{(0,L/\eps)}(z)\Big)\check{p}^\e_\omega (z,x)+i \alpha_\eps \check{p}^\e_\omega (z,x)\\
& \hspace{2cm}=\check{f}_0\Big(\omega,\frac{x}{\e} \Big)\delta'(z-L_S).
\end{align}
We will construct solutions to \fref{helm1} in section \ref{exist}, and show that $\check{p}^\e_\omega$   satisfies the required regularity to justify all the calculations. Taking the Fourier transform w.r.t to $x$ of the wavefield rescaled around the propagation axis $\check{p}^\e_\omega (z,\e x)$, that is
\[\hat{p}^\e_\omega (z,\kappa):=\frac{1}{(2\pi)^2}\int \check{p}^\e_\omega (z,\e x) e^{i\kappa x} dx,\] 
we find
 \begin{equation}\label{helmholtzF}\begin{split}
\partial^2_z\hat{p}^\e_\omega (z,\kappa)&+\frac{k^2_\omega}{\e^4}\Big(1 - \frac{\e^2 \vert \kappa \vert ^2}{k^2_\omega}\Big)\hat{p}^\e_\omega (z,\kappa) 
+i \alpha_\eps \hat{p}^\e_\omega (z,\kappa)
\\
&+\e^{s-4} k_\omega^2\int_{\calS}m(dq)\hat{V}\Big(\frac{ z}{\e},q \Big)\hat{p}^\e_\omega(z,\kappa-q)=\hat{f}_0(\omega, \kappa)\delta'(z-L_S),
\end{split}\end{equation}
where $k_{\omega}=\omega/c_0$ is the wavenumber, and $\hat{V}=\Theta(\mathcal{B}_{\mathfrak{H}})$ according to \eqref{Vtheta}. Following the standard terminology in absence of absorption ($\alpha_\e=0$), when the wavevector $\kappa$ satisfies $\vert \kappa \vert < \vert k_\omega \vert/\e$ we will refer to the corresponding mode as a \emph{propagating mode}. These modes can propagate over large distances. When the wavevector satisfies $\vert \kappa \vert > \vert k_\omega \vert/\e$, we refer to an \emph{evanescent mode}. The proof is then based on a decomposition of $\hat{p}^\e_\omega$ into right- and left-going propagating modes, see Figure \ref{dessin2}, with amplitudes $\hat{a}^\e_\omega$ and $\hat{b}^\e_\omega$, respectively. These amplitudes are  defined as the solutions to \bea\label{lrmode}
 \hat{p}^\e_\omega(z,\kappa)&=&\frac{1}{\sqrt{\lambda_{\e,\omega}(\kappa)}}\Big( \hat{a}^\e_\omega(z,\kappa)e^{ik_{\omega}\lambda_{\e,\omega}(\kappa)z/\e^2}+ \hat{b}^\e_\omega(z,\kappa)e^{-ik_{\omega}\lambda_{\e,\omega}(\kappa)z/\e^2}\Big)\\
\label{lrmode2}\partial_z\hat{p}^\e_\omega(z,\kappa)&=&\frac{ik_{\omega}\sqrt{\lambda_{\e,\omega}(\kappa)}}{\e^2}\left( \hat{a}^\e_\omega(z,\kappa)e^{ik_{\omega}\lambda_{\e
,\omega}(\kappa)z/\e^2}- \hat{b}^\e_\omega(z,\kappa)e^{-ik_{\omega}\lambda_{\e,\omega}(\kappa)z/\e^2}\right),
 \eea
for $\vert \kappa \vert < \vert k_\omega \vert/\e$, and where
\be \label{deflam}\lambda_{\e,\omega}(\kappa):=\sqrt{1 -\e^2 \vert \kappa \vert ^2/k^2_\omega+i \alpha_{\omega,\e}},\qquad \textrm{with} \qquad \alpha_{\omega,\e} =\frac{\alpha_\e \e^4}{k_\omega^2}.\ee
\begin{figure}
\begin{center}
\includegraphics*[scale=0.3]{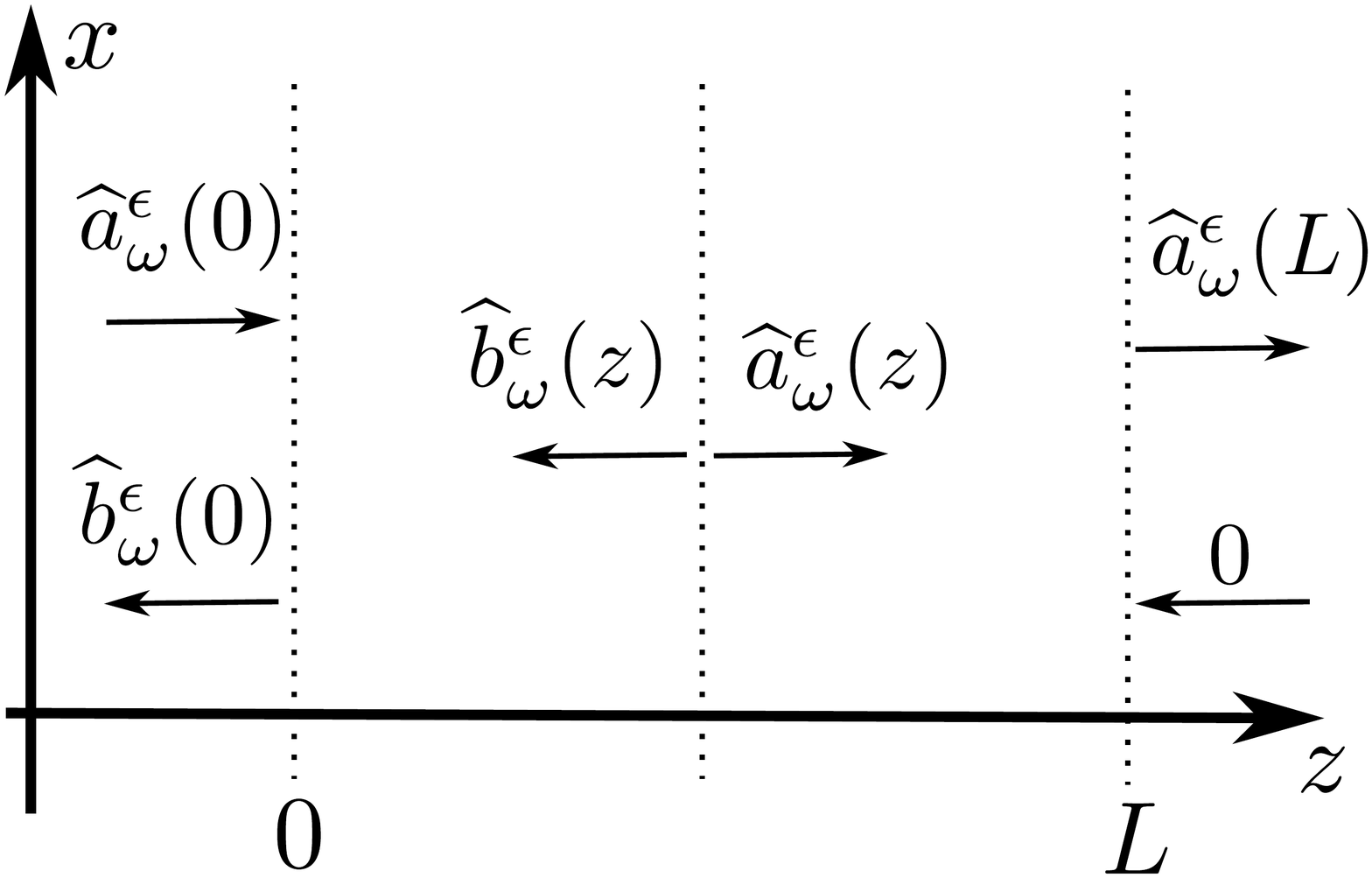}
\end{center}
\caption{\label{dessin2} Illustration of the right- and left going mode amplitudes.}
\end{figure}
We consider here the principal square root for complex numbers, namely the square root with positive imaginary part. Using \fref{helmholtzF} we find the following coupled-mode equations
\begin{equation}\label{cmode}\begin{split}
\partial_z\begin{bmatrix} \hat{a}^\e_\omega(z,\kappa)\\ \hat{b}^\e_\omega(z,\kappa) \end{bmatrix}&=\int_{\{\vert \kappa-q \vert <\vert k_\omega\vert/\e\}\cap\calS}m(dq) H^{\e}_\omega(z,\kappa,q)   \begin{bmatrix} \hat{a}^\e_\omega(z,\kappa-q)\\ \hat{b}^\e_\omega(z,\kappa-q) \end{bmatrix}\\
&+\e^{s-2}\frac{i k_\omega}{2\sqrt{\lambda_{\e,\omega}(\kappa)}}\int_{\{\vert \kappa-q \vert > \vert k_\omega \vert /\e\}\cap\calS}m(dq)\hat{V}( z/\e,q) \begin{bmatrix} e^{-ik_\omega \lambda_{\e,\omega}(\kappa)z/\e^2} \\- e^{ik_\omega \lambda_{\e,\omega}(\kappa)z/\e^2}  \end{bmatrix}\hat{p}_\omega(z,\kappa-q),
\end{split}\end{equation}
for $\vert \kappa \vert < \vert k_\omega \vert/\e$, with 
\begin{equation}\label{defH}\begin{split}
H^{\e}_\omega(z,\kappa,q)=\e^{s-2}&\frac{i k_\omega}{2\sqrt{\lambda_{\e,\omega}(\kappa)\lambda_{\e,\omega}(q)}}\hat{V}( z/\e,q)\\
&\times \begin{bmatrix} e^{ik_\omega( \lambda_{\e,\omega}(q)- \lambda_{\e,\omega}(\kappa))z/\e^2} & e^{-ik_\omega( \lambda_{\e,\omega}(q)+\lambda_{\e,\omega}(\kappa))z/\e^2} \\ -e^{ik_\omega( \lambda_{\e,\omega}(q)+\lambda_{\e,\omega}(\kappa))z/\e^2} & -e^{-ik_\omega( \lambda_{\e,\omega}(q)- \lambda_{\e,\omega}(\kappa))z/\e^2}
\end{bmatrix}.
\end{split}\end{equation}
The system \eqref{cmode} is equipped with the boundary conditions
\begin{equation}\label{defa0}
 \hat{a}^\e_\omega(0,\kappa)=\frac{\sqrt{\lambda_{\e,\omega}(\kappa)}}{2}e^{-ik_\omega\lambda_{\e,\omega} L_S/\e^2}   \hat{f}_0(\omega,\kappa)
 \qquad\text{and}\qquad  \hat{b}^\e_\omega(L,\kappa)=0,
 \end{equation}
where the first condition represents the (known) amplitude of the wave coming from the left-homogeneous half-space and entering the slab $(0,L)$, and the second condition implements the fact that no wave is entering the right-hand side of the slab. These conditions will be investigated in more details in Section \ref{exist}. Despite its formulation, the system \fref{cmode} is not an initial value problem (IVP), but rather a boundary value problem. The limiting problem will nevertheless be shown to be an IVP.

Note that in the homogeneous case where $V=0$, the justification of the paraxial approximation is straightforward. Indeed, in that case $\hat{a}^\e_\omega(z,\kappa)$ is constant for $z>L_S$ and therefore equal to $\hat{a}^\e_\omega(0,\kappa)$ defined in \fref{defa0}, so that the right-going wave \eqref{pulse} is given by 
 \[\begin{split}
p^\e_z(t,x)&=p\Big( \e^2 t+ \frac{z-L_S}{c_0},z, \e x \Big)\\
&=\frac{1}{2}\iint  e^{-i\omega t} e^{-i\kappa x} e^{ik_\omega(\lambda_{\e,\omega}(\kappa)-1)(z-L_S)/\e^2} \hat{f}_0(\omega,\kappa) d\omega d\kappa,
\end{split}\]
and therefore, pointwise in $(t,x)$,
\[\begin{split}
\lim_{\e\to 0}p^\e_z(t,x)&=\frac{1}{2}\iint e^{-i\omega t}e^{-i\kappa x}e^{-i\frac{\vert \kappa \vert ^2}{2k_{\omega}}(z-L_S)} \hat{f}_0(\omega,\kappa)d\omega d\kappa\\
&=\frac{1}{2}\int e^{-i\omega t}\Big[e^{\frac{i}{2k_\omega}(z-L_S)\Delta_x} \check{f}_0(\omega,\cdot)\Big](x)d\omega,
\end{split}\]
where $e^{iz\Delta_x /(2k_\omega)}$ is the semigroup of the free Schr\"odinger equation. 

When $V$ is random, the core of the  proof is an asymptotic analysis of the amplitudes $\hat{a}^\e_\omega$ and $\hat{b}^\e_\omega$ solutions to \fref{cmode}. There are several steps, most of which involve computing moments of the form
\be \label{defEB} \E\Big[ \prod_{j=1}^n \Theta\Big(\BB_\mathfrak{H}\big(\frac{u_j}{\e},p_j\big) \Big) \Big].\ee
The random variables $\Theta(\BB_\mathfrak{H}(u/\e,q))$ are not Gaussian, but we will see in Section \ref{RP} that they asymptotically behave as Gaussian variables. This is based on the ideas of \cite{taqqu}: suppose that $n$ is even; the even function $\Theta$ is then projected on the basis of the Hermite polynomials, which gives rise to a series of terms in \fref{defEB}. The leading one corresponds to the product of the first order terms in the Hermite expansions, and consists of the product of $n/2$ moments of order 2 of $\BB_\mathfrak{H}$. Any other term involves at least $n/2+1$ moments of order 2 which, after integration, the use of the scaling $u_j/\eps$ and of the long-range dependance property, leads to negligible contributions. The fact that a fractional Brownian field is obtained at the limit is a direct consequence of the asymptotic behavior of the correlation function \fref{LRr}.

Owing the technical results on terms of the form \fref{defEB}, we can then proceed to the analysis of $\hat{a}^\e_\omega$ and $\hat{b}^\e_\omega$. The first step is to show that the coupling between propagating and evanescent modes can be neglected. This is actually a fairly subtle point. The fact that evanescent modes are expected to decrease exponentially as $z$ increases in $(0,L)$ cannot be exploited close to the transition propagating/evanescent modes. Indeed, in this case, $\lambda_{\e,\omega}$ is too small around the transition, which essentially yields exponentials in \fref{lrmode} of order one. Our method then goes as follows: we start from \fref{cmode} with $\vert\kappa\vert<\vert k_\eps\vert /\e$, and therefore only consider the propagating modes; the first term in the r.h.s corresponds to propagating modes with momentum $\kappa-p$ scattered to propagating modes with momentum $\kappa$ after interaction with the random medium; the second term corresponds to evanescent modes with momentum $\kappa-p$ scattered  to propagating modes with momentum $\kappa$. We will prove the convergence of $(\hat{a}^\e_\omega, \hat{b}^\e_\omega)$ in the distribution sense, which will limit $|\kappa|$ to some bounded domain independent of $\eps$, say $\vert \kappa\vert<R$. In this case, the second term  is zero, since only propagating modes with large wavenumbers of order $\vert k_\omega\vert /\eps$ (which is of order $\eps^{-1}$ since $k_\omega$ is bounded from below independently of $\eps$ according to assumption \fref{hypsup}) are coupled to the evanescent modes. This naturally does not mean that evanescent modes have no influence on the propagating modes, the coupling appears in the first term of the r.h.s via modes with larger and larger wavenumbers that get closer to  the transition. 

As an approximation, it is therefore natural to introduce the following system, that only describes propagating modes, for all $\vert\kappa\vert<\vert k_\omega\vert/\e$,
\be \label{AB}\begin{split}
\begin{bmatrix} \hat{A}^\e_\omega(z,\kappa)\\ \hat{B}^\e_\omega(z,\kappa) \end{bmatrix}=\begin{bmatrix} \hat{a}^\e_\omega(0,\kappa)\\ \hat{b}^\e_\omega(0,\kappa) \end{bmatrix}&+\int_0^zdu \int_{\{\vert \kappa-q \vert <\vert k_\omega\vert/\e  \}\cap\calS}m(dq) \tilde{H}^{\e}_\omega(u,\kappa,q)   \begin{bmatrix} \hat{A}^\e_\omega(u,\kappa-q)\\ \hat{B}^\e_\omega(u,\kappa-q) \end{bmatrix},
\end{split}\ee
extended to $0$ for $\vert\kappa\vert>\vert k_\omega\vert/\e$, and where
\begin{equation}\label{defHt}\begin{split}
\tilde H^{\e}_\omega(z,\kappa,q)=&\frac{i \e^{s-2} k_\omega}{2}\hat{V}(z/\e,\kappa-q) \begin{bmatrix} e^{ik_\omega( \lambda^r_{\e,\omega}(q)- \lambda^r_{\e,\omega}(\kappa))z/\e^2} & e^{-ik_\omega( \lambda^r_{\e,\omega}(q)+\lambda^r_{\e,\omega}(\kappa))z/\e^2} \\ -e^{ik_\omega( \lambda^r_{\e,\omega}(q)+\lambda^r_{\e,\omega}(\kappa))z/\e^2} & -e^{-ik_\omega( \lambda^r_{\e,\omega}(q)- \lambda^r_{\e,\omega}(\kappa))z/\e^2}
\end{bmatrix}
\end{split}\end{equation}
with
\[\lambda^r_{\e,\omega}(q):=\sqrt{1-\e^2\vert q\vert^2/k^2_\omega}.\]
The important fact is that evanescent modes are absent in the system above, a minor point is that $H^{\e}_\omega$ is approximated by $\tilde{H}^{\e}_\omega$, where the absorption $\alpha_\eps$ is set to zero and the $\lambda_{\e,\omega}$ in the denominator in $H^{\e}_\omega$ are set to one. We want to show that $(\hat{A}^\e_\omega, \hat{B}^\e_\omega)$ is a good approximation of $(\hat{a}^\e_\omega, \hat{b}^\e_\omega)$. For the sake of clarity, let us assume that $\tilde{H}^{\e}_\omega$ is replaced by $H^{\e}_\omega$ in \fref{AB} since this is not the main issue. Integrating \fref{cmode} in $z$ and taking the difference with \fref{AB}, we then need to prove that two homogeneous solutions to \fref{AB} for $|\kappa|<R$, that is solutions with the first term in the r.h.s set to zero, are close to each other as $\eps \to 0$. The solutions are not equal for $\eps$ fixed, otherwise there would not be any evanescent modes in \fref{cmode}, which is obviously wrong. The main difficulty is that the system is not closed: even if we only consider modes with $|\kappa|<R$ in the l.h.s., modes with larger wavenumbers are involved in the r.h.s.. We will use the following observation to overcome this issue: two modes with very different wavenumbers, say one with $|\kappa|<R$ and the other with $|\kappa'|\gg R$, are related to each other only after a large number of interactions with the medium; since there is some loss of amplitude at each interaction (due to scattering and not the artificial absorption $\alpha_\eps$), the coupling between the $|\kappa|$ and the $|\kappa'|$ modes is expected to be small. This idea can be seen as a form of asymptotic closure of the non-closed system, and is implemented in the following proposition, proved in Section \ref{secevaprop}:

\begin{prop}[Coupling with evanescent modes is negligible]\label{evaprop}
For all $z\in[0,L]$, for all $\mu>0$, and  for all test function $\phi \in \calC^\infty_0(\Rd)\times  \calC^\infty_0(\Rd)$, we have
\[
\lim_{\e\to 0}\Pro\Big( \Big\vert \Big<\begin{bmatrix} \hat{a}^\e_\omega(z)\\ \hat{b}^\e_\omega(z) \end{bmatrix}-\begin{bmatrix} \hat{A}^\e_\omega(z)\\ \hat{B}^\e_\omega(z) \end{bmatrix} ,\phi \Big>_{L^2(\Rd)\times L^2(\Rd)} \Big\vert>\mu\Big)=0.
\]
\end{prop}

An easy consequence of this result is that for all $\mu>0$, $T>0$, and test function $\phi$ with $\hat \phi\in  \calC^\infty_0(\Rd)$, we have
\be \label{CVpp1} \lim_{\e\to 0}\Pro\Big(\sup_{t\in[-T,T]}\vert \big< p^\e_{L}(t)-p^\e_{1,L}(t),\phi\big>_{L^2(\Rd)}\vert>\mu\Big)=0,\ee
where $p^\e_L$ is given by \eqref{pulse} and
\[
p^\e_{1,L}(t,x):=\iint d\omega d\kappa e^{-i\omega t}e^{-i\kappa\cdot x} \frac{ e^{-i k_\omega (L-L_S)/\e^2}  }{2\sqrt{\lambda_{\e,\omega}(\kappa)}}\hat{A}^{\e}_\omega(L,\kappa) e^{ik_\omega\lambda_{\e,\omega}(\kappa)L/\e^2}.\]
Let us remark that the left-going mode amplitude is not involved in $p^\e_{1,L}$ because of the boundary condition $\hat{b}^{\e}_\omega(L)=0$ (no wave is coming from the right homogeneous space). Since $p^\e_{L}-p^\e_{1,L}$ converges in probability to $0$, it is then enough to investigate the limit in law of $p^\e_{1,L}$ to prove Theorem \ref{mainth1} (see \cite[Theorem 3.1 pp. 27]{billingsley}). Even if this latter convergence only holds in a weak sense, energy estimates given in Section \ref{exist} will allow us to obtain the strong convergence in $L^2(\Rd)$.

The second step of the proof is to study $p^\e_{1,L}$ and therefore the couple $(\hat{A}^\e_\omega, \hat{B}^\e_\omega)$. The system \fref{AB} is closed, but the backscattered mode amplitude $\hat{b}^\e_\omega(0)$ is unknown and cannot be considered as an initial condition. This issue is corrected by introducing the propagator $\calP^\e_\omega$, defined as the solution to, for all $|\kappa|<|k_\omega|/\eps$,
\[
\calP^\e_\omega(z,\kappa,r)=\begin{bmatrix} \delta(\kappa-r)&0\\0& \delta(\kappa-r)\end{bmatrix} +\int_0^z du\int_{\{\vert \kappa-q \vert <\vert k_\omega\vert/\e  \}\cap\calS}m(dq) \tilde{H}^{\e}_\omega(u,\kappa,q)  \calP^\e_\omega(u,\kappa-q,r),
\]
where the initial condition is the identity operator. The propagator is extended to zero for $|\kappa|>|k_\omega|/\eps$. According to symmetries in $\tilde{H}^{\e}_\omega$, the propagator can be decomposed as follows 
\[\calP^\e_\omega(z,\kappa,r)=\begin{bmatrix} \calP^{a,\e}_\omega(z,\kappa,r) &\overline{ \calP^{b,\e}_\omega(z,\kappa,r) }  \\  \calP^{b,\e}_\omega(z,\kappa,r)  & \overline{\calP^{a,\e}_\omega(z,\kappa,r) }  \end{bmatrix},\]
where $(\calP^{a,\e}_\omega,\calP^{b,\e}_\omega)$ is the solution of
\begin{equation}\label{schroeq2}
\begin{bmatrix}\calP^{a,\e}_\omega(z,\kappa,r)\\ \calP^{b,\e}_\omega(z,\kappa,r)\end{bmatrix}=\begin{bmatrix} \delta(\kappa-r)\\0\end{bmatrix} +\int_0^z du\int_{\{\vert \kappa-q \vert <\vert k_\omega\vert/\e  \}\cap\calS}m(dq) \tilde{H}^{\e}_\omega(u,\kappa,q) \begin{bmatrix}\calP^{a,\e}_\omega(u,\kappa-q,r)\\ \calP^{b,\e}_\omega(u,\kappa-q,r)\end{bmatrix}.
 \end{equation}
The term $\calP^{a,\e}_\omega$ describes scattering to the same direction of propagation, while $\calP^{b,\e}_\omega$ describes scattering to the opposite direction. We then find the following relation between the right and left going modes in terms of the propagator: 
\[\begin{bmatrix}\hat{A}^\e_\omega(z,\kappa)\\ \hat{B}^\e_\omega(z,\kappa)\end{bmatrix}=\int \calP^\e_\omega(z,\kappa,r)\begin{bmatrix}\hat{a}^\e_\omega(0,r)\\ \hat{b}^\e_\omega(0,r)\end{bmatrix}dr,\qquad \forall z\in[0,L].\]
The expression of the wave exiting the random section is thus, after integration against a test function $\hat \phi\in \calC^\infty_0(\Rd)$,
\[ \big< p^\e_{1,L}(t),\phi\big>:=p^\e_{a,L}(t,\phi)+p^\e_{b,L}(t,\phi),\]
with
\begin{align*}
& p^\e_{a,L}(t,\phi):=\iint  d\omega d\kappa e^{-i\omega t} \psi^\e_\omega(\kappa)\overline{\hat{\phi}(\kappa)}  \int dr \calP^{a,\e}_\omega(L,\kappa,r)\phi^\e_\omega(r)\\
&p^\e_{b,L}(t,\phi):=\iint d\omega dr  e^{-i\omega t}e^{ik_\omega L_S/\e^2}\hat{b}^\e_\omega(0,r)\int d\kappa\overline{ \calP^{b,\e}_\omega(L,\kappa,r)}   \psi^\e_\omega(\kappa)\overline{\hat{\phi}(\kappa)}.
\end{align*}
Here, we have defined the following functions,
\[\begin{split}
\phi^\e_\omega(r)&:=\frac{\sqrt{\lambda_{\e,\omega}(r)}}{2} \hat{f}_0(\omega,r) e^{-ik_\omega(\lambda_{\e,\omega}(r)-1)L_S/\e^2}, \quad \psi^\e_\omega(\kappa):=\frac{e^{ik_\omega(\lambda_{\e,\omega}(\kappa)-1)L/\e^2}}{2\sqrt{\lambda_{\e,\omega}(\kappa)}}.
\end{split}\]
Before describing the asymptotic behavior of the propagator, we need to introduce a few more notations:
\begin{equation}\label{defphi0}
\begin{split}
&\calP^{a,\e}_{\omega,\phi^\e_\omega}(L,\kappa):=\int \calP^{a,\e}_\omega(L,\kappa,r)\phi^\e_\omega(r)dr,\qquad \calP^{\ast,b,\e}_{\omega,\psi^\e_\omega}(L,r):=\int \overline{\calP^{b,\e}_\omega(L,\kappa,r)}\psi^\e_\omega(\kappa) \overline{\hat{\phi}(\kappa)}d\kappa\\
&\phi^0_\omega(r):=\hat{f}_0(\omega,r)e^{i\vert r \vert ^2 L_S/ k_\omega }.
\end{split}\end{equation}
The next result shows that the backscattering is negligible and provides us with the leading term in the propagator.

\begin{prop}[Backscattering is negligible]\label{noback} 
We have the  following two statements:
\begin{enumerate}
\item For all $\eta>0$, we have
\[\lim_{\e\to 0}\Pro\Big(\int  \|  \calP^{\ast,b,\e}_{\omega,\psi^\e_\omega}(L)\|_{L^2(\Rd)} d\omega >\eta \Big)=0,
\]
where $\calP^{\ast,b,\e}_{\omega,\psi^\e_\omega}$ is extended by $0$ for $\vert \kappa\vert >\vert k_\omega \vert/\e$.
\item For all $\eta>0$, we have 
\[
\lim_{\e\to 0}\Pro\Big(\int \|  \calP^{a,\e}_{\omega,\phi^\e_\omega}(L)-\calX^\e_\omega (L)\|_{L^2(\Rd)} d\omega >\eta \Big)=0,
\]
where $\calX^\e_\omega$ is the solution to
\begin{equation}\label{eqdiff0}
\partial_z \calX^\e_\omega (z,\kappa)=\int_{\{\vert \kappa-q\vert<\vert k_\omega\vert /\e\}\cap\calS}m(dq) \tilde{H}^{\e}_{\omega,1,1}(z,\kappa,q)\calX^\e_\omega (z,\kappa-q)\quad\text{with}\quad\calX^\e_\omega (0)=\phi^\e_\omega,
\end{equation}
for $\vert \kappa\vert <\vert k_\omega\vert/\e$, and $\tilde{H}^{\e}_{\omega}$ is defined by \eqref{defHt} ($ \tilde{H}^{\e}_{\omega,1,1}$ is the $(1,1)$ entry of $\tilde{H}^{\e}_{\omega}$). Here, $ \calP^{a,\e}_{\omega,\phi^\e_\omega}$ and $\calX^\e_\omega$ are extended by $0$ for $\vert \kappa\vert >\vert k_\omega \vert/\e$.
\end{enumerate}
\end{prop}

Proposition \ref{noback} is proved in Section \ref{proofnoback}. The proof is based on a series expansion of the propagator and on the fact that the coupling between right and left going modes appears via an oscillatory integral. Note also that neglecting the backscattering leads to an IVP on $\calX^\e_\omega$. 

The last step is to characterize the limit of $\calX^\e_\omega$. With the same arguments as before, we only need to investigate the convergence in law of $\calX^\e_\omega$ to prove Theorem \ref{mainth1}. The proof, given in Section \ref{proofItoSc}, is also based on a series expansion of $\calX^\e_\omega$ and on the computation of the limiting moments of $\calX^\e_\omega$. We will need the following functional spaces: for $k\in\mathbb{N}^\ast$, let us denote by
 \[\calH_k:=\Big\{\phi=\hat{\psi}\quad \text{with}\quad  \psi\in H^k(\Rd) \Big\} \qquad \text{with} \qquad\|\phi\|^2_{\calH_k}:=\int (1+\vert \kappa \vert^2)^{k/2} \vert \phi (\kappa)\vert^2 d\kappa.\]
Consider also $\hat{W}^{\alpha}_k(0,L):=W^{\alpha}(0,L,\calH_k)$, equipped with the norm $\| \cdot \|_{\alpha,k}:= \| \cdot \|_{\alpha,\calH_k}$, and  the complete metric space
\[  \hat{W}^{\alpha}_\infty(0,L):=\bigcap_{k\in \mathbb{N}^\ast}\hat{W}^{\alpha}_k(0,L) \qquad \textrm{equipped with} \qquad
\hat{d}_{\alpha,\infty} (\phi,\psi) :=\sum_{k\geq 1} \frac{1}{2^k} \left(1\wedge \| \phi -\psi \|_{\alpha,k}\right).\] 
We finally introduce the complete metric space of $\calH_k$-valued functions with H\"older regularity 
\[\hat{\mathcal{C}}^{H-\theta}_\infty(0,L):=\bigcap_{k\in \mathbb{N}^\ast}\mathcal{C}^{H-\theta}([0,L],\mathcal{H}_k),\]
equipped with 
 \[ \hat{d}_{H-\theta,\mathcal{C},\infty} (\phi,\psi) :=\sum_{k\geq 1} \frac{1}{2^k} \left(1\wedge \| \phi -\psi \|_{H-\theta,\mathcal{C},\mathcal{H}_k}\right).\] 
We will also use the spaces $\hat{\mathcal{C}}^{0}_\infty(0,L)$ and $\hat{\mathcal{C}}^{\infty}_\infty(0,L)$ with immediate definitions.  Let us recall that according to \eqref{inclusion}, we have $\hat{\mathcal{C}}^{H-\theta}_\infty(0,L)\subset \hat{W}^{\alpha}_\infty(0,L)$ for $\theta=H-\alpha-1$. The convergence result is the following:

\begin{prop}[Convergence to the fractional It\^o-Schr\"odinger equation]\label{ItoSc} We have the three statements below: 
\begin{enumerate}
\item For all $M\in\mathbb{N}^\ast$ and frequencies $(\omega_{1},\dots,\omega_M)$, the family $(\calX^{\e}_{\omega_1}(L),\dots,\calX^{\e}_{\omega_M}(L))$ converges in law in $L^2(\mathbb{R}^{2M})$ to $(\calX_{\omega_1}(L),\dots,\calX_{\omega_M}(L))$. Here, $\calX_{\omega}$ is the unique pathwise solution in $\hat{W}^{\alpha}_\infty(0,L)$, for all $\alpha\in(1-H,1/2)$, of 
\be \label{fracit}
\calX_\omega(z,\kappa)= \phi^0_\omega(\kappa)+ik_\omega \sigma_H\int_{\calS} m(dq) \int_0^z du    e^{-i(\vert \kappa-q\vert ^2-\vert \kappa \vert^2)u/(2k_\omega)}\calX_\omega(u,\kappa-q) dB_H(u,q),
 \ee
where  $B_H$ is defined by \eqref{FBMB} and $\sigma_H$ by \eqref{FBM}. Here, the stochastic integral is defined $\Pro$-almost surely pointwise in $\kappa$ and $q$. Moreover, $\calX_{\omega}\in \hat{\mathcal{C}}^{H-\theta}_\infty(0,L)$ for all $\theta\in(0,H-1/2)$, and
\be \label{conservX}
\| \calX_\omega(z)\|_{L^2(\Rd)}=\| \calX_\omega(0)\|_{L^2(\Rd)}=\frac{1}{2}\| \hat{f}_0(\omega,\cdot) \|_{L^2(\Rd)}.\ee
\item We have for all $(M_1,M_2)\in \mathbb{N}^2$, all frequencies $(\omega_{ij})_{(i,j)\in \{1,2\}\times\{1,\dots,M_i\}}$, and test function $\varphi \in L^2(\mathbb{R}^{2(M_1+M_2)})$,
\[\begin{split}
\lim_{\e\to 0}\E\Big[\Big<\prod_{j_1=1}^{M_1}& \calX^{\e}_{\omega_{1,j_1}}(L)\prod_{j_2=1}^{M_2} \overline{\calX^{\e}_{\omega_{2,j_2}}(L)},\varphi\Big>_{L^2(\mathbb{R}^{2(M_1+M_2)})} \Big]\\
&=\E\Big[\Big<\prod_{j_1=1}^{M_1} \calX_{\omega_{1,j_1}}(L)\prod_{j_2=1}^{M_2} \overline{\calX_{\omega_{2,j_2}}(L)},\varphi\Big>_{L^2(\mathbb{R}^{2(M_1+M_2)})} \Big].
\end{split}\]
\item Therefore, the process defined in the Fourier domain by 
\[\hat{\Psi}_\omega(z,\kappa):=e^{-i\vert \kappa\vert^2z/(2k_\omega)}\calX_\omega(z,\kappa)\]
 satisfies all the requirements of Theorem \ref{mainth2} with $W_H$ defined by \eqref{FBM}.
\end{enumerate}
\end{prop}

In order to identify the moments of $\calX_{\omega}$ with the limits of those of $\calX^{\e}_{\omega}$, we will identify the moments of each term in a Duhamel expansion of $\calX_{\omega}$. The main technical difficulty is to handle the fractional stochastic integral and to justify the calculations, in particular exchange of expectation, limit, and integration. We will for this proceed by regularization, and start by constructing an approximate solution $\calX^A_{\omega}$ that solves \fref{fracit} with $B_H$ replaced by $B_H^A$ defined by, for $q\in\calS$,
\[ B^A_H(u,q):=C^{1/2}_H  \sum_{n\geq 1}  \sqrt{\beta_n} e_n(q)\int_{-A}^{A} \frac{e^{iru}-1}{ir \vert r \vert^{H-1/2}} w_n(dr),\]
with $C_H=H\Gamma(2H)\sin(\pi H)/\pi$, and where $(w_n(dr))_{n\geq 1}$ is the family of independent complex Gaussian random measure given in the spectral representation \eqref{spectralB}. Since $B^A_H$ is $\calC^\infty$ w.r.t. to $u$, the integral in \fref{fracit} is now simply a Lebesgue integral, and computations can be easily justified. We then pass to the limit $A \to+ \infty$ in order to construct and characterize solutions to \fref{fracit}. These points are addressed in the proof of Proposition \ref{ItoSc} in Section \ref{proofItoSc}, and the proofs of Theorems \ref{mainth1} and \ref{mainth2} are then straightforward owing the previous propositions.

\section{Existence theory and  estimates} \label{exist}

This section is devoted to the existence of solutions to \fref{helm1}, their regularity, and to the derivation of some important estimates that will be used throughout the proof.  We introduce first the Green's function
\[
G^\eps_\omega(z,x)=\frac{e^{i k_\omega^\eps |\bx|}}{4 \pi |\bx|}, \qquad \textrm{with} \quad\bx=(z,x) \quad \textrm{and} \quad  k_\omega^\eps = k_\omega \sqrt{1+ i \alpha_\eps \eps^4 /k_\omega^2}.
\]
We then recast \fref{helm1} into the integral form
\be \label{intform}
\check{p}^\e_\omega-T_\eps \check{p}^\e_\omega =u_{0,\eps},
\ee
where
\begin{align*}
&T_\eps u(z,x)=\frac{k_\omega^2 \eps^{s-4}}{4\pi}\int_{(0,L)\times \Rm^2} G_\omega^\eps(z-z',x-x')V\Big(\frac{ z'}{\e},\frac{x'}{\e} \Big) u(z',x') dz' dx'\\
&u_{0,\eps}(\omega,z,x)=\int_{\Rm^2}  L_{0,\eps}(\omega,z,x,x') \check{f}_0\Big(\omega,\frac{x'}{\e} \Big)dx'
\end{align*}
and 
\[
L_{0,\eps}(\omega,z,x,x')=(z-L_S) \frac{e^{i k_\omega^\eps \sqrt{|z-L_S|^2+|x-x'|^2}}}{4 \pi (|z-L_S|^2+|x-x'|^2)^\frac{3}{2}}\left( i k^\eps_\omega\left(|z-L_S|^2+|x-x'|^2\right)^\frac{1}{2} -1\right).
\]
We have then the following lemma:
\begin{lem} \label{lemex} For all $\eps>0$ and all $k_\omega \in \Rm$, \fref{intform} admits a unique solution $\check{p}^\e_\omega$ in
\[ W^{2,p}((-\infty,L_S) \times \Rm^2) \cap W^{2,p}((L_S,\infty) \times \Rm^2)
\]
that satisfies \fref{helm1} in the distribution sense and almost surely.
\end{lem}
\begin{proof} First of all, since the potential $V$ is bounded, it follows from Riesz compactness criterion (see \cite[Theorem XIII.66 pp. 248]{RS-80-4}) that the operator $T_\eps$ is compact in $L^p(\Rm^3)$, $1\leq p<\infty$. Moreover, application of the Young inequality show that $u_{0,\eps} \in L^q((0,L) \times \Rm^2)$, $1\leq q  \leq \infty$. The existence and uniqueness of a solution to \fref{intform} in $L^p((0,L) \times \Rm^2)$ is then a consequence of the Fredholm alternative. Remarking further that $u_{0,\eps} \in L^p((-\infty,L_S) \times \Rm^2) \cap L^p((L_S,\infty) \times \Rm^2)$, the solution $\check{p}^\e_\omega$ belongs to the latter space and satisfies \fref{helm1} in the distribution sense. Standard elliptic regularity finally yields $\check{p}^\e_\omega \in W^{2,p}((-\infty,L_S) \times \Rm^2) \cap W^{2,p}((L_S,\infty) \times \Rm^2)$.
\end{proof}

A first consequence of the latter lemma is that the Helmholtz equation is satisfied almost everywhere for $z >L_S$ and $z<L_S$. A second consequence is that $\hat{p}^\e_\omega$ and $\partial_z \hat{p}^\e_\omega$ both admit limits as $z \to L_S^\pm$. The Helmholtz equation then yields the following jump conditions  across the plane $z=L_S$, 
\be \label{jump}
\hat{p}^\e_\omega(L_S^+,\kappa)-\hat{p}^\e_\omega(L_S^-,\kappa)=\hat{f}_0(\omega,\kappa)\quad\text{and}\quad \partial_z\hat{p}^\e_\omega(L_S^+,\kappa)-\partial_z\hat{p}^\e_\omega(L_S^-,\kappa)=0.
\ee 
We then use these relations to solve the Helmholtz equation for $z<L_S$, for $z\in(L_S,0)$, and for $z>L$. This will allow us to derive boundary conditions at $z=0$ and $z=L$ that will lead to some estimates on $\hat{p}^\e_\omega$. We will need for this some properties of the $\lambda_{\eps,\omega}$ defined in \fref{deflam}, which follows from some of the principal square root. For a complex number $z=u+iv$, with $v\neq 0$, the principal square root admits the expression
\be \label{expsqrt}
\sqrt{z}=\frac{1}{\sqrt{2}}\Big(\sqrt{\sqrt{u^2+v^2}+u}+i\,sign(v)\sqrt{\sqrt{u^2+v^2}-u} \Big).
\ee
As a consequence,
\be \label{propsqrt}
\textrm{Re}(\lambda_{\e,\omega}(\kappa)) \geq 0, \quad \forall \kappa \in \Rm^2, \quad \textrm{and} \quad \textrm{Im}(\lambda_{\e,\omega}(\kappa)) \leq C \sqrt{\alpha_{\omega,\e}}, \quad \textrm{for} \quad 1 -\e^2 \vert \kappa \vert ^2/k^2_\omega \geq 0.
\ee
For the second inequality above, we used the fact that the square root is of H\"older regularity $1/2$. We will also need the following expressions, that are consequences of \fref{lrmode}-\fref{lrmode2}:
\begin{equation}\label{formulea}  \hat{a}^\e_\omega(z,\kappa)=\frac{ \sqrt{\lambda_{\e,\omega}(\kappa)}e^{-ik_{\omega}\lambda_{\e,\omega}(\kappa)z/\e^2}}{2} \Big[\hat{p}^\e_\omega(z,\kappa)+\frac{\e^2}{ik_{\omega}\lambda_{\e,\omega}(\kappa)}\partial_z\hat{p}^\e_\omega(z,\kappa)\Big]  \end{equation}
and
\begin{equation}\label{formuleb} \hat{b}^\e_\omega(z,\kappa)=\frac{\sqrt{\lambda_{\e,\omega}(\kappa)}e^{ik_{\omega}\lambda_{\e,\omega}(\kappa)z/\e^2}}{2} \Big[\hat{p}^\e_\omega(z,\kappa)-\frac{\e^2}{ik_{\omega}\lambda_{\e,\omega}(\kappa)}\partial_z\hat{p}^\e_\omega(z,\kappa)\Big].\end{equation}

\paragraph{Solution for $ z\in (-\infty,L_S) \cup (L_S,0)$.} Since there are no sources at $z=-\infty$, there are no right-traveling waves in $z<L_S$ and therefore $\hat{a}^\e_\omega(z,\kappa)=0$ for $z<L_S$. Moreover, since the medium is homogeneous in $(-\infty,0)$, the coefficients $\hat{a}^\e_\omega(z,\kappa)$ and $\hat{b}^\e_\omega(z,\kappa)$ are constant in $(-\infty,L_S) \cup (L_S,0)$. Using the jump conditions \fref{jump}, we find 
\be \label{expa}
\hat{a}^\e_{\omega}(L_S^+,\kappa)=\frac{\sqrt{\lambda_{\e,\omega}(\kappa)}}{2}\hat{f}_0(\omega,\kappa)e^{-i k_\omega \lambda_{\e,\omega}(\kappa) L_S/\e^2}.
\ee
The coefficient $\hat{b}^\e_{\omega}$ is unknown at this point in $(L_S,0)$. We then eliminate it in the expression of $\hat{p}^\e_\omega(z,\kappa)$, which leads after direct manipulations to the following boundary condition at $z=0$:
\be \label{CL1}
\partial_z \hat{p}^\e_\omega(0,\kappa)+ \frac{i k_\omega \lambda_{\e,\omega}(\kappa)}{\e^2}\hat{p}^\e_\omega(0,\kappa)=\frac{2 i k_\omega \sqrt{\lambda_{\e,\omega}(\kappa)} }{\e^2}\hat{a}^\e_{\omega}(L_S^+,\kappa), \qquad \forall \kappa \in \Rd.
\ee
Note that we used here the fact that $\partial_z \hat{p}^\e_\omega(z,\kappa)$ and $\hat{p}^\e_\omega(z,\kappa)$ are continuous in $z$, $\kappa$ a.e., according to the regularity of Lemma \ref{lemex}.
\paragraph{Solution for $ z\in (L,+\infty)$.} As in the previous case, there are no sources at $z=+\infty$, and therefore $\hat{b}^\e_\omega(z,\kappa)=0$ for $z \geq L$. Since $\hat{a}^\e_\omega(z,\kappa)$ is constant in $(L,+\infty)$ and unknown, we can eliminate it in the same fashion as above to obtain the following boundary condition at $z=L$: \be \label{CL2}
\partial_z \hat{p}^\e_\omega(L,\kappa)= \frac{i k_\omega \lambda_{\e,\omega}(\kappa)}{\e^2}\hat{p}^\e_\omega(L,\kappa), \qquad \forall \kappa \in \Rd.
\ee
We then use the boundary conditions \fref{CL1}--\fref{CL2} to arrive at the following result:
\begin{lem}\label{borneL2} The wavefield $\hat{p}^\e_\omega$ satisfies the following estimates:
\begin{align}
\label{L21}&\int_{\Rd}\textrm{Re}(\lambda_{\e,\omega}(\kappa))|\hat{p}^\e_\omega(L,\kappa)|^2 d\kappa+\int_{\Rd}\textrm{Re}(\lambda_{\e,\omega}(\kappa))|\hat{p}^\e_\omega(0,\kappa)|^2 d\kappa+ \frac{\alpha_{\e}\eps^2}{k_\omega}\|\hat{p}^\e_\omega\|^2_{L^2((0,L)\times \Rd)}\\
& \hspace{2cm} \leq C \|\hat f_0(\omega,\cdot)\|^2_{L^2(\Rm^2)}, \nonumber\\
\label{L22}&\int_{\Rd}\textrm{Im}(\lambda_{\e,\omega}(\kappa))|\hat{p}^\e_\omega(L,\kappa)|^2 d\kappa+\int_{\Rd}\textrm{Im}(\lambda_{\e,\omega}(\kappa))|\hat{p}^\e_\omega(0,\kappa)|^2 d\kappa+\frac{\eps^2}{k_\omega}\| \nabla \check{p}^\e_\omega\|^2_{L^2((0,L)\times \Rd)} \\
& \nonumber \hspace{2cm}\leq  C \|\hat f_0(\omega,\cdot)\|^2_{L^2(\Rm^2)} \left( 1+\frac{k_\omega^2}{\alpha_\eps\eps^4 }\right),\\
&\label{L23}\|\hat p_\omega^\eps(L,\cdot)\|_{L^2(\Rm^2)}+\|\hat p_\omega^\eps(0,\cdot)\|_{L^2(\Rm^2)}\leq  C \|\hat f_0(\omega,\cdot)\|_{L^2(\Rm^2)}.
\end{align}
We have moreover
\be\label{L24}
\overline{\lim_{\e\to0}}\|\hat p_\omega^\eps(L,\cdot)\|_{L^2(\Rm^2)}\leq \frac{1}{2} \|\hat f_0(\omega,\cdot)\|_{L^2(\Rm^2)}.
\ee
\end{lem}
\begin{proof}
We start by multiplying \eqref{helmholtzF} by $\overline{\hat{p}^\e_\omega(z,\kappa)}$. Integrating in $(z,\kappa)$ over $[0,L]\times \Rd$, using boundary conditions \fref{CL1}--\fref{CL2}, and taking first the imaginary part leads to
\begin{align*}
&\int_{\Rd}\textrm{Re}(\lambda_{\e,\omega}(\kappa))|\hat{p}^\e_\omega(L,\kappa)|^2 d\kappa+\int_{\Rd}\textrm{Re}(\lambda_{\e,\omega}(\kappa))|\hat{p}^\e_\omega(0,\kappa)|^2 d\kappa+ \frac{\alpha_{\e}\eps^2}{k_\omega} \|\hat{p}^\e_\omega\|^2_{L^2((0,L)\times \Rd)}\\
&\hspace{2cm}=2 \textrm{Re}\left( \int_{\Rd}\sqrt{\lambda_{\e,\omega}(\kappa)} \hat{a}^\e_{\omega}(L_S^+,\kappa)\overline{\hat{p}^\e_\omega(0,\kappa)} d\kappa\right).
\end{align*}
Since Re$(\lambda_{\e,\omega}(\kappa)) \geq 0$ for all $\kappa$ according to \fref{propsqrt}, we then find 
\[
\int_{K} \textrm{Re}(\lambda_{\e,\omega}(\kappa))|\hat{p}^\e_\omega(0,\kappa)|^2  d\kappa\leq 2 \left(\int_{K} |\lambda_{\e,\omega}(\kappa)||\hat{a}^\e_{\omega}(L_S^+,\kappa)|^2 d\kappa \right)^{1/2} \left(\int_{K} |\hat{p}^\e_\omega(0,\kappa)|^2 d\kappa \right)^{1/2},
\]
where $K$ is the support in the $\kappa$ variable of $\hat f_0$. Since there are constants $C_1$, $C_2$ and $C_3$ such that (the last inequality follows from \fref{propsqrt}),
\[
0<C_1 \leq \textrm{Re}(\lambda_{\e,\omega}(\kappa)), \qquad |\lambda_{\e,\omega}(\kappa)| \leq C_2, \qquad e^{ \vert k_\omega \vert \textrm{Im}(\lambda_{\e,\omega}(\kappa))L_S/\eps^2}\leq C_3,\qquad \forall \kappa \in K,
\]
we deduce from \fref{expa} that 
\be \label{estP0}
\int_{K} |\hat{p}^\e_\omega(0,\kappa)|^2 d\kappa \leq C  \int_{\Rd} |\hat f_0(\omega,\kappa)|^2 d\kappa,
\ee
and therefore \fref{L21}. For the second estimate, we take now the real part, and obtain
\begin{align*}
&\frac{\eps^2}{k_\omega} \int_{(0,L)\times \Rd} |\nabla \check{p}^\e_\omega(z,x)|^2 dzdx+\int_{\Rd}\textrm{Im}(\lambda_{\e,\omega}(\kappa))|\hat{p}^\e_\omega(L,\kappa)|^2 d\kappa+\int_{\Rd}\textrm{Im}(\lambda_{\e,\omega}(\kappa))|\hat{p}^\e_\omega(0,\kappa)|^2 d\kappa\\
&\hspace{0.3cm}=\frac{k_\omega}{\e^2}\int_{(0,L)\times \Rd} \Big( 1+\e^{s} V\Big(\frac{ z}{\e},\frac{x}{\e} \Big)\Big)|\check{p}^\e_\omega (z,x)|^2 dzdx+2 \textrm{Im}\left( \int_{\Rd}\sqrt{\lambda_{\e,\omega}(\kappa)} \hat{a}^\e_{\omega}(L_S^+,\kappa)\overline{\hat{p}^\e_\omega(0,\kappa)} d\kappa\right).
\end{align*}
Using the fact that $\textrm{Im}(\lambda_{\e,\omega}(\kappa)) \geq 0$, that $V$ is bounded uniformly in $(z,x)$, expression \fref{expa}, estimates \fref{estP0} and \fref{L21}, the estimate \fref{L22} follows directly. Estimates \fref{L23} and \fref{L24} are obtained in a similar manner as \fref{L21}, we instead multiply \eqref{helmholtzF} by $\overline{\hat p^\eps_\omega(z,\kappa)}\lambda_{\e,\omega}(\kappa)^{-1}$. In that case, we obtain
 \[\|\hat p_\omega^\eps(L,\cdot)\|^2_{L^2(\Rm^2)}+\|\hat p_\omega^\eps(0,\cdot)\|^2_{L^2(\Rm^2)} \leq \|\hat{f}_0(\omega,\cdot)e^{\vert k_\omega\vert Im(\lambda_{\e,\omega}(\cdot))L_S/\e^2}\|_{L^2(\Rm^2)} \|\hat p_\omega^\eps(0,\cdot)\|_{L^2(\Rm^2)},\]
 and we conclude using Young's inequality. This ends the proof.
\end{proof}

Following definitions \fref{formulea}-\fref{formuleb}, it is then direct to estimates $\hat{a}^\e_{\omega}$ and $\hat{b}^\e_{\omega}$ from Lemma \ref{borneL2}:
\begin{cor} The following estimates are satisfied:
\begin{align}
\label{E1}&\int_{(0,L)\times \{\vert \kappa \vert <\vert k_\omega \vert /\e\}} \Big(|\hat{a}^\e_{\omega}(z,\kappa)|^2+|\hat{b}^\e_{\omega}(z,\kappa)|^2 \Big)dzd \kappa \leq \frac{C}{\alpha_\eps^{\frac{3}{2}} \eps^4} \|\hat f_0(\omega,\cdot)\|^2_{L^2(\Rm^2)}\\
\label{E2}&\int_{\{|\kappa| < \vert k_\omega\vert /\e\}} |\hat{b}^\e_{\omega}(0,\kappa)|^2d \kappa \leq C \|\hat f_0(\omega,\cdot)\|^2_{L^2(\Rm^2)}.
\end{align}
\end{cor}
\begin{proof} The bound \fref{E1} is a consequence of \fref{formulea}-\fref{formuleb}-\fref{L21}-\fref{L22} and the relation
\[
\sqrt{\alpha_{\omega,\e}} \leq |\lambda_{\e,\omega}| \leq \frac{1}{2} (1+|\lambda_{\e,\omega}|^2)\leq \frac{1}{2} (1+\alpha_{\omega,\e}+(1+\eps^2|\kappa|^2/k_\omega^2)).
\]
For \fref{E2}, we use \fref{lrmode} at $z=0$, together with \fref{L21}, along with the fact that $ \sqrt{2} Re(\lambda_{\e,\omega}(\kappa)) \geq |\lambda_{\e,\omega}(\kappa)|$ when $|\kappa| \leq k_\omega /\e$, and the calculation below:
\bee
\nonumber \int_{\{|\kappa| < \vert k_\omega\vert /\e\}} |{b}^\e_{\omega}(0,\kappa)|^2dzd \kappa & \leq &C_1 \int_{\{\vert \kappa\vert <\vert k_\omega\vert /\e\}}d\kappa \vert \lambda_{\e,\omega}(\kappa)\vert \left(\vert \hat{p}^\e_\omega(0,\kappa)\vert^2 + \vert \hat{f}_0(\omega,\kappa)\vert ^2\right)  \\
\nonumber &\leq &C_2 \int_{\{\vert \kappa\vert <\vert k_\omega\vert /\e\}}d\kappa Re(\lambda_{\e,\omega}(\kappa)) \vert \hat{p}^\e_\omega(0,\kappa)\vert^2 + C_2 \|\hat{f}_0(\omega,\cdot)\|^2_{L^2(\Rd)}\\
& \leq& C_3 \|\hat{f}_0(\omega,\cdot)\|^2_{L^2(\Rd)}.
\eee
This ends the proof.
\end{proof}

Let us remark that the absorption term $\alpha_\eps$ is introduced precisely in order to obtain \fref{E1}, which allows us to control $\hat{a}^\e_{\omega}$ and $\hat{b}^\e_{\omega}$ on $(0,L)$ and not just at $z=0$ and $z=L$. The estimate is used in the proof of the fact that the coupling with evanescent modes in negligible, see section \ref{secevaprop}.

\section{Technical results on iterated integrals}\label{RP}

This section is devoted to crucial technical results that will be used throughout the paper. The following fact is of importance: for $f(u_1,\dots,u_n)$ an integrable function, invariant with respect to any permutation $\sigma$, that is
\[
f(u_1,\dots,u_n)=f(u_{\sigma(1)},\dots,u_{\sigma(n)}),
\]
we have
\be \label{fact}
\int_{\Delta_n(z)}f(u_1,\dots,u_n) du_1\dots du_n=\frac{1}{n!}\int_{[0,z]^n}f(u_1,\dots,u_n) du_1\dots du_n,
\ee
where $\Delta_n(z)$ is the simplex defined by
\be \label{simplexe}\Delta_n(z):=\big\{(u_1,\dots,u_n)\in[0,z]^n,\quad\text{s.t.}\quad 0\leq u_j \leq u_{j-1}\quad \forall j\in\{2,\dots,n\}\big\}.
\ee The next proposition shows that the random field $V$ satisfies a long-range property in the $z$ direction.

\begin{prop}\label{LRV} For all $z_0\in \mathbb{R}$ and $(x,y)\in \Rd\times \Rd$, we have
\[\E[V(z+z_0,x)V(z_0,y)]\underset{z\to +\infty}{\sim} \frac{C_\mathfrak{H}}{z^{\mathfrak{H}}} R_0(x-y) \quad\text{with}\quad C_\mathfrak{H}:=\frac{c_\mathfrak{H}}{2\pi}\Big(\int_{-\infty}^{+\infty}  u \Theta(u)  e^{-u^2/2}du\Big)^2.\]
The correlation function $R_0$ is defined in \fref{defR}.
\end{prop}
The proof of this proposition follows the lines of \cite[Lemma 1]{marty}. We give its proof below as a preliminary to the proof of Proposition \ref{LRM} further. 

\begin{proof}
Let us first note that
\[\begin{split}
\E[V(z&+z_0,x)V(z_0,y)]=\int_\calS \mathfrak{m}(dq) e^{-iq\cdot (x-y)}\E[\Theta(\mathcal{B}_\mathfrak{H}(z+z_0,q))\Theta(\mathcal{B}_\mathfrak{H}(z_0,q))],
\end{split}\]
so that we just need to investigate the term $\E[\Theta(\mathcal{B}_\mathfrak{H}(z+z_0,q))\Theta(\mathcal{B}_\mathfrak{H}(z_0,q))]$. The analysis is based on the Hermite polynomials defined by
\be \label{defHerm}
H_l(u):=(-1)^l \frac{g^{(l)}(u)}{g(u)}, \qquad\text{with}\qquad g(u):=\frac{e^{-u^2/2}}{\sqrt{2\pi}},\ee
which form an orthogonal basis of $L^2(\mathbb{R}, g(u)du)$: 
\be \label{ortho} \big<H_l,H_m\big>_{L^2(\mathbb{R}, g(u)du)}=l! \delta_{lm}.\ee
Decomposing $\Theta$ with respect to this basis, we have
\[\Theta(u)=\sum_{l\geq 1}\frac{\Theta_l}{l!} H_l(u)\qquad\text{where}\qquad \Theta_l:=\big<H_l,\Theta\big>_{L^2(\mathbb{R}, g(u)du)}.\]
We will also use Mehler's formula which, for two centered Gaussian random variables such that $\E[X_1^2]=\E[X_2^2]=1$, yields
\[\E[H_{l}(X_1)H_m(X_2)]= l! \E[X_1X_2]^l\delta_{lm}.\]
Then, we obtain
\[\begin{split}
\E[\Theta\big(\BB_\mathfrak{H}(z+z_0,q)\big)\Theta\big(\BB_\mathfrak{H}(z_0,q)\big)]&=\sum_{l,m\geq 1} \frac{\Theta_{l}\Theta_{m}}{l!m!}\E[H_{l}(\BB_\mathfrak{H}(z+z_0,q))H_m(\BB_\mathfrak{H}(z_0,q))]\\
&=\sum_{l\geq 1} \frac{\Theta_l^2}{l!} r^l_\mathfrak{H}(z) \hat R^l(q,q)\\
&=\Theta_1^2 r_\mathfrak{H}(z)+\sum_{l\geq 2} \frac{\Theta_l^2}{l!} r^l_\mathfrak{H}(z).
\end{split}\] 
Moreover, following \eqref{LRr}, we have $z^\mathfrak{H} r^l_\mathfrak{H}(z)\to 0$ as $z\to +\infty$ for $l\geq 2$, and also
\[\sum_{l\geq 2}\Big\vert \frac{\Theta_l^2}{l!} r^l_\mathfrak{H}(z)\Big\vert\leq C\sum_{l\geq 2}\frac{\Theta_l^2}{l!} \leq C\big<\Theta,\Theta\big>_{L^2(\mathbb{R}, g(u)du)}<+\infty,\]
 for $z$ large enough. As a result, using dominated convergence for series, we obtain
 \[z^{\mathfrak{H}}\E[\Theta\big(\BB_\mathfrak{H}(z+z_0,q)\big)\Theta\big(\BB_\mathfrak{H}(z_0,q)\big)]\underset{z\to +\infty}{\sim}  c_\mathfrak{H} \Theta_1^2,\]
which concludes the proof of Proposition \ref{LRV}.
\end{proof}

Since the proof of Theorem \ref{ItoSc} is based on a moment technique, we will be required to compute moments of the form
\[\frac{1}{\e^{n(2-s)}}\int_{\Delta_n(z)}\E\Big[ \prod_{j=1}^n \Theta\Big(\BB_\mathfrak{H}\big(\frac{u_j}{\e},p_j\big) \Big) \Big]\varphi_\e(z,u_1,\dots,u_n)du_1\dots du_n,\]
where $n$ is an even number (otherwise this moment is 0 by symmetry), and $\varphi_\e$ is a bounded function. The following result is extensively used in the forthcoming sections up to simple modifications. It provides us with crucial uniform (in $\eps$) bounds as well as with an important convergence result. 
\begin{prop}\label{LRM} 
For all even number $n\geq 2$ and $s=2-\mathfrak{H}/2$, there exists a constant $C>0$ such that
\[
 \sup_{\e\in(0,1)}\sup_{p_1,\dots,p_n}\frac{1}{\e^{n(2-s)}}\int_{[0,z]^n}\Big\vert \E\Big[\prod_{j=1}^n \Theta\Big(\BB_\mathfrak{H}\Big(\frac{u_j}{\e},p_{j}\Big)\Big) \Big] \Big\vert du_1\dots du_n \leq  C^n n^{n/2},
\]
and, for the $C_\mathfrak{H}$ of Proposition \ref{LRV},
\[\begin{split}
\lim_{\e\to0}\frac{1}{\e^{n(2-s)}}\int_{\Delta_n(z)}&\E\Big[\prod_{j=1}^n \Theta\Big(\BB_\mathfrak{H}\Big(\frac{u_j}{\e},p_j\Big)\Big) \Big] \varphi_\e(z,u_1,\dots,u_n)du_1\dots du_n  \\
&= \lim_{\e\to0} C_\mathfrak{H}^{n/2}\int_{\Delta_n(z)} \sum_{\mathcal{F}}  \prod_{(\alpha,\beta)\in\mathcal{F}}  \frac{\hat{R}(p_\alpha,p_\beta)}{\vert u_\alpha-u_\beta\vert^\mathfrak{H}}\,\varphi_\e(z,u_1,\dots,u_n)du_1\dots du_n,
\end{split}\]
where $\varphi_\e$ is a uniformly bounded function in $\e$. Here, the sum runs over the pairings $\mathcal{F}$ of $\{1,\dots,n\}$, and the limit $\e\to0$ is uniform with respect to $(p_1,\dots,p_n)$. A pairing over vertices of $ \{1,\dots,n\}$ is a partition of this set made of $n/2$ pairs of couples $(\alpha,\beta)$, for which $\alpha<\beta$ and such that all the elements of  $\{1,\dots,n\}$ appear in only one of the pairs. Note that the number of pairings behaves like $n^{n/2}$, which appears in the estimate above.
\end{prop}

\begin{proof} 
The proof follows some of the ideas of \cite{taqqu}.  For the first result of the proposition, we decompose $\Theta( \lambda^{-1}_n\,\cdot)$ over the Hermite polynomials (with resulting coefficients $\Theta_{n,l}$), and obtain
\[\begin{split}
\E\Big[ \prod_{j=1}^n \Theta\Big(\BB_\mathfrak{H}\Big(\frac{u_j}{\e},p_j\Big)\Big) \Big]&=\E\Big[ \prod_{j=1}^n \Theta\Big( \frac{1}{\lambda_n} \lambda_n \BB_\mathfrak{H}\Big(\frac{u_j}{\e},p_j\Big)\Big) \Big]\\
&=\sum_{\substack{ l_\beta\geq 1\\ \beta\in\{1,\dots,n\}}} \left(\prod_{j=1}^n\frac{\Theta_{n,l_j}}{l_j!} \right)\E\Big[\prod_{j=1}^n H_{l_j}\Big( \lambda_n\BB_\mathfrak{H}\Big(\frac{u_j}{\e^s},p_j\Big)\Big) \Big].
\end{split}\]
Note that we introduce the factor  
\[\lambda_n :=\frac{1}{(n-2)^{1/2}}\]
in order to force the convergence of a series, as will be explicit further. This is a key point of the proof. We want to use now  \cite[Lemma 3.2]{taqqu}, which states that for $n \geq 2$, and a $(X_1,\dots,X_n)$ mean zero Gaussian vector such that
\[\E[X_j^2]=1 \qquad \text{and}\qquad \vert \E[X_jX_l]\vert \leq 1\qquad \forall(j,l)\in\{1,\dots,n\}^2\quad\text{with}\quad j\neq l,\]
we have
\be \label{taq32}
\E\Big[ \prod_{j=1}^n H_{l_j}(X_j)\Big]=\left\{
\begin{array}{l}
\displaystyle\frac{l_1! \cdots l_n!}{2^q (q!)} \sum_{I(l_1,\dots,l_n)} r_{i_1 j_1}r_{i_2 j_2} \cdots r_{i_q j_q}\\
\displaystyle\qquad\textrm{if}\quad l_1+\cdots+l_n=2q \textrm{ and } 0\leq l_1, \dots, l_n \leq q\\
\displaystyle0 \quad\textrm{ otherwise}
\end{array} \right.
\ee
where $r_{ij}= \E [X_i X_j]$, and 
\[\begin{split}
I(l_1,\dots,l_n)=\big\{ (i_1,j_1,\dots,i_q,j_q)&\in \{1,\dots,n\}^{2q},\quad\text{s.t.}\quad i_\beta \neq j_\beta\quad\forall \beta\in\{1,\dots,q\}\\
& \text{and all index}\quad r \in\{1,\dots,n\}\quad\text{appears } l_r\text{ times}\big\}.
\end{split}\]
Above such a Gaussian vector is said to be standard. Nevertheless, because of the factor $\lambda_n$ we cannot apply \eqref{taq32} directly, we first have to make use of the following multiplication theorem \cite{Feldheim}:
\[H_l(\lambda_n u)= \lambda^l_n \sum_{k=0}^{ [l/2]} (1-\lambda_n^{-2})^k  \frac{l!}{2^k(l-2k)! \, k!} H_{l-2k}(u).
\]
Specializing \eqref{taq32} to our case, we find
\[\begin{split}
\E\Big[\prod_{j=1}^n & H_{l_j-2k_j}\Big( \BB_\mathfrak{H}\Big(\frac{u_j}{\e},p_j\Big)\Big)\Big]\\
&=\left\{\begin{array}{l}
\displaystyle\frac{\tilde{l}_1!\cdots \tilde{l}_n!}{2^q q! }\sum_{ I(\tilde{l}_1,\dots,\tilde{l}_n)} \prod_{\beta=1}^q r_\mathfrak{H}\Big(\frac{u_{i_\beta}-u_{j_\beta}}{\e}\Big)\hat{R}(p_{i_\beta},p_{j_\beta})\\
\displaystyle\qquad\text{if}\quad \tilde{l}_1+\dots+\tilde{l}_n=2q  \textrm{ and } 0\leq \tilde{l}_1, \dots, \tilde{l}_n \leq q\quad \text{with}\quad \tilde{l}_j:=l_j-2k_j,\\
\displaystyle0\quad \text{otherwise}.
\end{array}\right.
\end{split}\]
Let us remark that all the indices $l$ are odd since $\Theta$ is assumed to be odd ($\Theta_{n,l}=0$ for $l$ even). Hence, $\tilde{l}_j=l_j-2k_j \geq 1$ for all $j=1,\dots,n$, so that $q \geq n/2$.
Consider now the term
\[
A_{q,n}:=\int_{[0,z]^n} \left|\prod_{m=1}^q r_\mathfrak{H}\Big(\frac{u_{i_m}-u_{j_m}}{\e}\Big)\hat{R}(p_{i_m},p_{j_m}) \right|du_1\dots du_n.
\]
We deduce from the definition of $I(\tilde{l}_1,\dots,\tilde{l}_n)$ that each of the $u_1, \dots, u_n$ appear at least once in the product above. Keeping $n/2$ of them for integrating $r_\mathfrak{H}$, and bounding $r_\mathfrak{H}$ by $\sup_u |r_\mathfrak{H}(u)|=1$ for the others, and using the fact that $r_\mathfrak{H}(u)$ is even, we find
\be \label{Aqn}
A_{q,n} \leq  (2z)^{n/2} \left(\sup_u |r_\mathfrak{H}(u)|\right)^{q-n/2} \Big(\underbrace{\sup_{p_1,p_2} |\hat{R}(p_1,p_2)|}_{\leq 1}\Big)^{q} \left(\int_0^z \Big| r_\mathfrak{H}\Big(\frac{u}{\e}\Big)\Big|du \right)^{n/2}.
\ee
We now need to estimate the cardinal of $I(\tilde{l}_1,\dots,\tilde{l}_n)$. For this, we use again \fref{taq32} with $X_1=\cdots=X_n=X$ where $X\sim \mathcal{N}(0,1)$, and find, with now $r_{i_m j_m}=1$, together with \fref{Aqn},
\[\begin{split}
\int_{[0,z]^n}\Big\vert \E\Big[\prod_{j=1}^n  H_{l_j-2k_j}&\Big( \BB_\mathfrak{H}\Big(\frac{u_j}{\e},p_j\Big)\Big)\Big] \Big\vert du_1\dots du_n \\
&\leq  C^n \left(\int_0^z\Big\vert r_{\mathfrak{H}}\Big(\frac{u}{\e}\Big)\Big\vert du\right)^{n/2} \E\Big[\Big\vert \prod_{j=1}^n H_{l_j-2k_j}(X)\Big\vert \Big].
\end{split}\]
Moreover, we have
\be \label{EH}\E\Big[\Big\vert \prod_{j=1}^n H_{r_j}(X)\Big\vert \Big]\leq \prod_{j=1}^n (n-1)^{r_j/2}\sqrt{r_j!},
\ee
according to \cite[Lemma 3.1]{taqqu}, which yields
\[\begin{split}
\int_{[0,z]^n}\E\Big[\prod_{j=1}^n H_{l_j}\Big( \lambda_n\BB_\mathfrak{H}\Big(\frac{u_j}{\e^s},&p_j\Big)\Big) \Big]du_1\dots du_n\\
&\leq  C^n \left(\int_0^z \Big\vert r_{\mathfrak{H}}\Big(\frac{u}{\e}\Big)\Big\vert du\right)^{n/2} \prod_{j=1}^n  \frac{\lambda_n^{l_j} l_j!}{[l_j/2]!} \\
&\quad \times\sum_{\substack{j=1,\dots,n \\ k_j=0,\dots, [l_j/2]}} \prod_{j=1}^n (n-1)^{l_j/2-k_j}(\lambda_n^{-2}-1)^{k_j}\frac{[l_j/2]!}{2^{k_j}k_j! \sqrt{(l_j-2k_j)!} }.
\end{split}\]
After standard computations, we find for $l_j$ odd,
\[\sqrt{(l_j-2k_j)!} \geq 2^{[l_j/2]-k_j} ([l_j/2]-k_j)!,\qquad \text{and}\qquad (n-1)^{l_j/2-k_j}\leq n^{1/2}(n-1)^{[l_j/2]-k_j},\]
and then, with the binomial theorem,
\[\begin{split}
\sum_{k_j=0}^{ [l_j/2]} & (n-1)^{l_j/2-k_j}(\lambda_n^{-2}-1) ^{k_j}\frac{[l_j/2]!}{2^{k_j}k_j! \sqrt{(l_j-2k_j)!} }\\
&\leq \frac{n^{1/2}}{2^{[l_j/2]}} \sum_{k_j=0}^{ [l_j/2]}  (n-1)^{[l_j/2]-k_j}(\lambda_n^{-2}-1)^{k_j}\frac{[l_j/2]!}{k_j! ([l_j/2]-k_j)! }\\
&\leq \frac{n^{1/2}}{2^{[l_j/2]}}  ( n+\lambda_n^{-2}-2)^{[l_j/2]}.
\end{split}\]
Hence, using again that all the indices $l_j$ are odds, we obtain
\[\begin{split}
\int_{[0,z]^n}\E\Big[ \prod_{j=1}^n &\Theta\Big(\BB_\mathfrak{H}\Big(\frac{u_j}{\e},p_j\Big)\Big) \Big]du_1\dots d u_n\\
&\leq n^{n/2}C^n \left(\int_0^z\Big\vert r_{\mathfrak{H}}\Big(\frac{u}{\e}\Big)\Big\vert du\right)^{n/2} \sum_{\substack{ l_\beta\geq 1\\ \beta\in\{1,\dots,n\}}}\prod_{j=1}^n \frac{\lambda_n^{l_j}\vert\Theta_{n,l_j}\vert}{2^{[l_j/2]}[l_j/2]!} ( n+\lambda_n^{-2}-2)^{[l_j/2]}\\
&\leq (\lambda_n n^{1/2})^nC^n \left(\int_0^z\Big\vert r_{\mathfrak{H}}\Big(\frac{u}{\e}\Big)\Big\vert du\right)^{n/2} \left( \sum_{ l\geq 0} \frac{\vert\Theta_{n,2l+1}\vert}{l!}\right)^n.
\end{split}\]
Now, let us consider
\[\begin{split}
 \sum_{ l\geq 0} \frac{\vert\Theta_{n,2l+1}\vert}{ l!} &= \left(\sum_{ l= 0}^{ [nM]-1}+\sum_{ l= [nM]}^{+\infty}\right) \frac{\vert\Theta_{n,2l+1}\vert}{ l!} \\
 &:=I+II,
\end{split}\]
where $M$ is independent of $n$ and will be specified later. In what follows, we just work with $l\geq1$ since the bound is direct for $l=0$ . For the first term, we perform an integration by parts in $\Theta_{n,2l+1}$ using definition \fref{defHerm}, and obtain
\[
\Theta_{n,2l+1}= \lambda_n^{-1} (-1)^{2l}\int  \Theta^{(1)} (\lambda_n^{-1} u) g^{(2l)}(u)=\lambda_n^{-1}\Theta^{(1)}_{n,2l},
\]
and according to \fref{ortho},
we have
\[\vert \Theta_{n,2l+1}\vert \leq \lambda_n^{-1} \| \Theta^{(1)}(u) \|_{L^2(\Rm,g(u)du)} \|H_{2l}\|_{L^2(\Rm,g(u)du)} \leq \lambda_n^{-1} \sup_u\vert \Theta^{(1)}(u)\vert \sqrt{(2l)!}.\]
As a result, using that $(2l)!\leq 2^{2l} (l!)^2$ we obtain
\[
I \leq C_1+C_2 n^{1/2} \sum_{ l= 1}^{ [nM]-1}  2^l \leq C_1+n^{1/2}C^{nM}. 
\]
For the second term $II$, we have after $2l$ integration by parts,
\[
\Theta_{n,2l+1}= \lambda_n^{-2l} (-1)^{1}\int  \Theta^{(2l)} (\lambda_n^{-1} u) g^{(1)}(u)du,
\]
and therefore, according to \eqref{hyptheta}, using that $l! \geq e (l/e)^l$,
\[\sum_{l\geq [nM]}\frac{\vert\Theta_{n,2l+1}\vert}{ l!} \leq C \sum_{l\geq [nM]} \frac{\lambda^{-2l}_n}{ l!} C^{2l}_{\Theta}\leq   C \sum_{l\geq [nM]} \Big(\frac{n}{l}\Big)^l (e C^2_{\Theta})^l.
\]
Then, setting $M>eC^2_{\Theta}$, we have $ II \leq C$. Hence,
\[ \int_{[0,z]^n}\E\Big[ \prod_{j=1}^n \Theta\Big(\BB_\mathfrak{H}\Big(\frac{u_j}{\e},p_j\Big)\Big) \Big]du_1\dots d u_n\leq n^{n/2}C^n \left(\int_0^z\Big\vert r_{\mathfrak{H}}\Big(\frac{u}{\e}\Big)\Big\vert du\right)^{n/2}.\]
We finally conclude by estimating the term involving $r_\mathfrak{H}$: following \eqref{LRr}, there exists $z_e$ such that for all $z>z_e$, we have $\vert r_\mathfrak{H}(z)\vert \leq C\vert z \vert^{-\mathfrak{H}}$, and therefore, for all $z>z_e$,
\begin{equation}\label{calcH}
\int_0^z \Big\vert r_{\mathfrak{H}}\Big(\frac{u}{\e}\Big)\Big\vert du \leq C\Big(\e+\e^{\mathfrak{H}} \int_{\e z_e}^z \vert u\vert^{-\mathfrak{H}}du\Big)\leq C\e^{2(s-2)},
\end{equation}
with $s=2-\mathfrak{H}/2$.

For the second result of the proposition, we decompose $\Theta$ itself over the Hermite polynomials to obtain
\[\begin{split}
\E\Big[ \prod_{j=1}^n \Theta\Big(\BB_\mathfrak{H}\Big(\frac{u_j}{\e},p_j\Big)\Big) \Big]&=\sum_{\substack{ l_\alpha\geq 1\\ \alpha\in\{1,\dots,n\}}} \left(\prod_{j=1}^n\frac{\Theta_{l_j}}{l_j!} \right)\E\Big[\prod_{j=1}^n H_{l_j}\Big(\BB_\mathfrak{H}\Big(\frac{u_j}{\e},p_j\Big)\Big) \Big]\\
&= \Theta_1^n \sum_{\mathcal{F}} \prod_{(\alpha,\beta)\in\mathcal{F}}  r_{\mathfrak{H}}\Big(\frac{u_\alpha-u_\beta}{\e}\Big)\hat{R}(p_\alpha,p_\beta)+ R^\e_n(u_1,\dots,u_n),
\end{split}\]
with 
\[
R^\e_n(u_1,\dots,u_n)= \sum_{j=1}^n \sum_{S_j(l_1,\dots,l_n)} \left(\prod_{j=1}^n\frac{\Theta_{l_j}}{l_j!} \right)\E\Big[\prod_{j=1}^n H_{l_j}\Big(\BB_\mathfrak{H}\Big(\frac{u_j}{\e},p_j\Big)\Big) \Big].
\]
and
\[
S_j(l_1,\dots,l_n)=\{l_k=1 \textrm{ for } k<j;\quad l_j \in \{2,\dots,n\}, \quad l_k \in \{1,\dots,n\} \textrm{ for } k>j\}.
\]
According to \fref{taq32}, $R^\e_n$ can be recast as
\[
R^\e_n(u_1,\dots,u_n)=\sum_{j=1}^n\sum_{q\geq n/2+1}\sum_{\tilde{S}_{j,q}(l_1,\dots,l_n)}\prod_{m=1}^n \left(\frac{\Theta_{l_m}}{l_m!}\right)\E\Big[\prod_{m=1}^n H_{l_m}\big(\BB_\mathfrak{H}\big(\frac{u_m}{\e},p_m\big)\big)\Big],
\]
where $\tilde{S}_{j,q}(l_1,\dots,l_n)=S_j(l_1,\dots,l_n) \cap \{l_1+\cdots+l_n=2q\}$. Let us emphasize the fact that $q\geq n/2+1$ since there is at least one index $l_j$ greater than $2$ and $n$ is even. This is what will allow us to gain some extra powers of $\eps$  to obtain the convergence to the leading term. We need to estimate for this the term $A_{q,n}$ for $q\geq n/2+1$ in the same way as before. Since $r_\mathfrak{H}$ and $\hat{R}$ are bounded by one, we directly find, for all $ (i_1,j_1,\dots,i_q,j_q)\in I(l_1,\dots,l_n)$,
\[\int_{[0,z]^n} \prod_{m=1}^q \big\vert r_\mathfrak{H}\big(\frac{u_{i_m}-u_{j_m}}{\e}\big)\hat{R}(p_{i_m},p_{j_m})\big\vert du_1\dots du_n \leq \int_{[0,z]^n} \prod_{m=1}^{n/2+1}  \big\vert r_\mathfrak{H}\big(\frac{u_{i'_m}-u_{j'_m}}{\e}\big)\big\vert  du_1\dots du_n,\]
where $(i'_1,j'_1,\dots,i'_{n/2+1},j'_{n/2+1})$ repeat $j$ twice. Since $n/2+1$ is odd, only one other index, denoted by $j'$, appears twice. In that context, two cases are possible. In the first case, we have a term of the form  $r^2_\mathfrak{H}((u_{j}-u_{j'})/\e)$ (if any there is only one), and 
\[
\int_0^z du_j\int_0^z du_{j'} r^2_\mathfrak{H}\big(\frac{u_{j}-u_{j'}}{\e}\big)\leq \left\{\begin{array}{ccl}
C_1\e^{2\mathfrak{H}} &\text{if}& \mathfrak{H}\in(0,1/2),\\
C'_1\e \log(1/\e)&\text{if}& \mathfrak{H}=1/2,\\
C''_1\e &\text{if}& \mathfrak{H}\in(1/2,1).
\end{array}\right.
\]
Using then \eqref{calcH},
\[\begin{split}
\int_{[0,z]^n} \prod_{m=1}^{n/2+1}  \big\vert r_\mathfrak{H}\big(\frac{u_{i'_m}-u_{j'_m}}{\e}\big)\big\vert  du_1\dots du_n&=\left(\int_0^z du\int_0^z dv \big\vert r_\mathfrak{H}\big(\frac{u-v}{\e}\big)\big\vert  dudv\right)^{n/2-1}\\
&\qquad \times \int_0^z du\int_0^z dv \; r^2_\mathfrak{H}\big(\frac{u-v}{\e}\big)  dudv\\
&\leq C  \e^{\mathfrak{H}n/2}\e^{\mathfrak{H}\wedge(1-\mathfrak{H})}\log(1/\e).
\end{split}\]
If we are not in the first case, we have a term of the form $r_\mathfrak{H}((u_{j}-u_{j'})/\e)r_\mathfrak{H}((u_{j}-u_{k})/\e)$, $k\neq j'$. Using then the Cauchy-Schwarz' inequality with respect to $u_j$, a change of variable, the fact that $r_\mathfrak{H}$ is even, and again \eqref{calcH} leads to
\[\begin{split}
\int_{[0,z]^n} \prod_{m=1}^{n/2+1}  \big\vert r_\mathfrak{H}\big(\frac{u_{i'_m}-u_{j'_m}}{\e}\big)\big\vert  du_1\dots du_n&\leq C \left(\int_0^z du\int_0^z dv \big\vert r_\mathfrak{H}\big(\frac{u-v}{\e}\big)\big\vert  dudv\right)^{n/2-1}\\
&\qquad \times \int_0^{2z} du  \; r^2_\mathfrak{H}\left(\frac{u}{\e}\right)\\
&\leq C  \e^{\mathfrak{H}n/2}\e^{\mathfrak{H}\wedge(1-\mathfrak{H})}\log(1/\e).
\end{split}\]
As a result, bounding the cardinal of $I(l_1,\dots,l_n)$ in the same way as before, we obtain
\[\begin{split}
\int_{[0,z]^n}\sup_{p_1,\dots,p_n}\Big\vert \E\Big[\prod_{m=1}^n &H_{l_m}\big(\BB_\mathfrak{H}\big(\frac{u_m}{\e},p_m\big)\big)\Big]\Big\vert  du_1\dots du_n\\
&\leq  C  \e^{\mathfrak{H}n/2}\e^{\mathfrak{H}\wedge(1-\mathfrak{H})} \log(1/\e) \Big\vert \E\Big[\prod_{m=1}^n H_{l_m}(X)\Big] \Big\vert\\ 
\end{split}\]
and therefore, using \fref{EH},
\[\begin{split} 
\frac{1}{\e^{n(2-s)}} \int_{[0,z]^n} &\sup_{p_1,\dots,p_n} \Big\vert \E\Big[ R^\e_n(u_1,\dots,u_n) \Big] \Big\vert du_1\dots du_n\\
&\leq C  \e^{\mathfrak{H}\wedge(1-\mathfrak{H})}\log(1/\e) \sum_{j=1}^n\sum_{q \geq n/2+1}\sum_{\tilde S_{j,q}(l_1,\cdots,l_n)}\prod_{m=1}^n \frac{\vert \Theta_{l_m}\vert(n-1)^{l_m}}{\sqrt{l_m!}}\\
&\leq   \e^{\mathfrak{H}\wedge(1-\mathfrak{H})}\log(1/\e) n C^n \left(\sum_{l \geq 1}\frac{\vert \Theta_{l}\vert(n-1)^{l}}{\sqrt{l!}} \right)^n.
\end{split} \]
According to \eqref{hyptheta}, we have 
\[\vert \Theta_l\vert =\Big\vert \int \Theta^{(l)}(u)g(u)du\Big\vert \leq C C^l_{\tilde{\Theta}},\]
so that the sum above is finite, and which shows that for $n$ fixed, the error term $R_n^\e$ converges to zero as $\eps \to 0$. It remains to treat the leading term. For this, we write
\[\begin{split}
 \sum_{\mathcal{F}}  \frac{1}{\e^{n(2-s)}}\int_{\Delta_n(z)}  \prod_{(\alpha,\beta)\in\mathcal{F}} & \Big\vert    r_{\mathfrak{H}}\Big(\frac{u_\alpha-u_\beta}{\e}\Big) - \frac{\e^{\mathfrak{H}}c_\mathfrak{H}}{\vert u_\alpha-u_\beta\vert ^\mathfrak{H}} \Big\vert du_1\dots du_n\\
 & \leq  \frac{(n-1)!!}{n! } \Big[ \frac{1}{\e^{2(2-s)} }  \int_0^z\int_0^z\Big\vert   r_{\mathfrak{H}}\Big(\frac{u-v}{\e}\Big) - \frac{\e^{\mathfrak{H}}c_\mathfrak{H}}{\vert u-v\vert ^\mathfrak{H}} \Big\vert dudv \Big]^{n/2},
\end{split}\]
where $(n-1)!!=n!/(2^{n/2}(n/2)!)$ is the number of pairings of $\{1,\dots,n\}$. According to \eqref{LRr}, for any $\eta>0$ and $z_e$ such that $z>z_e$, we have $\vert r_\mathfrak{H}(z)-c_\mathfrak{H}\vert z\vert^{-\mathfrak{H}}\vert \leq \eta c_\mathfrak{H}\vert z \vert^{-\mathfrak{H}}$, and as a result,
\[\begin{split}
  \frac{1}{\e^{2(s-2)} } \int_0^z\int_0^z\Big\vert   r_{\mathfrak{H}}\Big(\frac{u-v}{\e}\Big) - \frac{\e^{\mathfrak{H}}c_\mathfrak{H}}{\vert u-v\vert ^\mathfrak{H}} \Big\vert dudv & \leq \eta c_{\mathfrak{H}} \int_{\vert u -v \vert > \e z_e } \vert u-v\vert ^{-\mathfrak{H} }du dv \\
 &\qquad +\e \int_{\vert u -v \vert \leq z_a } r_\mathfrak{H} (u-v)dudv\\
 &\qquad +c_{\mathfrak{H}} \int_{\vert u-v \vert \leq \e z_a } \vert u-v\vert^{-\mathfrak{H}}dudv,
\end{split}\]
This finally yields, for all $\eta>0$, 
\[\overline{\lim_{\e\to0}}    \frac{1}{\e^{2(s-2)} } \int_0^z\int_0^z\Big\vert   r_{\mathfrak{H}}\Big(\frac{u-v}{\e}\Big) - \frac{\e^{\mathfrak{H}}c_\mathfrak{H}}{\vert u-v\vert ^\mathfrak{H}} \Big\vert dudv \leq  \eta c_{\mathfrak{H}} \int_0^z\int_0^z \vert u-v\vert ^{-\mathfrak{H} }du dv.\] 
The proof of the proposition is complete.
\end{proof}

\section{Proof of Proposition \ref{evaprop}}\label{secevaprop}

Let $\phi=(\phi_1,\phi_2) \in  \calC^\infty_0(\Rd)\times  \calC^\infty_0(\Rd)$ be a test function such that $supp(\phi) \subset B(0,r_\phi)$. In order to control the transition between propagative and evanescent modes, we introduce the following integer 
\begin{align*}
&n_\e:=\inf\Big( n\geq 0 \quad \text{s.t.} \quad \{Q_n \in \calQ^n_{\omega,\e} \quad \text{s.t.} \quad 1<\e \vert Q_n\vert^2/k^2_\omega  \}\neq \emptyset \Big)
\end{align*}
where $\calQ^n_{\omega,\e}$ is the set
\[\calQ^n_{\omega,\e}=\{Q_n=\kappa-q_1-\dots-q_n \quad \textrm{s.t.} \quad (\kappa,q_1,\dots,q_n)\in B(0,r_\phi)\times \calS^n_{\omega,\e}(\kappa) \}.
\]
Above, $\calS^n_{\omega,\e}(\kappa)$ is defined by
\begin{equation}\label{defS}
\calS^n_{\omega,\e}(\kappa):=\Big\{(q_1,\dots,q_n)\in \calS^n \quad \text{s.t.}\quad\vert Q_l\vert<\vert k_\omega \vert/\e\quad \forall l\in\{1,\dots,n\}\Big\}.
\end{equation}
Let us remark that with this definition, we have, for all $ \e \leq \e_0$,
\be \label{low}n_\e>\frac{1}{r_\calS}\Big( \frac{\vert k_\omega \vert}{ \sqrt{\e}}-r_\phi \Big)\qquad\text{and}\qquad \vert \lambda_{\e,\omega}(Q_n) \vert >\eta:=\sqrt{1-\e_0},
\ee
for all $Q_n \in \calQ^n_{\omega,\e}$ with $n< n_\e$. The integer $n_\eps$ measures the number of iterations it takes for the momentum $Q_n$ to be at least of order $1/\sqrt{\e}$. Note that this order is arbitrary, any order of the form $\eps^{-\alpha}$, $\alpha>0$ would work just fine.  Integrating then \eqref{cmode} in $z$ and iterating $n_\e-1$ times this relation, we obtain
\[\begin{split}
\begin{bmatrix} \hat{a}^\e_\omega(z,\kappa)\\ \hat{b}^\e_\omega(z,\kappa) \end{bmatrix}&=\sum_{n=0}^{n_\e-1} \int_{\Delta_n(z)}d\mathbf{u}^{(n)} \int_{\calS^n_{\omega,\e}(\kappa)}\mathbf{m}(d\mathbf{q}^{(n)}) \prod_{j=1}^n H^{\e}_\omega(u_j,Q_{j-1},Q_j)   \begin{bmatrix} \hat{a}^\e_\omega(0,Q_n)\\ \hat{b}^\e_\omega(0,Q_n) \end{bmatrix} \\
&\qquad +\int_{\Delta_{n_\e}(z)}d\mathbf{u}^{(n_\e)} \int_{\calS^{n_\e}_{\omega,\e}(\kappa)}\mathbf{m}(d\mathbf{q}^{(n_\e)}) \prod_{j=1}^{n_\e} H^{\e}_\omega(u_j,Q_{j-1},Q_j)   \begin{bmatrix} \hat{a}^\e_\omega(u_{n_\e},Q_{n_\e})\\ \hat{b}^\e_\omega(u_{n_\e},Q_{n_\e}) \end{bmatrix}\\
&:=I^\e(z,\kappa)+J^\e(z,\kappa),
\end{split}\]
where $H^\e_\omega$ and $\Delta_n(z)$ are defined by \eqref{defH} and \fref{simplexe}, $\mathbf{m}(d\mathbf{q}^{(n)}):=m(dq_1)\dots m(dq_n)$, and $Q_j:=\kappa-q_1-\dots-q_j$, and $Q_0:=\kappa$, . Let us point out the important fact that iterations are stopped before evanescent modes appear in the series (that is before the second term in the r.h.s of \fref{cmode} enters the expansion), and this gives rise to only the  two terms $I^\e(z,\kappa)$ and $J^\e(z,\kappa)$. The first one is the driving term, and the second one is a reminder. The latter will be shown to be negligible, the main argument being that it is essentially of order $1/n_\eps !$ after integration over the simplex $\Delta_{n_\eps}(z)$. Writing a complete series expansion, i.e. choosing $n_\eps=+\infty$ (which then sets $J^\e$ to zero), requires to handle the transition propagative-evanescent modes which is more difficult than treating the reminder $J^\e$. Note also that in Proposition \ref{evaprop}, we are only interested in the convergence of $\hat{a}^\e_\omega(z,\kappa)$ and $\hat{b}^\e_\omega(z,\kappa)$ for $\kappa$ in the support of $\phi$, and we can therefore only consider the above equation for $|\kappa|<\vert k_\omega \vert/\eps$. Thus, we set $I^\e(z,\kappa)=J^\e(z,\kappa)=0$ for $\vert \kappa \vert >\vert k_\omega\vert /\e$. We then have the following two lemmas, whose proofs are postponed to the end of the section.
\begin{lem}\label{error} For all $z\in[0,L]$, we have, for all $\mu>0$,
\[\lim_{\e\to 0}\Pro\Big( \big\vert \big<J^\e(z),\phi \big>_{L^2(\Rd)\times L^2(\Rd)} \big\vert>\mu\Big)=0.\]
\end{lem}
This first lemma shows that $J^\e$ gives a negligible contribution, and therefore that $I^\e$ is the leading term. In the second lemma below, we introduce an auxiliary process $\tilde{I}^\e$ that approximates $I^\e$ by letting the absorption $\alpha_\eps$ vanish in the complex exponentials and by replacing the $\lambda_{\e,\omega}$ in the denominator by one.
\begin{lem}\label{terme1}
For all $z\in[0,L]$, we have, for all $\mu>0$
\[\lim_{\e\to 0}\Pro\big( \big\vert \big<I^\e(z)-\tilde{I}^\e(z),\phi \big>_{L^2(\Rd)\times L^2(\Rd)} \big\vert>\mu\big)=0,\]
where
\[\tilde{I}^\e(z,\kappa):=\sum_{n=0}^{n_\e-1} \int_{\Delta_n(z)}d\mathbf{u}^{(n)} \int_{\calS^n_{\omega,\e}(\kappa)}\mathbf{m}(d\mathbf{q}^{(n)}) \prod_{j=1}^n \tilde{H}^{\e}_\omega(u_j,Q_{j-1},Q_j)   \begin{bmatrix} \hat{a}^\e_\omega(0,Q_n)\\ \hat{b}^\e_\omega(0,Q_n) \end{bmatrix},\]
where $\tilde{H}^\e_\omega$ is defined by \eqref{defHt}, and $\tilde{I}^\e(z,\kappa)=0$ for $\vert \kappa \vert >\vert k_\omega\vert /\e$.
\end{lem}
Now, with the notation
\[L^\e(z,\kappa):=\begin{bmatrix} \hat{A}^\e_\omega(z,\kappa)\\ \hat{B}^\e_\omega(z,\kappa) \end{bmatrix} \qquad
\textrm{with} \qquad L^\e(z,\kappa)=0 \quad \textrm{for} \quad \vert \kappa \vert >\vert k_\omega\vert /\e, \]
we have
\[\begin{split}
 \Pro\Big( \Big\vert \Big<\begin{bmatrix} \hat{a}^\e_\omega(z)\\ \hat{b}^\e_\omega(z) \end{bmatrix}&-\begin{bmatrix} \hat{A}^\e_\omega(z)\\ \hat{B}^\e_\omega(z) \end{bmatrix} ,\phi \Big>_{L^2(\Rd)\times L^2(\Rd)} \Big\vert>\mu\Big) \\
&\leq\Pro\big( \big\vert \big<I^\e(z)-\tilde{I}^\e(z),\phi \big>_{L^2(\Rd)\times L^2(\Rd)} \big\vert>\mu/3\big)\\
 &\qquad+\Pro\big( \big\vert \big<\tilde{I}^\e(z)-L^\e(z),\phi \big>_{L^2(\Rd)\times L^2(\Rd)} \big\vert>\mu/3\big) +\Pro\Big( \big\vert \big<J^\e(z),\phi \big>_{L^2(\Rd)\times L^2(\Rd)} \big\vert>\mu/3\Big)
\end{split}\]
and owing Lemmas \ref{error} and \ref{terme1}, it just remains to prove that, $\forall \mu'>0$,
\be \label{IL}\lim_{\e\to 0}\Pro\big( \big\vert \big<\tilde{I}^\e(z)-L^\e(z),\phi \big>_{L^2(\Rd)\times L^2(\Rd)} \big\vert>\mu'\big)=0.\ee
This follows from the calculation below. After straightforward algebra and the change of variable $\kappa\to \kappa+q_1+\dots+q_n$, we have
\[ \big<\tilde{I}^\e(z)-L^\e(z),\phi \big>_{L^2(\Rd)\times L^2(\Rd)} = \sum_{n\geq n_\e} \Big<\begin{bmatrix} \hat{a}^\e_\omega(0)\\ \hat{b}^\e_\omega(0) \end{bmatrix} , \tilde{R}_n^\e\Big>_{L^2(\Rd)\times L^2(\Rd)},\]
where $\tilde{R}_n^\e (\kappa')=0$ for $\vert \kappa'\vert >\vert k_\omega\vert /\e$, and is given by, for $\vert \kappa'\vert <\vert k_\omega\vert /\e$,
\[\tilde{R}_n^\e (\kappa'):= \int_{\Delta_n(z)}d\mathbf{u}^{(n)} \int_{\tilde{\calS}^n_{\omega,\e}(\kappa')}\mathbf{m}(d\mathbf{q}^{(n)}) \prod_{j=1}^{n} \tilde{H}^{\e}_\omega(u_j,\tilde{Q}_{j-1},\tilde{Q}_j) \begin{bmatrix} \overline{\phi_1}(\tilde{Q}_0)\\ \overline{\phi_2}(\tilde{Q}_0)\end{bmatrix},\]
where $\tilde{Q}_j=\kappa'+q_n+\dots+q_{j+1}$, and
\[\tilde{\calS}^n_{\omega,\e}(\kappa'):=\Big\{(q_1,\dots,q_n)\in \calS^n \quad \text{s.t.}\quad\vert \tilde{Q}_l\vert<\vert k_\omega \vert/\e,\quad \forall l\in\{1,\dots,n\}\Big\}.\]
We need now to bound $R_n^\eps$. For this, the following estimate, which is a consequence of the first result of Proposition \ref{LRM}, will be used several times in the course of the proof of the proposition:
\be \label{defTV} 
T_{n,\eps}:=\frac{1}{\eps^{2n(2-s)}} \sup_{\mathbf{q}_1^{(n)},\mathbf{q}_2^{(n)}} \int_{[0,L]^{2n}} d\mathbf{u}^{(2n)} \vert \E\Big[\prod_{j=1}^n \hat{V}\Big(\frac{u_{1,j}}{\e},q_{1,j}\Big)\hat{V}\Big(\frac{u_{2,j}}{\e},q_{2,j}\Big) \Big] \vert \leq (n^{1/2}C)^{2n}.
\ee
Since $\tilde Q_0 \in supp(\phi)$, the support of $R_n^\e$ is included in $K_{n}=\{\vert \kappa' \vert \leq nr_\calS +r_{\phi}\}$, and then, according to \fref{fact},
\bee
\E\left[ \| \tilde{R}_n^\e \|^2_{L^2(K_n)\times L^2(K_n)}\right ] &\leq&  C^{2n} \E [|m(\calS)|^{2n}]\;|K_n|\; (\sup_{\substack{j=1,2\\ u\in\Rd}}\vert \phi_j(u) \vert)^2 \; T_{n,\e}/(n!)^2\\
&\leq& \frac{(n^{1/2}C)^{2n}}{(n!)^2}(nr_\calS+r_\phi)^2.
\eee
Using finally \fref{expa} and estimate \fref{E2} in order to bound $ \hat{a}^\e_\omega(0)$ and $ \hat{b}^\e_\omega(0)$, as well as the Markov' and Cauchy-Schwarz' inequalities, we find, 
\[\begin{split}
\Pro\big( \big\vert \big<\tilde{I}^\e(z)-&L^\e(z),\phi \big>_{L^2(\Rd)\times L^2(\Rd)} \big\vert>\mu' \big)\\
&\leq C \frac{ \|\hat{f}_0(\omega,\cdot)\|_{L^2(\Rd)}}{\mu'}\sum_{n\geq n_\e}\left(\E\left[ \| \tilde{R}_n^\e \|^2_{L^2(K_n)\times L^2(K_n)}\right ]\right)^{1/2}\\
&\leq \sum_{n\geq n_\e} \frac{C^{n}}{\sqrt{n!}} (nr_\calS+r_\phi),
\end{split}\]
which concludes the proof of Proposition \ref{evaprop}. We end this section with the proofs of Lemmas \ref{error} and \ref{terme1}.

\begin{proof}[Proof of Lemma \ref{error}] The proof is very similar to the one above, and we only detail the differences. The main ingredient is the fact that $J_\e$ is proportional to $1/n_\eps !$ after integration. After the change of variable $\kappa\to \kappa+q_1+\dots+q_n$, we have 
\[ \begin{split}
\big<J^\e(z),\phi \big>_{L^2(\Rd)\times L^2(\Rd)}&= \int_0^z du_{n_\e}\Big<\begin{bmatrix} \hat{a}^\e_\omega(u_{n_\e})\\ \hat{b}^\e_\omega(u_{n_\e}) \end{bmatrix} , \tilde{H}_{n_\e}(u_{n_\e})\Big>_{L^2(\Rd)\times L^2(\Rd)}
\end{split}\]
with $\tilde{H}_n(u_n,\kappa')=0$ for $\vert\kappa'\vert>\vert k_\omega \vert /\e$, and for $\vert\kappa'\vert<\vert k_\omega \vert /\e$,
\[\begin{split}
\tilde{H}_n(u_n,\kappa'):=& \int_{\tilde{\Delta}_n(z,u_n)} d\mathbf{u}^{(n-1)} \int_{\tilde{\calS}^n_{\omega,\e}(\kappa')}\mathbf{m}(d \mathbf{q}^{(n)} ) \prod_{j=1}^{n} H^{\e}_\omega(u_j,\tilde{Q}_{j-1},\tilde{Q}_j)\begin{bmatrix}\overline{\phi_1}(\tilde{Q}_0) \\ \overline{\phi_2}(\tilde{Q}_0) \end{bmatrix}.
\end{split}\]
Here, $\tilde{Q}_j$ is as before, and
\[\tilde{\Delta}_n(z,u_n):=\big\{(u_1,\dots,u_{n-1})\in [0,z]^{n-1},\quad\text{s.t.}\quad u_{j+1}\leq u_{j}\quad\forall j\in\{1,\dots,n-1\} \big\}.\]
As a result, using the Cauchy-Schwarz' inequality, \fref{defTV} with $n=n_\e-1$, as well as estimate \fref{E1}, we find    
\[\begin{split}
\E\Big[\vert \big<J^\e(z),\phi \big>_{L^2(\Rd)\times L^2(\Rd)}\vert \Big]&\leq \frac{\e^{(s-2)}C^{n_\e-1} }{\sqrt{(n_\e-1)!}} \E\left[\left(\int_0^L du \int_{\{\vert \kappa \vert <\vert k_\omega \vert /\e\}}d\kappa(\vert  \hat{a}^\e_\omega(u,\kappa) \vert^2+\vert  \hat{b}^\e_\omega(u,\kappa) \vert^2)\right)^{1/2}\right]\\
&\leq \frac{\e^{(s-2)-2}}{\alpha_\eps^{3/4} \sqrt{(n_\e-1)!}} C^{n_\e}.
\end{split}\]
According to the bound from below for $n_\e$ given in \fref{low}, and the fact that $n!>C n^{n+1/2} e^{-n}$,
we find
\[
\frac{\e^{s-4}}{\alpha_\eps^{3/4}\sqrt{(n_\e-1)!}} \leq \frac{C e^{n_\e}}{\alpha_\eps^{3/4}} e^{(8+1/2-2s-n_\e/2)\log(n_\e)}.
\]
This concludes the proof of the lemma.
\end{proof}

\begin{proof}[Proof of Lemma \ref{terme1}] The proof simply consists in sending the absorption $\alpha_\e$ to zero and using the fact that $\eps^2 |Q_n|^2/k_\omega^2 \to 0$ as $\eps \to 0$ for $n<n_\eps$. Let us for this introduce, for $l\in\{0,\dots,n\}$,
\begin{align*}&\tilde{I}^\e_{l,n}(z,\kappa):= \int_{\Delta_n(z)}d\mathbf{u}^{(n)} \int_{\calS^n_{\omega,\e}(\kappa)}\mathbf{m}(d\mathbf{q}^{(n)}) \\
& \hspace{3cm}\times \prod_{j=1}^l \tilde{H}^{\e}_\omega(u_j,Q_{j-1},Q_j)  \prod_{j=l+1}^n H^{\e}_\omega(u_j,Q_{j-1},Q_j)   \begin{bmatrix} \hat{a}^\e_\omega(0,Q_n)\\ \hat{b}^\e_\omega(0,Q_n) \end{bmatrix},
\end{align*}
so that
\[I^\e(z,\kappa)-\tilde{I}^\e(z,\kappa)=\sum_{n=1}^{n_\e-1} \sum_{l=0}^{n-1} \left(\tilde{I}^\e_{l,n}(z,\kappa)-\tilde{I}^\e_{l+1,n}(z,\kappa)\right).\]
After the usual change of variable $\kappa\to \kappa+q_1+\dots+q_n$, we have
\[\big< \tilde{I}^\e_{l,n}(z)-\tilde{I}^\e_{l+1,n}(z),\phi\big>_{L^2(\Rd)\times L^2(\Rd)}=\Big<\begin{bmatrix} \hat{a}^\e_\omega(0)\\ \hat{b}^\e_\omega(0) \end{bmatrix} , \mathbf{I}^\e_{l,n}\Big>_{L^2(\Rd)\times L^2(\Rd)},\]
with $\mathbf{I}^\e_{l,n}(\kappa')=0$ for $\vert\kappa'\vert>\vert k_\omega \vert /\e$, and for $\vert\kappa'\vert<\vert k_\omega \vert /\e$,
\[\begin{split}  
\mathbf{I}^\e_{l,n}(\kappa')&= \int_{\Delta_n(z)}d\mathbf{u}^{(n)} \int_{\tilde{\calS}^n_{\omega,\e}(\kappa')}\mathbf{m}(d\mathbf{q}^{(n)}) \prod_{j=1}^{l} \tilde{H}^{\e}_\omega(u_j,\tilde{Q}_{j-1},\tilde{Q}_j) \\
 &\hspace{1cm}\times [H^{\e}_\omega(u_{l+1},\tilde{Q}_{l},\tilde{Q}_{l+1})-\tilde{H}^{\e}_\omega(u_{l+1},\tilde{Q}_{l},\tilde{Q}_{l+1})] \prod_{j=l+2}^n H^{\e}_\omega(u_j,\tilde{Q}_{j-1},\tilde{Q}_j)   \begin{bmatrix} \overline{\phi_1}(\tilde{Q}_0)\\ \overline{\phi_2}(\tilde{Q}_0)\end{bmatrix},
\end{split}\]
where $\tilde{Q}_j=\kappa'+q_n+\dots+q_{j+1}$. Now, we deduce from \fref{expsqrt} and the fact that the square root is of H\"older regularity $1/2$, that for all $|q| < \vert k_\omega\vert /\e$,
\be \label{estre}
|Re(\lambda_{\e,\omega}(q))-\lambda^r_{\e,\omega}(q)|^2 \leq \frac{1}{2} \left( \sqrt{(1-\eps^2 |q|^2/k_\omega^2)^2+\alpha_{\omega,\e}^2}-(1-\eps^2 |q|^2/k_\omega^2)\right) \leq \frac{1}{2} \alpha_{\omega,\eps}.
\ee
With \fref{estre}, \fref{propsqrt}, the definition of $n_\e$ and \fref{low}, we then find, for $l<n_\eps$,
\[\begin{split}
\left| \frac{1}{\lambda_{\e,\omega}(\tilde{Q}_l)}-1\right|&\leq  \left|\frac{\lambda^r_{\e,\omega}(\tilde{Q}_l)-\lambda_{\e,\omega}(\tilde{Q}_l)}{\lambda_{\e,\omega}(\tilde{Q}_l)\lambda^r_{\e,\omega}(\tilde{Q}_l)}\right|+\left|\frac{\lambda^r_{\e,\omega}(\tilde{Q}_l)-1}{\lambda^r_{\e,\omega}(\tilde{Q}_l)}\right| \leq \frac{C \alpha^{1/2}_{\e,\omega}}{\eta^2}+\frac{C \sqrt{\e}}{\eta}:= \gamma_\eps,
\end{split}\]
and
\[\begin{split} 
\left| e^{ik_\omega(\lambda_{\e,\omega}(\tilde{Q}_l)-\lambda_{\e,\omega}(\tilde{Q}_{l-1})u_l/\e^2}\right.& -\left. e^{i(\lambda^r_{\e,\omega}(\tilde{Q}_l)- \lambda^r_{\e,\omega}(\tilde{Q}_{l-1}))u_l/(2k_\omega)}\right| \\
&= \left|e^{ik_\omega(\lambda_{\e,\omega}(\tilde{Q}_l)-\lambda^r_{\e,\omega}(\tilde{Q}_l)-(\lambda_{\e,\omega}(\tilde{Q}_{l-1})-\lambda^r_{\e,\omega}(\tilde{Q}_{l-1})))u_l/\e^2} -1 \right| \\
&\leq  C \left|e^{ik_\omega(Re(\lambda_{\e,\omega}(\tilde{Q}_l))-\lambda^r_{\e,\omega}(\tilde{Q}_l)-(Re(\lambda_{\e,\omega}(\tilde{Q}_{l-1}))-\lambda^r_{\e,\omega}(\tilde{Q}_{l-1})))u_l/\e^2} -1\right| \\
&\qquad+  \left| e^{-k_\omega Im(\lambda_{\e,\omega}(\tilde{Q}_l)-\lambda_{\e,\omega}(\tilde{Q}_{l-1}))u_l/\e^2}-1\right|\\
&\leq C \alpha^{1/2}_\e.
\end{split}\]
The rest of the proof is now  classical and follows from the same techniques as \fref{IL}: we first apply \fref{defTV} to obtain
\bee
\E\left[ \| \mathbf{I}^\e_{l,n} \|^2_{L^2(K_n)\times L^2(K_n)}\right ] &\leq& ( \gamma^2_\eps+\alpha_\e) \; C^{2n} \E [\vert m\vert (\calS)^{2n}]\; |K_n| \;  (\sup_{\substack{j=1,2\\ u\in\Rd}}\vert \phi_j(u) \vert)^2 \; T_{n,\eps} / ((n!)^2 \; \eta^{2(n-l-1)})\\
&\leq& \frac{(\gamma^2_\eps+\alpha_\e) C^{2n}}{n! }(nr_\calS+r_\phi)^2.
\eee
and then use \fref{expa} and estimate \fref{E2} in order to bound $ \hat{a}^\e_\omega(0)$ and $ \hat{b}^\e_\omega(0)$, as well as the Markov' and Cauchy-Schwarz' inequalities to arrive at
\[\begin{split}
 \Pro\big( \big\vert \big<I^\e(z)&-\tilde{I}^\e(z),\phi \big>_{L^2(\Rd)\times L^2(\Rd)} \big\vert>\mu\big)\\
 &\leq \frac{\gamma_\e C}{\mu} \sum_{n=1}^{n_\e-1} \sum_{l=0}^{n-1} \left(\E\left[\|\mathbf{I}^\e_{l,n}\|^2_{L^2(K_{n})\times L^2(K_{n})} \right]\right)^{1/2}\\
 &\leq  \frac{\gamma_\e}{\mu} \sum_{n=1}^{n_\e-1}\frac{ C^{n}}{\sqrt{n!}} n(nr_\calS +r_\phi)
\end{split}\]
This concludes the proof.
\end{proof}

\section{Proof of Proposition \ref{noback}}\label{proofnoback}

The proof is based on iteration techniques. We only prove the second point of the proposition since it is the most interesting one and the first point follows from similar calculations. We start by writing $\calP^{a,\e}_{\omega,\phi^\e_\omega}$ as an infinite series obtained by iterating \eqref{schroeq2}, that is
\begin{equation}\label{decP}
\calP^{a,\e}_{\omega,\phi^\e_\omega}(z,\kappa)=\sum_{n\geq 0}\calT^{n,\e}_\omega(z,\kappa),\qquad\text{with}\qquad\calT^{0,\e}_\omega(z,\kappa):=\phi^\e_\omega(\kappa),
\end{equation}
where,  for $n\geq 1$,
\[\begin{split}
\calT^{\e,n}_\omega(z,\kappa):= \int_{\Delta_n(z)} d\mathbf{u}^{(n)}  \int_{\calS^n_{\omega,\e}(\kappa)}\mathbf{m}(d\mathbf{q}^{(n)}) \mathbf{H}^\e_{\omega,1,1}\big(\mathbf{u}^{(n)},\mathbf{q}^{(n)}\big)\phi^\e_\omega(Q_n),
\end{split}\]
and
\[\mathbf{H}^\e_\omega(\mathbf{u}^{(n)},\mathbf{q}^{(n)}):=\prod_{j=1}^n \tilde{H}^{\e}_\omega(u_j,Q_{j-1},Q_j).\]
Here, $\calS^n_{\omega ,\e}(\kappa)$ is defined by \eqref{defS}, $\tilde{H}^{\e}_\omega$ by \eqref{defHt}, and $\mathbf{H}^\e_{\omega,1,1}$ denotes the $(1,1)$ entry of the matrix $\mathbf{H}^\e_{\omega}$. As we will see, the leading term in $\calT^{n,\e}_\omega$ is the one obtained by the product of the diagonal elements of the matrices $\tilde{H}^{\e}_\omega(u_j,Q_{j-1},Q_j)$. Any other term involving an off diagonal component introduces an oscillatory integral leading to a vanishing limit as $\eps \to 0$. This will be proved further.  The leading term is therefore
\begin{equation}\label{transferQ}
\calX^\e_\omega(z,\kappa):= \sum_{n\geq 0}\calX^{\e,n}_\omega(z,\kappa),\qquad \calX^{0,\e}_\omega(z,\kappa):=\phi^\e_\omega(\kappa),
\end{equation}
where, for $n\geq 1$,
\be \label{defXen}
\calX^{\e,n}_\omega(z,\kappa):=\Big(\frac{ik_\omega}{\e^{2-s}}\Big)^n \int_{\Delta_n(z)} d\mathbf{u}^{(n)} \int_{\calS^n_{\omega ,\e}(\kappa)}\mathbf{m}(d\mathbf{q}^{(n)} ) \Big(\prod_{j=1}^n \hat{V}(u_j/\e,q_j) \Big)e^{iG^\e_n(\mathbf{u}^{(n)},\mathbf{q}^{(n)})}\phi^\e_\omega(Q_n),
\ee
with
\[G^\e_n(\mathbf{u}^{(n)},\mathbf{q}^{(n)}):=\frac{k_\omega}{\e^2}\sum_{j=1}^n (\lambda^r_{\e,\omega}(Q_j)-\lambda^r_{\e,\omega}(Q_{j-1}))u_j.\]
Before getting to the core of the proof, we present some technical results that show that the two series above are well-defined, and that expectation and limits can be taken term by term. 

\begin{lem}\label{inversion1}
The series \eqref{decP} is well-defined, and we have, for all $\phi \in L^2(\Rm^2)$,
\[\E\left[\big<\calP^{a,\e}_{\omega,\phi^\e_\omega}(z),\phi\big>_{L^2(\Rd)}\right]=\E\big[\sum_{n=0}^{+\infty} \big<\calT^{n,\e}_{\omega}(z),\phi\big>_{L^2(\Rd)}\big]=\sum_{n=0}^{+\infty}\E\left[ \big<\calT^{n,\e}_{\omega}(z),\phi\big>_{L^2(\Rd)}\right],\]
and
\[\lim_{\e\to 0}\E\left[\big<\calP^{a,\e}_{\omega,\phi^\e_\omega}(z),\phi\big>_{L^2(\Rd)}\right]=\sum_{n=0}^{+\infty}\lim_{\e\to 0}\E\big[ \big<\calT^{n,\e}_{\omega}(z),\phi\big>_{L^2(\Rd)}\big].\]
Similar properties hold for \fref{transferQ}. 
\end{lem} 

\begin{proof} We only consider \fref{decP}, and just need to show that
\[
\sum_{n\geq 1}\sup_{\e\in(0,1)}\left(\E\left[\big\| \calT^{n,\e}_{\omega}(z)\big\|^2_{L^2(\Rd)}\right]\right)^{1/2}<+\infty.
\]
This follows from a direct adaptation of the estimate on $\tilde{R}_n^\e$ in the proof of \fref{IL} in Proposition \ref{evaprop}, which yields 
\[\begin{split}
\E\big[\big\|&\calT^{n,\e}_{\omega}(z)\big\|^2_{L^2(\Rd)}\big]\leq \frac{C^{2n}}{n!} (nr_\calS+r_{\hat{f}_0})^2,
\end{split}\]
where $r_{\hat{f}_0}$ is such that $supp\hat{f}_0 \subset B(0,r_{\hat{f}_0})$. This concludes the proof.
\end{proof}
\medskip

Owing Lemma \ref{inversion1}, it is not difficult to show that $\calX^\e_\omega$ is the unique solution to \fref{eqdiff0} such that 
\[
\E\left[\big\|\calX^\e_\omega(z)\big\|^2_{L^2(\Rd)}\right]\leq C \| \hat f_0(\omega,\cdot)\|^2_{L^2(\Rd)}.
\]
We now proceed to the proof itself and write 
 \[\calP^{a,\e}_{\omega,\phi^\e_\omega}(z,\kappa)-\calX^\e_\omega(z,\kappa)=\sum_{n\geq 1} \tilde{\calT}^{n,\e}_{\omega}(z,\kappa),\]
 where
\[ \tilde{\calT}^{n,\e}_{\omega}(z,\kappa): = \sum_{(l_1,\dots,l_{n-1})\in \calL_n} \int_{\Delta_n(z)} d\mathbf{u}^{(n)}  \int_{\calS^n_{\omega ,\e}(\kappa)}\mathbf{m}(d\mathbf{q}^{(n)})\prod_{m=1}^n \tilde{H}^{\e}_{\omega,l_{m-1},l_m}(u_m,Q_{m-1},Q_m)\phi^\e_\omega(Q_n) \]
 with $l_0=l_n:=1$, and
 \begin{equation}\label{defindices}\begin{split}
 \calL_n:=\big\{\mathbf{l}^{(n)}=(l_1,\dots,l_{n-1})\in \{1,2\}^{n-1},\quad\text{s.t.}&\quad \exists m_0\in\{2,\dots,n-1\}\\
 &\text{with}\quad l_{m_0}=1\quad \text{and}\quad l_{m_0-1}=2 \big\}.
\end{split} \end{equation}
The set $\calL_n$ is such that there is at least one contribution of the off-diagonal of $\tilde{H}^{\e}_\omega$. Following  Lemma \ref{inversion1}, we can study the series above term by term, and in particular,
\[
\lim_{\e\to 0}\E\big[\|\calP^{a,\e}_{\omega,\phi^\e_\omega}(z)-\calX^\e_\omega(z) \|_{L^2(\Rd)} \big] \leq \sum_{n\geq 1} \lim_{\e\to 0} \E\big[\| \tilde{\calT}^{n,\e}_{\omega}(z) \|^2_{L^2(\Rd)} \big]^{1/2},
\]
where, using  the second point of Proposition \ref{LRM},
\[\begin{split}
 \lim_{\e\to 0}&\E\left[\| \tilde{\calT}^{n,\e}_{\omega}(z) \|^2_{L^2(\Rd)} \right]\\
 &= \Big(\frac{k_\omega C^{1/2}_{\mathfrak{H}}}{2}\Big)^{2n} \int_{\{\vert \kappa \vert \leq  n r_\calS+r_{\hat{f}_0}\}}d\kappa  \int_{\calS^n\times \calS^n} \E[\mathbf{m}(d\mathbf{q}^{(n)}_1) \mathbf{m}(d\mathbf{q}^{(n)}_2 )]\phi^0_\omega(Q_{1,n})\overline{\phi^0_\omega(Q_{2,n})}\\
 &\qquad \times  \sum_{\substack{(l_{1,1},\dots,l_{1,n-1})\in \calL_n \\ (l_{2,1},\dots,l_{2,n-1})\in \calL_n }} \sum_{\mathcal{F}_{2,n}} I_{\mathcal{F}_{2,n}}(\mathbf{l}^{(n)}_1,\mathbf{l}^{(n)}_2,\mathbf{q}^{(n)}_1,\mathbf{q}^{(n)}_2) \prod_{(\alpha,\beta)\in\mathcal{F}_{2,n}} \hat{R}(q_{\alpha},q_{\beta})  .
\end{split}\]
We will show that the limit is actually zero. Above, $r_{\hat f_0}$ is as in the proof of Lemma \ref{inversion1} and the second sum is taken over all the pairings $\mathcal{F}_{2,n}$ of $\{1,2\}\times\{1,\dots,n\}$, and $\phi^0_\omega$ is defined by \eqref{defphi0}. Here, we have introduced the notations
\[\begin{split}
 I_{\mathcal{F}_{2,n}}(\mathbf{l}^{(n)}_1,&\mathbf{l}^{(n)}_2,\mathbf{q}^{(n)}_1,\mathbf{q}^{(n)}_2)= \lim_{\e\to 0} \int_{\Delta_n(z)} d\mathbf{u}^{(n)}_1 \int_{\Delta_n(z)} d\mathbf{u}^{(n)}_2\\
 &\prod_{(\alpha,\beta)\in\mathcal{F}_{2,n}} \frac{\calE^\e_{l_{\alpha-(0,1)},l_\alpha}(\alpha(1),u_\alpha,Q_{\alpha-(0,1)},Q_\alpha)\calE^{\e}_{l_{\beta-(0,1)},l_\beta}(\beta(1),u_\beta,Q_{\beta-(0,1)},Q_\beta)}{\vert u_\alpha-u_\beta \vert^{\mathfrak{H}}}
 \end{split}\]
 with
\[
\calE^\e_{j,l}(1,u,p,q):=\left\{ \begin{array}{ccc}
e^{ik_\omega (\lambda_{\e,\omega}(q)-\lambda_{\e,\omega}(p))u/\e^2}&\text{if}& (j,l)=(1,1)\\
e^{-ik_\omega (\lambda_{\e,\omega}(q)+\lambda_{\e,\omega}(p))u/\e^2}&\text{if}& (j,l)=(1,2)\\
-e^{ik_\omega (\lambda_{\e,\omega}(q)+\lambda_{\e,\omega}(p))u/\e^2}&\text{if}& (j,l)=(2,1)\\
-e^{-ik_\omega (\lambda_{\e,\omega}(q)-\lambda_{\e,\omega}(p))u/\e^2}&\text{if}& (j,l)=(2,2),
\end{array}\right.\]
and $\calE^\e_{j,l}(2,u,p,q):=\overline{\calE^\e_{j,l}(1,u,p,q)}$. Now, let us consider a pairing $\mathcal{F}_{2,n}$ and $\mathbf{l}^{(n)}_1\in \calL_n$. Using the notation of \eqref{defindices}, take also a couple $(\alpha_0,\beta_0)$ such that $\alpha_0(1)=1$ and $\alpha_0(2)=m_0$ for instance. The other cases follow exactly the same lines and are omitted. Using the following relation
\[
\vert u-v\vert^{-\mathfrak{H}}=\tilde{c}_{\mathfrak{H}}\int\frac{e^{ir(u-v)}}{\vert r\vert^{1-\mathfrak{H}}}dr,
\]
where $\tilde{c}_{\mathfrak{H}}:=\Gamma(2H-1)\sin(\pi H)/\pi$ with $H=(2-\mathfrak{H})/2$, we single out the pairing $(\alpha_0,\beta_0)$ and obtain 
\[\begin{split}
 &I_{\mathcal{F}_{2,n}}(\mathbf{l}^{(n)}_1,\mathbf{l}^{(n)}_2,\mathbf{q}^{(n)}_1,\mathbf{q}^{(n)}_2)=\tilde{c}_{\mathfrak{H}} \lim_{\e\to 0}\int \frac{dr}{\vert r \vert^{1-\mathfrak{H}}}  \int_{\Delta_n(z)} d\mathbf{u}^{(n)}_1 \int_{\Delta_n(z)} d\mathbf{u}^{(n)}_2\\
& \hspace{1cm}\times\prod_{\substack{(\alpha,\beta)\in\mathcal{F}_{2,n}\\ (\alpha,\beta)\neq (\alpha_0,\beta_0)}} \frac{\calE^\e_{l_{\alpha-(0,1)},l_\alpha}(\alpha(1),u_\alpha,Q_{\alpha-(0,1)},Q_\alpha)\calE^\e_{l_{\beta-(0,1)},l_\beta}(\beta(1),u_\beta,Q_{\beta-(0,1)},Q_\beta)}{\vert u_\alpha-v_\beta\vert^{\mathfrak{H}}} \\
&\hspace{1cm}\times e^{ir u_{\alpha_0}}\calE^\e_{l_{\alpha_0-(0,1)},l_{\alpha_0}}(\alpha_0(1),u_{\alpha_0},Q_{\alpha_0-(0,1)},Q_{\alpha_0})\\
 &\hspace{1cm}\times e^{-ir u_{\beta_0}}\calE^\e_{l_{\beta_0-(0,1)},l_{\beta_0}}(\beta_0(1),u_{\beta_0},Q_{\beta_0-(0,1)},Q_{\beta_0}).
\end{split}\]
Integrating by parts (with respect to the variable $u_{\alpha_0}$) the function  $e^{iu_{\alpha_0}(r+k_\omega (\lambda_{\e,\omega}(Q_{\alpha_0})+\lambda_{\e,\omega}(Q_{\alpha_0-(0,1)}))/\e^2)}$, with antiderivative
\[\frac{e^{iu_{\alpha_0}(r+k_\omega (\lambda_{\e,\omega}(Q_{\alpha_0})+\lambda_{\e,\omega}(Q_{\alpha_0-(0,1)}))/\e^2)}-1}{i(r+k_\omega (\lambda_{\e,\omega}(Q_{\alpha_0})+\lambda_{\e,\omega}(Q_{\alpha_0-(0,1)}))/\e^2)},\]
we find, using dominated convergence,
\[\begin{split}
&\vert  I_{\mathcal{F}_{2,n}}(\mathbf{l}^{(n)}_1,\mathbf{l}^{(n)}_2,\mathbf{q}^{(n)}_1,\mathbf{q}^{(n)}_2)\vert\\[3mm]
& \hspace{1cm}\leq C_1 \int_{\Delta_{n-1}(z)} d\mathbf{u}^{(n)} _{\alpha_0}\int_{\Delta_n(z)} d\mathbf{u}^{(n)}_2  \prod_{\substack{(\alpha,\beta)\in\mathcal{F}_{2,n}\\ (\alpha,\beta)\not= (\alpha_0,\beta_0)}}  \vert u_\alpha-u_\beta \vert^{-\mathfrak{H}} \\
 &\hspace{2.5cm}\times  \lim_{\e\to0}\int  \Big\vert \frac{e^{iu_{\alpha_0-(0,1)}(r+k_\omega (\lambda_{\e,\omega}(Q_{\alpha_0})+\lambda_{\e,\omega}(Q_{\alpha_0-(0,1)}))/\e^2)}-1}{i(r+k_\omega (\lambda_{\e,\omega}(Q_{\alpha_0})+\lambda_{\e,\omega}(Q_{\alpha_0-(0,1)}))/\e^2)}\Big\vert \frac{dr}{\vert r\vert^{1-\mathfrak{H}}}\\
 &\hspace{1.5cm}+C_2   \int_{\Delta_{n-1}(z)} d\mathbf{u}^{(n)} _{\alpha_0+(0,1)}\int_{\Delta_n(z)} d\mathbf{u}^{(n)}_2  \prod_{\substack{(\alpha,\beta)\in\mathcal{F}_{2,n}\\ (\alpha,\beta)\not= (\alpha_0+(0,1),\beta_0)}}  \vert u_\alpha-u_\beta \vert^{-\mathfrak{H}} \\
 &\hspace{2.5cm}\times \lim_{\e\to0} \int  \Big\vert\frac{e^{iu_{\alpha_{0}}(r+k_\omega (\lambda_{\e,\omega}(Q_{\alpha_0})+\lambda_{\e,\omega}(Q_{\alpha_0-(0,1)}))/\e^2)}-1}{i(r+k_\omega (\lambda_{\e,\omega}(Q_{\alpha_0})+\lambda_{\e,\omega}(Q_{\alpha_0-(0,1)}))/\e^2)}\Big\vert\frac{dr}{\vert r\vert^{1-\mathfrak{H}}}
 \end{split}\]
 with $\mathbf{u}^{(n)}_{\alpha_0}:=(u_{1,1},\dots,u_{1,m_0-1},u_{1,m_0+1},\dots,u_{1,n})$, and where $\Delta_{n-1}(z)$ is the same simplex as \fref{simplexe} with now $n-1$ elements.  Let us remark that we are working here with fixed $\mathbf{q}^{(n)}_1$ and $\mathbf{q}^{(n)}_2$, so that
 \[\lim_{\e\to 0}\lambda_{\e,\omega}(Q_{\alpha_0})+\lambda_{\e,\omega}(Q_{\alpha_0-(0,1)})=2. \]
 Therefore, together with the Markov inequality and dominated convergence, the following lemma concludes the proof of Proposition \ref{noback}.
  \begin{lem} \label{fastphase}
  For all $a\not=0$ and $u\not=0$, we have
  \[\lim_{\e\to 0}\int \frac{\vert e^{iu(r-a/\e^s)}-1\vert}{\vert r-a/\e^s\vert \vert r\vert^{1-\mathfrak{H}} } dr=0.\]
  \end{lem}
  
 \begin{proof} Let $\mu>0$ and $\eta>0$ be small parameters, and let us decompose the integral into three parts as follows:
 \[\int \frac{\vert e^{iu(r-a/\e^s)}-1\vert}{\vert r-a/\e^s\vert \vert r\vert^{1-\mathfrak{H}} } dr = \Big(\int_{\vert r-a/\e^s\vert>\mu/\e^s}+\int_{\eta<\vert r-a/\e^s\vert<\mu/\e^s}+\int_{\vert r-a/\e^s\vert<\eta}  \Big)\frac{\vert e^{iu(r-a/\e^s)}-1\vert}{\vert r-a/\e^s\vert \vert r\vert^{1-\mathfrak{H}} }dr.\]
We treat the last integral first and  make the change of variable $r\to r+a/\e^s r$ to obtain
\[\begin{split}
\int_{\vert r-a/\e^s\vert<\eta}\frac{\vert e^{iu(r-a/\e^s)}-1\vert}{\vert r-a/\e^s\vert \vert r\vert^{1-\mathfrak{H}} }dr&= \int_{\vert r \vert<\eta}\frac{\vert e^{iur}-1\vert}{\vert r\vert \vert r+a/\e^s\vert^{1-\mathfrak{H}} }dr \\
&\leq \vert u \vert  \int_{\vert r \vert<\eta} \frac{dr}{\vert r+a/\e^s\vert^{1-\mathfrak{H}} }\\
&\leq  \vert u \vert \e^{s(1-\mathfrak{H})} \int_{\vert r \vert<\eta} \frac{dr}{\vert \vert a \vert -\e^{s}\eta \vert^{1-\mathfrak{H}} }\\
&\leq C \e^{s(1-\mathfrak{H})}.
\end{split}\]
 For the second integral, we have, with the change of variable $r\to r/\e^s$,
\[\begin{split}
\int_{\eta<\vert r-a/\e^s\vert<\mu/\e^s}\frac{\vert e^{iu(r-a/\e^s)}-1\vert}{\vert r-a/\e^s\vert \vert r\vert^{1-\mathfrak{H}} }dr &=\e^{s} \int_{\e^s \eta<\vert \e^s r-a\vert<\mu}\frac{\vert e^{iu(\e^s r-a)/\e^s}-1\vert}{\vert \e^s r-a\vert \vert r\vert^{1-\mathfrak{H}} }dr \\
&\leq \frac{\e^{s(1-\mathfrak{H})}}{(\vert a\vert-\e^s\eta)^{1-\mathfrak{H}}} \int_{\e^s\eta<\vert  r-a\vert<\mu}\frac{dr}{\vert  r-a\vert  }\\
&\leq C \e^{s(1-\mathfrak{H})} \log(1/\e).
\end{split}\]
 For the last term, we find, after the change of variable $r\to r/\e^s$,
 \[\begin{split}
 \int_{\vert r-a/\e^s\vert>\mu/\e^s} \frac{\vert e^{iu(r-a/\e^s)}-1\vert}{\vert r-a/\e^s\vert \vert r\vert^{1-\mathfrak{H}} }dr&=\e^{s(1-\mathfrak{H})}   \int_{\vert r-a\vert>\mu} \frac{\vert e^{iu(r-a)/\e^s}-1\vert}{\vert r-a\vert \vert r\vert^{1-\mathfrak{H}} }dr \\
& \leq \e^{s(1-\mathfrak{H})} \int_{\vert r-a\vert>\mu}\frac{dr}{\vert r-a\vert  \vert r\vert^{1-\mathfrak{H}}  },
 \end{split}\]
 which concludes the proof of Lemma \ref{fastphase}.
\end{proof}

\section{Proof of Proposition \ref{ItoSc}}\label{proofItoSc}

The proof is split into two steps. We start by constructing solutions to the fractional It\^o-Schr\"odinger equation \fref{fracit} in the Fourier form. We prove the announced pathwise regularity in $\hat W^\alpha_\infty(0,L)$ and in $\hat \calC ^{H-\theta}_\infty(0,L)$, and show that, up to a phase shift, the obtained solution is the Fourier transform of the solution in the sense of Definition 1.1. The second step is to prove the convergence of the process $\calX^\e_\omega$. We will show for this the convergence of the moments of $\calX^\e_\omega$ and use a regularized process.

\subsection{Existence theory for the fractional It\^o-Schr\"odinger equation \fref{fracit}}

As explained in the outline, we construct solutions via a regularization procedure. The solutions will be written in terms of Duhamel expansions, since, as in the proof of Proposition \ref{noback}, it will enough to check the term by term convergence to obtain the convergence of the whole series. We then introduce the process $\calX^A_{\omega}$, defined formally as

\begin{equation}\label{defXA}  \calX^A_{\omega}(z,\kappa) = \phi^0_\omega(\kappa)+ \sum_{n\geq 1}  \calX^{A,n}_{\omega}(z,\kappa),
\end{equation}
where
\be \label{expXA}\begin{split}
 \calX^{A,n}_{\omega}(z,\kappa)=(i k_\omega)^n& \int_{\Delta_n(z)}d\mathbf{u}^{(n)}  \int_{\calS^n}\mathbf{m}(d\mathbf{q}^{(n)})e^{iG_n(\mathbf{u}^{(n)},\mathbf{q}^{(n)})}  \phi^0_\omega(Q_n)  \\
 &\times \int_{(-A,A)^n} d\mathbf{r}^{(n)}\prod_{m=1}^n \frac{e^{ir_m u_m}}{\vert r_m \vert ^{H-1/2}} w(dr_m,q_m),
\end{split}\ee
with
 \begin{equation}\label{defG}
 G_n(\mathbf{u}^{(n)},\mathbf{q}^{(n)})=-\frac{1}{2k_\omega}\sum_{m=1}^{n} (\vert Q_m\vert^2-\vert Q_{m-1}\vert^2)u_m.
 \end{equation}
Here, we use the same notation as in Section \ref{proofnoback}. Moreover, $(w(dr,q))_{q\in\calS}$ is the family of complex Gaussian random measures, independent of $m(dq)$, defined by \eqref{defw}, such that $w^\ast(du,q_1)=w(-du,q)$, and with covariance function
\[\mathbb{E}[w(du,q_1)w^\ast(dv,q_2)]=\frac{C_{\mathfrak{H}}\Gamma(2H-1)\sin(\pi H) }{\pi}\hat{R}(q_1,q_2)\delta(u-v)dudv.\] 
We introduce as well the regularized standard fractional Brownian field $B^A_H$ and its $k-$th derivative $b^{A,k}_H := \partial^k_u B^A_H$ ($b^{A}_H:=b^{A,1}_H$),  for $q\in\calS$,
\[ \begin{split}B^A_H(u,q)&:=\sqrt{\frac{H(2H-1)}{C_\mathfrak{H}}}\int_{-A}^{A} \frac{e^{iru}-1}{ir \vert r \vert^{H-1/2}}w(dr,q), \\
b^{A,k}_H(u,q)&:=\sqrt{\frac{H(2H-1)}{C_\mathfrak{H}}} \int_{-A}^{A} \frac{(ir)^{k-1} e^{iru}}{\vert r \vert^{H-1/2}}w(dr,q). 
\end{split}\]
Note that $b^{A,k}_H$ is well-defined since
\be \label{estba}
\E\left[\int_0^z du\int_{\calS} |m|(dq)|b^{A,k}_H(u,q)|^2 \right] \leq C \E\left[\int_{\calS} |m|(dq) \hat R(q,q) \right]\int_0^z du\int_{-A}^A \frac{|r|^{2k-2}dr}{|r|^{2H-1}}  \leq C_A.
\ee
We will use the notation $B^{+\infty}_H=B_H$, with $B_H$ defined by \eqref{FBMB}. Let finally
\[ \calI^A(\psi)(z,\kappa):= ik_\omega \sigma_H \int_{\calS}m(dq) \int_0^z du  e^{-i(\vert \kappa -q \vert^2 -\vert \kappa \vert^2)u/(2k_\omega)} \psi(u,\kappa-q)dB^A_H(u,q),\]
 for $A\in [1,+\infty]$, whenever it is well-defined, and let us remind the reader about the following notation
 \[ \Lambda_{\alpha}(B_H^A(q)):=\frac{1}{\Gamma(1-\alpha)}\sup_{0\leq s\leq t \leq L} \vert D^{1-\alpha}_{t-}[B^A_H]_{t^-}(s,q)\vert.\] 
When $A<+\infty$, we wrote $\calI^A$ in terms of the fractional integral defined in \fref{defintsto}. Since $B_H^A$ has smooth trajectories, it follows from \cite[Theorem 2.4]{zahle} that the fractional integral is equal to the usual Lebesgue integral. 

Our goals are then to show that $(\calX^A_{\omega})_{A \geq 1}$ forms a Cauchy sequence in the appropriate metric space, and to show that the limit satisfies \fref{fracit}. We  start by addressing the path regularity of $\calX^A_\omega$, and by characterizing it as a mild solution to a regularized fractional It\^o-Schr\"odinger equation.

\begin{prop}\label{eqXA} 
The series $\calX^A_{\omega}$ defined by \eqref{defXA} converges in $\hat{\mathcal{C}}^{0}_\infty(0,L)$, and we have 
\begin{equation}\label{intstoXA}\begin{split}
\calX^A_{\omega}(z,\kappa)&= \phi^{0}_\omega(\kappa)+ik_\omega  \sigma_H \int_{\calS} m(dq) \int_0^z  e^{-i(\vert \kappa-q\vert^2-\vert \kappa \vert^2)u/(2k_\omega)}\calX^A_{\omega}(u,\kappa-q) dB_H^A(u,q),
\end{split} \end{equation}
with $\sigma_H^2=C_\mathfrak{H}/(H(2H-1))$. Moreover, the trajectories of $\calX^A_{\omega}$ belong to $\hat{\mathcal{C}}^{\infty}_\infty(0,L)$, and for all $z\geq 0$,
\be \label{conservA}\|\calX^A_{\omega}(z)\|_{L^2(\Rd)}=\|\calX^A_{\omega}(0)\|_{L^2(\Rd)}=\frac{1}{2}\|\hat{f}_0(\omega,\cdot)\|_{L^2(\Rd)}.\ee
\end{prop}

The path regularity of $\calX^A_\omega$ will be useful to justify future calculations. The proposition is proved in Section \ref{proofeqXA}. The next step is to recover the mild formulation of \eqref{ItoSch0} in the Fourier domain by passing to the limit $A\to +\infty$ in \eqref{intstoXA}. We need for this uniform estimates in $A$. They are consequences of next two technical lemmas that follow the ideas of \cite{nualart}. The first one (Lemma \ref{techBA}) is based on the Garsia-Rademich-Rumsey inequality \cite{garcia} below: \\

\noindent \textbf{Garsia-Rademich-Rumsey inequality:}\label{garcia} Let $\tilde{p}\geq 1$ and $\tilde{\alpha}>1/\tilde{p}$. Then, there exists a constant $C_{\tilde{\alpha},\tilde{p}}>0$ such that, for any continuous function $f$ on $[0,L]$, we have
\[\vert f(t)-f(s)\vert^{\tilde{p}}\leq C_{\tilde{\alpha},\tilde{p}}\vert t-s \vert^{\tilde{\alpha}\tilde{p}-1} \int_0^L \int_0^L \frac{\vert f(u)-f(v)\vert^{\tilde{p}}}{\vert u-v\vert^{\tilde{\alpha} \tilde{p}+1}}dudv,\]
for all $(t,s)\in[0,L]$, and with the convention $0/0=0$.\\

Owing the inequality above, the increments (in time) of $B_H^A(t,q)$ can be estimated as follows:

\begin{lem}\label{techBA}
Let $\theta \in (0,H)$ and $A\in [1,+\infty]$. There exist positive random variables $(Z^A_{\theta}(q))_{q\in\calS}$, such that for all $q\in\calS$,
\[ \vert B_H^A(t,q)-B_H^A(s,q)\vert \leq Z^A_{\theta}(q)\vert t-s\vert^{H-\theta} \]
with probability one, and for all $p\geq 1$,
\[\sup_{A\geq 1}\sup_{q\in\calS} \E[\vert Z^A_\theta(q)\vert ^p]<+\infty.\]
Therefore, the trajectories of $B^A_H$ belong to $\hat{\mathcal{C}}^{H-\theta}_\infty(0,L)$ for all $\theta\in (0,H)$, and we also have
\[\sup_{A\geq 1} \sup_{q\in\calS} \E\Big[\Lambda^p_{\alpha}(B_H^A(q)) \Big]<+\infty,\] 
for $\alpha\in(1-H,1/2)$, and all $p\geq 1$. 
\end{lem}
 The second lemma we need consists in adapting \cite[Proposition 4.1]{nualart} to our context, and provides us with estimates on $\calI^A(\psi)$.
\begin{lem}\label{techXnA}
Let $A\in [1,+\infty]$, $k\geq 0$, $\alpha\in(1-H,1/2)$, and $\psi\in  \hat{W}^{\alpha}_{k+2}(0,L)$. We have the following relations:
\begin{enumerate}
\item For all $0\leq s <t\leq L$,
\[\begin{split}
\| &\calI^A(\psi)(t)-\calI^A(\psi)(s)\|_{\calH_k}\leq K_{1,\alpha,k}\int_\calS \vert m\vert (dq) \Lambda_{\alpha}(B_H^A(q)) \\
&\times  \int_s^t dr \Big[ r_\calS(1+r_\calS)\|\psi(r) \|_{\calH_{k+2}}+\frac{\|\psi(r) \|_{\calH_k}}{(r-s)^\alpha}+\int_s^r dy \frac{\|\psi(r)-\psi(y) \|_{\calH_k}}{(r-y)^{\alpha+1}}\Big].
\end{split}\]
\item For all $t\in[0,L]$,
 \[\begin{split}
\|\calI^A(\psi)&(t) \|_{\calH_k} + \int_0^t\frac{\|\calI^A(\psi)(t)-\calI^A(\psi)(s)\|_{\calH_k}}{(t-s)^{\alpha+1}}ds  \leq K_{2,\alpha,k} \int_\calS \vert m\vert(dq) \Lambda_{\alpha}(B_H^A(q)) \\
&\times \int_0^t dr ((t-r)^{-2\alpha}+r^{-\alpha}) \Big[ \|\psi(r) \|_{\calH_{k+2}}+\|\psi(r) \|_{\calH_k} + \int_0^r dy \frac{\|\psi(r)-\psi(y)\|_{\calH_k}}{(r-y)^{\alpha+1}}\Big].
\end{split}\]
\item The following estimate holds:
\[
\| \calI^A(\psi)\|_{1-\alpha,\mathcal{C},\mathcal{H}_k} \leq K_{3,\alpha,k} \int_\calS \vert m\vert (dq) \Lambda_{\alpha}(B_H^A(q))\big[\|\psi \|_{\alpha,k}+\sup_{r\in[0,L]}\|\psi(r) \|_{\mathcal{H}_{k+2}}\big].
\]
\end{enumerate}
Here, $(K_{j,\alpha,k})_{j\in\{1,2,3\}}$ are non random positive constants.
\end{lem}

Note that there is a loss of regularity in the $q$ variable since estimates on $\calI^A(\psi)$ in $\calH_k$ require $\psi \in \calH_{k+2}$. This is due to the the lack of regularizing effects of the Schr\"odinger semigroup: regularity in time has to be exchanged for some regularity in space. The proofs of the last two lemmas are offered below.

\begin{proof}[Proof of Lemma \ref{techBA}] Let $\theta\in(0,H)$ and $p\geq1$  such that $p/\theta\in\mathbb{N}^\ast$. According to Lemma \ref{garcia} with $\tilde{\alpha}=H-\theta/2$ and $\tilde{p}=2/\theta$, we have for all $0\leq s <t\leq L$,
\[  \vert B^A_H(t,q) -B^A_H(s,q)\vert  \leq  \vert t - s  \vert^{H-\theta} \underbrace{C \Big[\int_0^L\int_0^L  \frac{ \vert B^A_H(u,q) -B^A_H(v,q)\vert ^{2/\theta}}{\vert u-v\vert^{2H/\theta}} dudv\Big]^{\theta/2}}_{:=Z^A_\theta(q)}\quad a.s.,\]
where $C$ is a deterministic constant. Moreover, 
\[
 \E\big[\vert Z^A_\theta(q)\vert ^p\big]\leq C^p L^{(p-1)\theta}\Big[\int_0^L\int_0^L \frac{\E[ \vert B^A_H(u,q) -B^A_H(v,q)\vert ^{2p/\theta}]}{\vert u-v\vert^{2Hp/\theta}} dudv\Big]^{\theta/2},
  \]
thanks to the Jensen's and H\"{o}lder's inequalities. Using that $B^A_H(u,q) -B^A_H(v,q)$ is a Gaussian random variable, we find
\[\begin{split}
 \E[ \vert B^A_H(u,q) -B^A_H(v,q)\vert ^{2p/\theta}]&\leq C_{p,\theta} \E[ \vert B^A_H(u,q) -B_H^A(v,q)\vert^2 ]^{p/\theta}\\
&\leq C_{p,\theta} \Big[  \int_{-A}^A \frac{\vert e^{ir(u-v)}-1\vert^2}{ \vert r \vert^{2H+1}}dr \Big]^{p/\theta} \hat{R}^{p/\theta}(q,q)\\
& \leq \tilde{C}_{p,\theta}  \vert u-v \vert^{2Hp/\theta}.
\end{split} \] 
This gives the first point of the lemma. Now, since 
\[\begin{split}
 \vert D^{1-\alpha}_{t-}[B^A_H]_{t^-}(s,q) \vert& \leq \frac{1}{\Gamma(\alpha)}\Big[  \frac{\vert B^A_H(t,q) -B^A_H(s,q) \vert }{(t-s)^{1-\alpha}}\\
 &+(1-\alpha)\int_s^t \frac{\vert B^A_H(y,q) - B^A_H(s,q) \vert }{(y-s)^{2-\alpha}}dy \Big],
\end{split} \]
we have, following the same lines as above, with now using a $\tilde{\theta}\in(0,H+\alpha-1)$ such that $p/\tilde{\theta}\in\mathbb{N}^\ast$,
\[\E\Big[  \Lambda^p_{\alpha}(B_H^A(q)) \Big]\leq  C_{p,\tilde{\theta},\alpha} \E[ \vert Z^A_\theta(q)\vert ^p] \leq \tilde{C}_{p,\tilde{\theta},\alpha}\hat{R}^{p/2}(q,q),\]
which concludes the proof of the lemma. 
 \end{proof}

\begin{proof}[Proof of Lemma \ref{techXnA}] According to the first item of \cite[Proposition 4.1]{nualart} and the fact that
\[\begin{split}
\int^{r}_s dy \frac{\vert 1- e^{i(\vert\kappa\vert^2-\vert \kappa-q\vert^2)(r-y)} \vert }{(r-y)^{\alpha+1}}&\leq \int^{L}_0 dy \frac{\vert 1- e^{i(\vert \kappa\vert^2-\vert \kappa-q\vert^2)y/(2k_\omega)} \vert }{y^{\alpha+1}} \\
&\leq (\vert \vert \kappa\vert^2-\vert \kappa-q\vert^2\vert)\frac{1}{2k_\omega}\int_0^L \frac{dy}{y^{\alpha}}\\
&\leq C \vert q\vert \vert 2\kappa-q\vert,
\end{split}\]
we have
\[\begin{split}
\| \calI^A(\psi)(t)-&\calI^A(\psi)(s)\|_{\calH_k}\leq \sigma_H \int_\calS \vert m\vert (dq) \Lambda_{\alpha}(B_H^A(q)) \int_s^t dr \Big[ C r_\calS \| \vert 2\cdot-q\vert\psi(r,\cdot-q) \|_{\calH_{k}}\\
&+\frac{\|\psi(r,\cdot-q) \|_{\calH_k}}{(r-s)^\alpha}+\alpha \int_s^r dy \frac{\|\psi(r,\cdot-q)-\psi(y,\cdot-q) \|_{\calH_k}}{(r-y)^{\alpha+1}}\Big],
\end{split}\]
which gives the first point of the lemma after the change of variable $\kappa\to \kappa +q$ and basic computations. The second point follows closely the proof of the first item of \cite[Proposition 4.1]{nualart}. The last point is a direct consequence of the first one. In fact, taking $s=0$ in the first relation yields
\[
\|\calI^A(\psi)(t)\|_{\mathcal{H}_k}\leq C_{1,\alpha,k}\Big[ \sup_{t\in[0,L]}\|\psi(t) \|_{\mathcal{H}_{k+2}}+ \sup_{t\in[0,L]} \Big[ \|\psi(t) \|_{\calH_{k}} +\int_0^t \frac{\| \psi(t)-\psi(s)  \|_{\calH^k}}{(t-s)^{\alpha+1}}ds\Big]\Big],
\]
and also
\[\begin{split}
\|\calI^A(\psi)&(t)-\calI^A(\psi)(s)\|_{\mathcal{H}_k}\\
&\leq C_{2,\alpha,k}(t-s)^{1-\alpha}\Big[ \sup_{t\in[0,L]}\|\psi(t) \|_{\mathcal{H}_{k+2}}+ \sup_{t\in[0,L]} \Big[ \|\psi(t) \|_{\calH^{k}} +\int_0^t \frac{\| \psi(t)-\psi(s)  \|_{\calH^k}}{(t-s)^{\alpha+1}}ds\Big]\Big].
\end{split}\]
This concludes the proof of the lemma.
\end{proof}

Now, using the previous two technical lemmas, we have the following proposition, proved in Section \ref{proofcvXA}.

\begin{prop}\label{cvXA}
Let $\theta\in (0,H-1/2)$. The family $(\calX^A_\omega)_{A\geq1}$ converges in probability as $A\to +\infty$ in $\hat{\mathcal{C}}^{H-\theta}_\infty(0,L)$, to a limit denoted by $\calX_\omega$, which is the unique pathwise solution to
\begin{equation}\label{intstoX}\begin{split}
\calX_{\omega}(z,\kappa)&= \phi^{0}_\omega(\kappa)+ik_\omega \sigma_H  \int_{\calS} m(dq)\int_0^z  e^{-i(\vert \kappa-q \vert^2-\vert \kappa\vert^2)u/(2k_\omega)}\calX_{\omega}(u,\kappa-q) dB_H(u,q).
\end{split} \end{equation}
\end{prop}
A corollary (proved in Section \ref{proofcorocvXA}) of this convergence result is the following.  
\begin{cor}\label{corocvXA}
The process defined by $\hat{\Psi}^A_\omega(z,\kappa) = e^{-i\vert \kappa \vert^2z/(2k_\omega)}\calX^A_\omega(z,\kappa)$ converges in probability as $A\to +\infty$ in $\hat{\mathcal{C}}^{H-\theta}_\infty(0,L)$ to $\hat{\Psi}_\omega(z,\kappa)=e^{-i\vert \kappa \vert^2 z/(2k_\omega)}\calX_\omega(z,\kappa)$, which is the unique pathwise solution to
 \begin{equation}\label{ItoschroF}
  \hat{\Psi}_\omega (z,\kappa) = \hat{\Psi}_\omega (0,\kappa) -\frac{i\vert \kappa\vert^2}{2k_\omega} \int_0^z  \hat{\Psi}_\omega(u,\kappa) du  +i k_\omega \sigma_H \int_\calS m(dq) \int_0^z \hat{\Psi}_\omega(u,\kappa -q) dB_H(u,q).
  \end{equation}
\end{cor}

It remains to address the last point of Proposition \ref{ItoSc}, and to show that we can Fourier transform $\hat{\Psi}_\omega$ to recover the fractional Schr\"odinger equation of Definition 1.1. It is just a matter of switching order of integration. The proof is given in section \ref{proofcorocvXA} for the sake of completeness. 

\begin{prop}\label{correspondance}
The Fourier transform realize a one-to-one correspondence between the solution of \eqref{ItoschroF} and the ones of \eqref{ItoSch0}. 
\end{prop}

Finally, let us also remark that according to propositions  \ref{eqXA} and \ref{cvXA}, we have, for all $\eta>0$
\[\begin{split}
\Pro\Big(  \sup_{z\in[0,L]} \Big\vert\|\calX_\omega(z)\|_{L^2(\Rd)}&-\frac{1}{2}\|\hat{f}_0(\omega,\cdot)\|_{L^2(\Rd)} \Big\vert >\eta \Big)\\
&=\lim_{A\to+\infty} \Pro\Big( \sup_{z\in[0,L]} \Big\vert \|\calX^A_\omega(z)\|_{L^2(\Rd)}-\frac{1}{2}\|\hat{f}_0(\omega,\cdot)\|_{L^2(\Rd)} \Big\vert >\eta \Big)=0,
\end{split}\]
which yields the conservation relation for $\calX_\omega$, and at the same time for $\hat{\Psi}_\omega$ and $\Psi_\omega$.

\subsubsection{Proof of Proposition \ref{eqXA}}\label{proofeqXA}

The first step consists in studying the regularity of each term $\calX^{n,A}_\omega$ in the series \eqref{defXA}. This is straightforward: since $\phi^0_\omega \in \calH_k$, for all $k \in \mathbb{N}$, recasting then $\calX_\omega^{A,n}$ in terms of $b^A_H$ as
\be \label{expXA2}\begin{split}
 \calX^{A,n}_{\omega}(z,\kappa)=(i k_\omega)^n& \int_{\Delta_n(z)}d\mathbf{u}^{(n)}  \int_{\calS^n}\mathbf{m}(d\mathbf{q}^{(n)})e^{iG_n(\mathbf{u}^{(n)},\mathbf{q}^{(n)})} \prod_{m=1}^n b_H^{A}(u_m,q_m) \phi^0_\omega(Q_n),
\end{split}\ee
it is a direct consequence of \fref{estba}, the Jensen inequality and the Fubini theorem that $\calX_\omega^{A,n} \in \calC^1([0,L],\calH_k)$. We also obtain the recursive formula below by permuting order of integration, 
\begin{equation}\label{intstoXAn}
\calX^{n,A}_{\omega}(z,\kappa)=ik_\omega \sigma_H  \int_{\calS} m(dq) \int_0^z due^{-i (\vert \kappa-q \vert^2-\vert \kappa \vert^2)u/(2k_\omega)} \calX^{n-1,A}_{\omega}(u,\kappa-q)dB_H^A(u,q).
\end{equation}
We show now that for all $A\geq 1$
\begin{equation}\label{XNA}
\calX^{A}_{\omega,N}:=\sum_{n=0}^N \calX^{n,A}_\omega
\end{equation}
is a Cauchy sequence (with respect to $N$) in probability in $\calC^0([0,L],\calH_k)$ for all $k\geq 1$. We will use the following result (see \cite[Theorem 3.9 pp. 104]{cinlar}).

\begin{thm}
Let $(E,d)$ be a complete metric space. A sequence $(X_n)_{n\geq 0}$ of $E$-valued random variable converges in probability if and only if  
\[ \forall \eta>0\text{ and }\nu>0,\quad \exists n_0>0\quad\text{s.t.}\quad \forall n,m \geq n_0,\qquad \Pro(d(X_n,X_m) > \eta )\leq \nu,\]
that is $(X_n)_{n\geq 0}$ is a Cauchy sequence in probability. 
\end{thm}
Let us first remark that it is enough to work on the event 
\[E=\left(\int_\calS \vert m\vert (dq) \left(\int_0^L du |b_H^{A,2}(u,q)| +|b_H^{A,1}(0,q)|\right) \leq M\right),\]
where $M>0$ is arbitrary.  Indeed, using the Markov and Cauchy-Schwarz inequalities,
\[
 \Pro\Big(\int_0^L du \int_\calS \vert m\vert (dq) |b_H^{A,2}(u,q)|> M \Big)\leq \frac{C}{M} \left(\E\left[\int_0^L du \int_\calS \vert m\vert (dq) |b_H^{A,2}(u,q)|^2 \right]\right)^{1/2}\leq \frac{C}{M},
\]
according to \fref{estba}, and a similar estimate holds for the term involving $b_H^{A,1}(0,q)$ after an easy adaptation of $\fref{estba}$. Hence, on the event $E$, we have $|b_H^{A,1}(u,q)| \leq M$, for all $u \in [0,L]$. It then follows from \fref{expXA} and the latter bound, that
\[\|\calX^{n,A}_\omega\|_{\calC^0([0,L],\calH_k)}\leq (1+n |\calS|)^k \| \hat{f}_0(\omega,\cdot) \|_{\calH_k} \frac{(CM)^n}{n!}.
\]
The division by $n!$ comes from the integration over $\Delta_n(z)$ as in \fref{fact}. This yields the desired Cauchy property of the partial sum  $(\calX^{A}_{\omega,N})_N$. Now, using
\[ \| \calI^A(\calX^{A}_\omega-\calX^{A}_{\omega,N})\|_{\calC^0([0,L],\calH_k)} \leq C \int_0^L du\int_\calS \vert m\vert (dq) |b_H^A(u,q)| \|\calX^{A}_\omega-\calX^{A}_{\omega,N} \|_{\calC^0([0,L],\calH_k)},
\]
it suffices to pass to the limit $N\to+\infty$ in \fref{intstoXAn} to obtain \eqref{intstoXA}. The fact that $\partial_z^p \calX^{A}_\omega \in \calC^0([0,L],\calH_k)$ for all $p \geq 0$ is a consequence of \fref{estba} and of \eqref{intstoXA}. It remains the obtain the conservation of the $L^2$ norm. Since we just proved that $\calX^{A}_\omega \in \calC^1([0,L],\calH_k)$, we can write
\be \label{dzX}\partial_z \calX^{A}_{\omega}(z,\kappa)=ik_\omega \sigma_H \int_{\calS} m(dq) e^{-i(\vert\kappa- q\vert^2-\vert \kappa \vert^2)z/(2k_\omega)}b^A_H(z,q)\calX^{A}_{\omega}(z,\kappa-q),\ee
so that
\[\begin{split}
\big< \partial_z \calX^A_{\omega}(z) ,\calX^A_{\omega}(z) \big>_{L^2(\Rd)}&=ik_\omega\sigma_H  \iint d\kappa m(dq) e^{-i(\vert \kappa- q\vert^2-\vert \kappa \vert^2)z/(2k_\omega)}b^A_H(z,q) \\
&\qquad \times \calX^{A}_{\omega}(z,\kappa-q)\overline{\calX^{A}_{\omega}(z,\kappa)} \in i\mathbb{R},
\end{split}\]
since $\calS$, $m$ and $w(dr,\cdot)$ are assumed to be symmetric. As a result, we obtain
\[\frac{d}{dz}\|\calX^A_{\omega}(z)\|^2_{L^2(\Rd)}= 2Re\Big(\big< \frac{d}{dz}\calX^A_{\omega}(z) ,\calX^A_{\omega}(z) \big>_{L^2(\Rd)}\Big)=0,\]
which concludes the proof of Proposition \ref{eqXA}.

\subsubsection{Proof of Proposition \ref{cvXA}} \label{proofcvXA}

The first step of the proof starts with the following lemma, proved further in this section.
\begin{lem}\label{cauchy2}
The family  $(\calX^{A}_\omega)_{A\geq 1}$ is Cauchy in probability on the complete metric space $(\hat{W}^{\alpha}_{\infty}(0,L),\hat{d}_{\alpha,\infty})$.
\end{lem}

Owing the latter lemma, let us denote by $\calX_\omega$ the limit of $(\calX^{A}_\omega)_{A\geq 1}$ in probability in $(\hat{W}^{\alpha}_{\infty}(0,L),\hat{d}_{\alpha,\infty})$, and let us prove that this limit satisfies \eqref{intstoX}. Note that the last point of Lemma \ref{techXnA} implies that  $(\calX^{A}_\omega)_{A\geq 1}$ is also Cauchy in probability in $ \hat{\mathcal{C}}^{H-\theta}_\infty(0,L)$. According to the second point of Lemma \ref{techXnA}, we have 
\[
\| \calI^{+\infty}(\calX_\omega-\calX^{A}_{\omega})\|_{\alpha,k}\leq  K_{\alpha,\omega} \int_\calS \vert m\vert (dq) \Lambda_{\alpha}(B_H(q)) \|\calX_\omega-\calX^{A}_{\omega} \|_{\alpha,k+2}\]
and
\[\| \calI^{+\infty}(\calX^A_\omega)-\calI^{A}(\calX^{A}_{\omega})\|_{\alpha,k} \leq\int_\calS \vert m\vert (dq) \Lambda_{\alpha}(B_H(q)-B^A_H(q)) \|\calX^{A}_{\omega} \|_{\alpha,k+2}.\]
For the first term, we have for all $\eta$ and $M>0$,
\[\begin{split}
\Pro\Big(\| \calI^{+\infty}(\calX_\omega-\calX^{A}_{\omega})\|_{\alpha,k}>\eta\Big)& \leq \Pro\Big(\|\calX_\omega-\calX^{A}_{\omega}\|_{\alpha,k+2}>\eta/(MK_{\alpha,k})\Big) \\
&\qquad + \Pro\Big( \int_\calS m(dq) \Lambda_{\alpha}(B_H(q))\geq M\Big),
\end{split}\]
so that according to Lemmas \ref{techBA} and \ref{cauchy2},
\begin{equation}\label{lim1}\lim_{A\to+\infty}\Pro\Big(\| \calI^{+\infty}(\calX_\omega-\calX^{A}_{\omega})\|_{\alpha,k}>\eta\Big)=0.\end{equation}
For the second term, we find in the same way,
\[\begin{split}
\Pro\Big(\| \calI^{+\infty}(\calX^A_\omega)-\calI^{A}(\calX^{A}_{\omega})\|_{\alpha,k}>\eta\Big)&\leq  \Pro\Big(\|\calX^{A}_{\omega}\|_{\alpha,k+2}\geq M\Big)\\
&\quad +\Pro\Big(\int_\calS \vert m\vert (dq) \Lambda_{\alpha}(B_H(q)-B^A_H(q))>\eta/(MK_{\alpha,k})\Big).
\end{split}\]
Here, for the first term on the right hand side, using that the convergence in probability implies the convergence in law, the mapping theorem \cite[Theorem 2.7 pp. 21]{billingsley}, and the Portmanteau Theorem \cite[Theorem 2.1 pp. 16]{billingsley}, we have
\[\limsup_{A\to+\infty}  \Pro\Big(\|\calX^{A}_{\omega}\|_{\alpha,k+2}\geq M\Big)\leq \Pro\Big(\|\calX_{\omega}\|_{\alpha,k+2}\geq M\Big),\]
and
\[\lim_{M\to+\infty} \Pro\Big(\|\calX_{\omega}\|_{\alpha,k+2} \geq M\Big)=0,\qquad \textrm{since} \qquad\Pro(\|\calX_{\omega}\|_{\alpha,k+2}<+\infty)=1.\] Second, we have
\begin{equation}\label{lim2}\lim_{A\to+\infty}\Pro\Big(\int_\calS \vert m\vert (dq) \Lambda_{\alpha}(B_H(q)-B^A_H(q))>\eta/(MK_{\alpha,k})\Big)=0,
\end{equation}
following the proof of Lemma \ref{cvB} below. As a result, combining \eqref{lim1} and \eqref{lim2} we obtain that $\calX_\omega$ is a solution of \eqref{intstoX}. The next lemma addresses the pathwise uniqueness of solutions to \eqref{intstoX}, and concludes the proof of Proposition \ref{eqXA}.

\begin{lem}\label{upath} 
Equation \eqref{intstoX} admits a unique pathwise solution in $\hat{W}^{\alpha}_\infty(0,L)$.
\end{lem}
The section is ended by the proofs of Lemmas \ref{cauchy2} and \ref{upath}.

\begin{proof}[Proof of Lemma \ref{cauchy2}] Let $\eta>0$, $\nu>0$, $N>0$, and let us write 
\begin{equation}\label{triple}\begin{split}
\Pro(\hat d_{\alpha,\infty}(\calX^{A+B}_\omega,\calX^{A}_\omega)> \eta)&\leq  \Pro(\hat d_{\alpha,\infty}(\calX^{A+B}_\omega,\calX^{A+B}_{\omega,N})> \eta/3)\\
&\quad +\Pro(\hat d_{\alpha,\infty}(\calX^{A+B}_{\omega,N},\calX^{A}_{\omega,N})> \eta/3)\\
&\quad +\Pro(\hat d_{\alpha,\infty}(\calX^{A}_{\omega,N},\calX^{A}_\omega)> \eta/3),
\end{split}\end{equation}
where $\calX^{A}_{\omega,N}$ is defined by \eqref{XNA}. First, let $M>0$  and $k_\eta$ such that $\sum_{k\geq k_\eta}2^{-k}\leq \eta/6$, so that
\[\begin{split}
\Pro(\hat d_{\alpha,\infty}(\calX^{A}_{\omega,N},\calX^{A}_\omega)> \eta/3)&\leq\Pro\Big( \sum_{k=0}^{k_\eta} \| \calX^{A}_{\omega,N}-\calX^{A}_\omega\|_{\alpha,k} > \eta/6\Big)\\
&\leq \Pro\Big( \sum_{k=0}^{k_\eta} \| \calX^{A}_{\omega,N}-\calX^{A}_\omega\|_{\alpha,k} > \eta/6,\quad \int_\calS \vert m\vert (dq) \Lambda_{\alpha}(B_H^A(q))\leq M\Big)\\
&\quad + \Pro\Big( \int_\calS \vert m\vert (dq) \Lambda_{\alpha}(B_H^A(q)) > M\Big).
\end{split}\]
In order to treat the first term in the r.s.h, we introduce
\[ \calJ_{n,k}(t):= \|\calX^{n,A}_\omega (t)\|_{\calH^k} + \int_0^t ds \frac{\|\calX^{n,A}_\omega(t)-\calX^{n,A}_\omega(s)\|_{\calH^k}}{(t-s)^{\alpha+1}},\]
and find, using inductively the second point of Lemma  \ref{techXnA},
\[\begin{split}
\calJ_{n,k}(t) & \leq K_{2,\alpha,k} \int_\calS \vert m\vert (dq) \Lambda_{\alpha}(B_H^A(q))  \Big[ \int_0^t dr \|\calX^{n-1,A}_\omega (r) \|_{\calH_{k+2}}\\
&\hspace{4cm}+ \int_0^t dr ((t-r)^{-2\alpha}+r^{-\alpha}) \calJ_{n-1,k}(r)\Big]\\
&\leq \tilde{K}_{2,\alpha,k} \int_\calS \vert m\vert (dq) \Lambda_{\alpha}(B_H^A(q)) t^{2\alpha}  \int_0^t dr (t-r)^{-2\alpha}r^{-2\alpha} \calJ_{n-1,k+2}(r)\\
&\leq  \Big[\tilde{K}_{2,\alpha,k}\int_\calS \vert m\vert (dq) \Lambda_{\alpha}(B_H^A(q)) \Big]^n \beta_n(t) \|\hat{f}_0(\omega,\cdot)\|_{\calH_{k+2n}} ,
 \end{split}\]
where $(\beta_n(t))_n$ is given by
\[\beta_0(t):=1\qquad\text{and}\qquad \beta_{n+1}(t):=t^{2\alpha} \int_0^t (t-r)^{-2\alpha} r^{-2\alpha}\beta_n(r)dr.\]
According to \cite[Lemma 7.6]{nualart}, we actually have
\[\beta_n(t)=\frac{\Gamma(1-2\alpha)}{\Gamma((n+1)(1-2\alpha))}\big[ t^{1-2\alpha} \Gamma(1-2\alpha)\big]^n.\]
As a result, since $\hat{f}_0$ is compactly supported, we have
\begin{equation}\label{eeq1}
\|\calX^{n,A}_\omega\|_{\alpha,k}\leq \Gamma(1-2\alpha)  \| \hat{f}_0(\omega,\cdot) \|_{\calH_k} \frac{\Big[K_{\alpha,k,L} \int_\calS \vert m\vert (dq) \Lambda_{\alpha}(B_H^A(q)) \Big]^n}{\Gamma((n+1)(1-2\alpha))}.
\end{equation}
Hence, using \eqref{eeq1}, we have for all $A\geq 1$ and $N_\eta$ sufficiently large  (but not depending on $A$)
\[
 \Pro\Big( \sum_{k=0}^{k_\eta} \| \calX^{A}_{\omega,N_\eta}-\calX^{A}_\omega\|_{\alpha,k} > \eta/6,\quad \int_\calS \vert m\vert (dq) \Lambda_{\alpha}(B_H^A(q))\leq M\Big)=0,
\]
since we have for the previous event
\[\|\calX^{n,A}_\omega\|_{\alpha,k} \leq \frac{(MK_{\alpha,k})^n}{\Gamma((n+1)(1-2\alpha))}.\]
 Moreover, according to Lemma \ref{techBA} and the Markov's inequality, it is clear that
\[ \sup_{A\geq 1}\Pro\Big( \int_\calS \vert m\vert (dq) \Lambda_{\alpha}(B_H^A(q)) > M\Big) \leq \frac{C_0}{M},\]
so that finally 
\begin{equation}\label{res1}\sup_{A\geq 1}\Pro(\hat d_{\alpha,\infty}(\calX^{A}_{\omega,N_\eta},\calX^{A}_\omega)> \eta/3)\leq \frac{C_1}{M}.\end{equation}
In the same way, we obtain
\begin{equation}\label{res2}\sup_{A\geq 1}\sup_{B\geq 0}\Pro(\hat d_{\alpha,\infty,0}(\calX^{A+B}_{\omega,N_\eta},\calX^{A+B}_\omega)> \eta/3)\leq \frac{C_2}{M}.\end{equation}
For the remaining term of \fref{triple}, we have first
\[\begin{split}
\Pro(\hat d_{\alpha,\infty}(\calX^{A+B}_{\omega,N_\eta},\calX^{A}_{\omega,N_\eta})> \eta/3)&\leq \Pro\Big( \sum_{k=0}^{k_\eta} \| \calX^{A+B}_{\omega,N_\eta}-\calX^{A}_{\omega,N_\eta}\|_{\alpha,k} > \eta/6\Big)\\
&\leq \sum_{k=0}^{k_\eta} \Pro\Big( \| \calX^{A+B}_{\omega,N_\eta}-\calX^{A}_{\omega,N_\eta}\|_{\alpha,k} > \eta/(6(k_\eta+1))\Big).
\end{split}\]
Second, according to \eqref{intstoXAn}, we have for all $n\in\{0,\dots,N_\eta-1\}$,
\[\begin{split}
(&\calX^{n+1,A+B}_\omega  -\calX^{n+1,A}_\omega)(z,\kappa) =\calI^{A+B} (\calX^{n,A+B}_\omega-\calX^{n,A}_\omega)(z,\kappa)\\
&+ik_\omega \int_{\calS} m(dq)    \int_0^z du e^{-i (\vert \kappa-q\vert^2-\vert \kappa\vert^2)u/(2k_\omega)} \calX^{n,A}_\omega(u,\kappa-q) d(B_H^{A+B}-B_H^{A})(u,q),
\end{split}\]
and using the second point of Lemma \ref{techXnA} together with \eqref{eeq1}, we find
\begin{equation}\label{eq3}\begin{split} 
\| \calX^{n+1,A+B}_\omega& - \calX^{n+1,A}_\omega \|_{\alpha,k} \\
& \leq \int_\calS \vert m\vert (dq) \Lambda_\alpha (B_H^{A+B}(q) - B_H^A (q))  \sum_{m=0}^{n} \| \calX^{m,A}_\omega \|_{\alpha,k+\alpha(n+1- m)} \\
&\hspace{2cm}\times \Big[C_{1,\alpha,k}\int_\calS \vert m\vert (dq) \Lambda_\alpha (B_H^{A+B}(q)) \Big]^{n-m} \\
&\leq C_{2,\alpha,k}  \int_\calS \vert m\vert (dq) \Lambda_\alpha (B_H^{A+B}(q) - B_H^A (q)) \sum_{m=0}^{n}\frac{\Big[C_{3,\alpha,k}\int_\calS \vert m\vert (dq) \Lambda_\alpha (B_H^{A}(q)) \Big]^{m}}{\Gamma((m+1)(1-2\alpha))} \\
&\hspace{2cm}\times \Big[C_{4,\alpha,k}\int_\calS \vert m\vert (dq) \Lambda_\alpha (B_H^{A+B}(q)) \Big]^{n-m} . 
\end{split}\end{equation}
Now, using the facts that
\[\begin{split}
 \Pro\Big( \| \calX^{A+B}_{\omega,N_\eta}-\calX^{A}_{\omega,N_\eta}&\|_{\alpha,k} > \eta/(6(k_\eta+1))\Big)\\
 &\leq \sum_{n=0}^{N_\eta}   \Pro\Big( \| \calX^{n,A+B}_{\omega}-\calX^{n,A}_{\omega}\|_{\alpha,k} > \eta/(6(k_\eta+1)(N_\eta+1))\Big)
\end{split}\]
as well as
\[\begin{split}
  \Pro&\Big( \| \calX^{n,A+B}_{\omega}-\calX^{n,A}_{\omega}\|_{\alpha,k} > C_3 \Big)\\
  &\leq \Pro\Big( \| \calX^{n,A+B}_{\omega}-\calX^{n,A}_{\omega}\|_{\alpha,k} > C_3,\quad  \int_\calS \vert m\vert (dq )\Lambda_\alpha (B_H^{A}(q)) \leq M,\quad  \int_\calS \vert m\vert (dq) \Lambda_\alpha (B_H^{A+B}(q)) \leq M \Big)\\
  &\quad +  \Pro\Big( \int_\calS \vert m\vert (dq) \Lambda_\alpha (B_H^{A+B}(q)) > M\Big)+\Pro\Big( \int_\calS \vert m\vert (dq) \Lambda_\alpha (B_H^{A}(q)) > M\Big),
 \end{split}\]
we have according to Lemma \ref{techBA}, the Markov inequality, and \eqref{eq3}
\[\begin{split} 
\Pro\Big( \| \calX^{n,A+B}_{\omega}-\calX^{n,A}_{\omega}\|_{\alpha,k} > C_3 \Big)&\leq \frac{C_4}{M}+ \Pro\Big( \int_\calS \vert m\vert (dq) \Lambda_\alpha (B_H^{A+B}(q)-B_H^{A}(q)) \geq C_{5,n}\Big)\\
&\leq  \frac{C_4}{M}+ \E\Big[ \int_\calS \vert m\vert (dq) \Lambda_\alpha (B_H^{A+B}(q)-B_H^{A}(q)) \Big]/C_{5,n}
\end{split}\]
for all $M>0$ and all $C_{5,n}$ sufficiently small (independently of $A$ and $B$). Gathering all the previous estimates, we find, combining \eqref{res1} and \eqref{res2} in \eqref{triple} and taking $M\to+\infty$,
\[\Pro\big(\hat{d}_{\alpha,\infty}(\calX^{A+B}_\omega,\calX^{A}_\omega)>\eta\big)\leq \frac{\nu}{2}+\tilde{C}_{5,N_\eta} \E\Big[ \int_\calS \vert m\vert (dq) \Lambda_\alpha (B_H^{A+B}(q)-B_H^{A}(q)) \Big].\]
Then, the following lemma allows us to conclude the proof of Lemma \ref{cauchy2}.

\begin{lem}\label{cvB}
Let $\tilde{\alpha}\in(1-H,1/2)$, and
\[B_A^{A+B}(u,q):= \int_{\{A<\vert r\vert <A+B\}} \frac{e^{iru}-1}{ir \vert r \vert^{H-1/2}}w(dr,q).\]
We have 
\[\lim_{A\to +\infty}\sup_{B\geq 0}\sup_{q\in\calS} \E\Big[ \Lambda_{\alpha}(B_A^{A+B}(q)) \Big]=0.\]
\end{lem}
As a result, we finally obtain with \eqref{hypmomentm},
\[\lim_{A\to+\infty}\sup_{B\geq 0}\Pro(\hat d_{\alpha,\infty}(\calX^{A+B}_\omega,\calX^{A}_\omega)> \eta)=0,\]
which ends the proof.\end{proof}
\begin{proof}[Proof of Lemma \ref{cvB}]
Following exactly the proof of Lemma \ref{techBA}, we have for $\tilde{\theta}\in(0,H+\tilde{\alpha}-1)$ such that $1/\tilde{\theta}\in\mathbb{N}^\ast$,
\[\begin{split}
 \E\Big[  \Lambda_{\tilde{\alpha}}(B_A^{A+B}(q)) \Big]&\leq K_{\tilde \theta,\tilde \alpha}L^{H+\tilde\alpha-1-\tilde\theta}\big[1+1/(H+\tilde\alpha-1-\tilde\theta)  \big] \\
 &\quad \times \Big[\int_0^L\int_0^L  \frac{\E[ \vert B_A^{A+B}(u,q) -B_A^{A+B}(v,q)\vert ^{2/\tilde\theta}]}{\vert u-v\vert^{2H/\tilde\theta}} dudv\Big]^{\tilde\theta/2},
 \end{split}\]
and since $B_A^{A+B}(u,q) -B_A^{A+B}(v,q)$ is a Gaussian random variable, one has
\[\begin{split}\sup_{q\in\calS}\E\Big[  \Lambda_{\tilde\alpha}(B_A^{A+B}(q)) \Big]&\leq \tilde{K}_{\tilde\theta,\tilde\alpha}\sup_{q\in \calS} \hat{R}^{1/\tilde\theta}(q,q) \\
&\quad \times \Big[\int_0^L\int_0^L  \Big(\int_{\{A<\vert r\vert <A+B\}} \frac{\vert e^{ir(u-v)}-1\vert^2}{ \vert r \vert^{2H+1}}dr\Big)^{1/\tilde\theta}\frac{dudv}{\vert u-v\vert^{2H/\tilde\theta}} \Big]^{\tilde\theta/2}, 
\end{split} \]
The proof is ended using the dominated convergence owing that
\[ \int_{\{\vert r\vert >A\}} \frac{\vert e^{ir(u-v)}-1\vert^2}{ \vert r \vert^{2H+1}}dr\leq d_H \vert u-v \vert^{2H}.\]
\end{proof}
We end the section with the proof of Lemma \ref{upath}.
\begin{proof}[Proof of Lemma \ref{upath}] Let us consider two solutions of \eqref{intstoX} denoted by $\calX_\omega$ and $\tilde{\calX}_\omega$, and let
\[\calJ_k(z)=\|\calX_\omega(z)-\tilde{\calX}_\omega(z)\|_{\calH_k}+\int_0^z ds \frac{\|\calX_\omega(z)-\calX_\omega(s)-(\tilde{\calX}_\omega(z)-\tilde{\calX}_\omega(s))\|_{\calH_k}}{(z-s)^{\alpha+1}}.\]
Using the second point of Lemma \ref{techXnA}, we find
\[\calJ_k(z)\leq Cz^{2\alpha}\int_\calS \vert m\vert(dq)\Lambda_\alpha(B_H(q)) \int_0^z dr(z-r)^{-2\alpha}r^{-2\alpha} \calJ_{k+2}(r),\]
so that considering
\[\tilde{\calJ}(z)=\lim_{N\to+\infty}\sum^N_{k= 0}\calJ_k(z),\]
which takes its values in $\mathbb{R}_+\cup\{+\infty\},$
we obtain by the monotone convergence theorem
\[
\tilde{\calJ}(z)\leq \tilde{C}z^{2\alpha}\int_\calS \vert m\vert(dq)\Lambda_\alpha(B_H(q)) \int_0^z dr(z-r)^{-2\alpha}r^{-2\alpha} \tilde{\calJ}(r).
\]
As a result, according to the particular version of the Gronwall lemma given in \cite[Lemma 7.6]{nualart}, we have $\tilde{\calJ}(z)=0$ for all $z\geq0$. 
\end{proof}

\subsubsection{Proof of Corollary \ref{corocvXA}} \label{proofcorocvXA}

First, it is clear $\hat{\Psi}^A$ converges in probability as $A\to +\infty$ in $\hat{\mathcal{C}}^{H-\theta}_\infty(0,L)$ to $\hat{\Psi}_\omega$. Moreover according to \eqref{dzX}, we obtain
\begin{equation}\label{tempA}\begin{split}
\hat{ \Psi}^A_\omega (z,\kappa)  &=\hat{ \Psi}^A_\omega (0,\kappa)- \frac{i\vert \kappa \vert^2}{2k_\omega}\int_0^z  \hat{\Psi}^A_\omega(u,\kappa)du  +i k_\omega  \sigma_H \int_\calS m(dq) \int_0^z \hat{\Psi}^A_\omega(u,\kappa-q) dB^A_H(u,q)
\end{split}\end{equation}
and we only need to address the convergence of the last term. Introducing
\[ 
\calK^A(\psi)(z,\kappa)=i k_\omega \sigma_H \int_\calS m(dq) \int_0^z \hat{\Psi}^A_\omega(u,\kappa-q) dB^A_H(u,q),
\]
for $A\in [1,+\infty]$ where $B_H^{+\infty}=B_H$, we have
\[\begin{split}
\| \calK^A(\hat{\Psi}^A_\omega)- \calK^{+\infty}(\hat{\Psi}_\omega)\|_{\alpha,k}&\leq \| (\calK^A- \calK^{+\infty})(\hat{\Psi}_\omega)\|_{\alpha,k}+\| \calK^A(\hat{\Psi}^A_\omega-\hat{\Psi}_\omega)\|_{\alpha,k}\\
&\leq  \vert k_\omega\vert \sigma_H \int_\calS \vert m\vert (dq)\Lambda_\alpha (B^{+\infty}_A(q)) \|\hat{\Psi}_\omega \|_{\alpha,k}\\
&\quad +\vert k_\omega\vert \sigma_H \int_\calS \vert m\vert (dq)\Lambda_\alpha (B^A_H(q)) \|\hat{\Psi}^A_\omega-\hat{\Psi}_\omega \|_{\alpha,k},
\end{split}\]
so that proceeding as in the proof of Proposition \ref{cvXA}, we obtain that, for all $\eta>0$
\[\lim_{A\to+\infty}\Pro\Big( \| \calK^A(\hat{\Psi}^A_\omega)- \calK^{+\infty}(\hat{\Psi}_\omega)\|_{\alpha,k}>\eta \Big)=0.\]
Therefore, we can pass to the limit $A \to + \infty$ in \eqref{tempA} and obtain that $\hat{\Psi}_\omega$ satisfies the desired equation. Uniqueness follows from the one-to-one correspondence between \eqref{intstoX} and \eqref{ItoschroF} via $\hat{\Psi}_\omega$ and $\calX_\omega$.

\subsubsection{Proof of Proposition \ref{correspondance}}

With the notations of the previous section, we just need to show that
\[
\big<\tilde{\calK}^{+\infty}(\Psi_\omega)(z) ,\phi \big>_{L^2(\Rd)}=(2 \pi)^2\big<\calK^{+\infty}(\hat{\Psi}_\omega)(z) ,\hat{\phi} \big>_{L^2(\Rd)}
\]
where, for $A\in [1,+\infty]$ and the notation $W_H^{+\infty}=W_H$,
\[\begin{split}
\tilde{\calK}^A(\psi)(z,x)=i k_\omega \int_0^z \psi(u,x)dW^A_H(u,x), \qquad W_H^A(z,x)=\sigma_H \int_\calS m(dq) e^{-iq\cdot x} B^A_H(u,q),
\end{split}\] 
The proof consists in approximating $W_H$ by $W^A_H$ with $A<+\infty$ in order to have sufficient regularity to justify the calculations. We thus write
\[
\big<\tilde{\calK}^{+\infty}(\Psi_\omega)(z) ,\phi \big>_{L^2(\Rd)}=\big<(\tilde{\calK}^{+\infty}-\tilde{\calK}^A)(\Psi_\omega)(z) ,\phi \big>_{L^2(\Rd)}+\big<\tilde{\calK}^A(\Psi_\omega)(z) ,\phi \big>_{L^2(\Rd)}
\]
with
\[\begin{split}
\E\left[\|\tilde{\calK}^{+\infty}(\Psi_\omega)(z)-\tilde{\calK}^A(\Psi_\omega)(z) \|_{L^2(\Rd)}\right]& \leq \E\Big[\|\Psi_\omega(z)\|_{L^2(\Rd)} \sup_{x\in\Rd}\Lambda_\alpha(W_H(x)-W^A_H(x))\Big]\\
&\leq C \| \check{f}_0(\omega,\cdot)\|_{L^2(\Rd)} \sup_{q\in\calS}\E\Big[ \Lambda_\alpha(B^{+\infty}_A(q)) \Big]
\end{split}\]
which converges to $0$ as $A\to+\infty$ according to Lemma \ref{cvB}. Following \fref{estba} and the fact that  $\Psi_\omega \in \calC^0([0,L],H^k(\Rm^2))$ for all $k \in \mathbb{N}$, we can write, 
\[ \tilde{\calK}^A(\Psi)(z,x)= ik_\omega \sigma_H \int_\calS m(dq) \int_0^z  du  \Psi_\omega(u,x) e^{-iq\cdot x}b^A_H(u,q),\]
and thus use the Fubini and Fourier-Plancherel theorems to arrive at
\[
\big<\tilde{\calK}^A(\Psi_\omega)(z) ,\phi \big>_{L^2(\Rd)}=(2\pi)^2\big<\calK^A(\hat{\Psi}_\omega)(z) ,\hat{\phi} \big>_{L^2(\Rd)}.
\]
The proof is ended by passing to the limit in $A$ by adapting what was done above for $\tilde{\calK}^A$.

\subsection{Convergence of $\calX^\e_\omega$}
We investigate in this section the limit in law of $(\calX^\e_{\omega_1}(z),\dots,\calX^\e_{\omega_M}(z))_{\e}$. The tightness of this family is addressed in Proposition \ref{tightness} below and is the straightforward consequence of the conservation of the $L^2$ norm of $\calX^\e_{\omega}$. The characterization of the limit in distribution of the family requires more work. We will use for this a moment method and the regularized process $\calX^A_\omega$ for which formal calculations are justified. In order to fix the ideas, we will investigate first the moment of order one in Proposition \ref{moment1}, and then generalize to moments of any order in Proposition \ref{momentmulti}. The various results are put together in a conclusion at the end of the section.

\subsubsection{Tightness}

This section is devoted to the tightness of the family $(\calX^\e_{\omega_1}(z),\dots,\calX^\e_{\omega_M}(z))_{\e}$. We have the following result:

 \begin{prop}\label{tightness}
For all $z\in [0,L]$, the family $(\calX^\e_{\omega_1}(z),\dots,\calX^\e_{\omega_M}(z))_{\e}$ is tight in $L^2(\Rd)$ equipped with the weak topology.
 \end{prop}
\begin{proof} It suffices to show that the family of complex-valued random variable $(\big<\calX^\e_{\omega_j}(z),\phi_j\big>)_{\e,j\in\{1,\dots,M\}}$ is tight on $\mathbb{C}^M$ for all $\phi_j\in L^2(\Rd)$ ($j\in\{1,\dots,M\}$), which amounts to prove that
\[ \forall \eta>0, \quad\exists \mu>0 \quad \text{such that}\quad\overline{\lim_{\e\to 0}} \mathbb{P}\Big(\sum_{j=1}^M\vert \big<\calX^\e_{\omega_j}(z),\phi_j\big>\vert ^2>\mu \Big)\leq \eta. \]
This is a direct consequence of the following lemma:
\begin{lem}\label{consXe} We have, for all $z\in [0,L]$,
\[\|\calX^\e_{\omega}(z)\|_{L^2(\Rd)}=\|\phi^\e_{\omega}\|_{L^2(\Rd)}.\]
\end{lem}
The proof of the lemma is left to the reader. Since $\calX^\e_{\omega}$ has sufficient regularity (i.e. at least $\calC^1([0,L],L^2(\Rm^2))$ almost surely), it suffices to adapt the proof of \fref{conservA} to obtain the result.
\end{proof}

 \subsubsection{Moment of order one}

\begin{prop}\label{moment1}
For all  $z\in[0,L]$ and $\phi\in \calC^{\infty}_0(\Rd)$, we have 
\[\lim_{\e\to 0}\E\big[\big<\calX^\e_{\omega}(z),\phi\big>_{L^2(\Rd)}\big]=\lim_{A\to +\infty }\E\big[\big<\calX^A_{\omega}(z),\phi\big>_{L^2(\Rd)}\big].\]
\end{prop}

\begin{proof}
 According to Lemma \ref{inversion1}, it suffices to show the term by term convergence of the series defining $\calX^\e_{\omega}$. Moreover, since the integrand in \fref{defXen} is $L^1$ in all variables, we can invoke Fubini Theorem to permute order of integration.  Using the second point of Proposition \ref{LRM}, we have for $n=2n'$ (if $n$ is odd the limit is $0$),
 \[\begin{split}
\lim_{\e\to 0}\E[\big<\calX^{\e,2n'}_{\omega}(z),\phi\big>]&=(i  k_\omega)^{2n'}\int d\kappa \overline{\phi(\kappa)}\int_{\Delta_{2n'}(z)}\int_{ \calS^n} d\mathbf{u}^{(2n')}  \E[\mathbf{m}(d\mathbf{q}^{(2n')})]\\ 
&\quad \times e^{i G_n(\mathbf{u}^{(n)},\mathbf{q}^{(n)})}  \phi^0_\omega(Q_n)C_\mathfrak{H}^{n'} \sum_{\mathcal{F}_{2n'}}  \prod_{(\alpha,\beta)\in\mathcal{F}_{2n'}}  \frac{\hat{R}(q_\alpha,q_{\beta})}{\vert u_\alpha-u_\beta\vert^\mathfrak{H}},
 \end{split}\]
where the sum runs over the pairings $\mathcal{F}_{2n'}$ of $\{1,\dots,2n'\}$ and $G_n$ is defined by \eqref{defG}. We want to relate now the term above with $\calX^{A,2n'}_{\omega}$. We use first for this the Gaussianity of the measures $w$ to find
\[
C_\mathfrak{H}^{n'} \sum_{\mathcal{F}_{2n'}}\prod_{(\alpha,\beta)\in\mathcal{F}_{2n'}} \frac{\hat{R}(q_\alpha,q_{\beta}) }{\vert u_\alpha-u_\beta \vert^{\mathfrak{H}}}=\lim_{A\to+\infty}\E\Big[\prod_{m=1}^{2n'}  \int_{-A}^A \frac{e^{i r_m u_m}}{\vert r_m \vert^{H-1/2}}w(dr_m,q_m) \Big].
\]
Moreover, since on the one hand,
\be \label{maj}
\left|\E\Big[\prod_{m=1}^{2}  \int_{-A}^A \frac{e^{i r_m u_m}}{\vert r_m \vert^{H-1/2}}w(dr_m,q_m) \Big] \right| \leq C |u_1-u_2|^{2H-2}, \qquad H \in (1/2,1),
\ee
and on the other hand that the integrand in \fref{expXA2} is $L^1$ in all variables, we can invoke both the Fubini Theorem and dominated convergence to obtain, for all $n\geq1$
 \[\begin{split}
\lim_{\e\to 0}\E[\big<\calX^{\e,n}_{\omega}(z),\phi\big>]&=(i  k_\omega)^{n}\int d\kappa \overline{\phi(\kappa)}\int_{\Delta_{n}(z)}d\mathbf{u}^{(n)}\int_{\calS^n}\E[\mathbf{m}(d\mathbf{q}^{(n)})]e^{i G_n(\mathbf{u}^{(n)},\mathbf{q}^{(n)})}  \phi^0_\omega(Q_n)\\ 
&\qquad \qquad \qquad\times \lim_{A\to+\infty }\E\Big[\prod_{j=1}^{n} \int_{-A}^A \frac{e^{i r_j u_j}}{\vert r_j \vert^{H-1/2}}w(dr_j,q_j) \Big]\\
&=\lim_{A\to+\infty }\E[\big<\calX^{A,n}_{\omega}(z),\phi\big>].
 \end{split}\]
Above, $\phi^0_\omega=\lim_\e \phi^\e_\omega$ is given by \eqref{defphi0}. It just remains to show that limit and expectation can be taken term by term in the series defining $\calX^{A}_{\omega}$. This is the object of the next lemma.
\begin{lem}\label{inversion1s}
We have for all $z\in[0,L]$
\[ \lim_{A\to+\infty }\E[\big<\calX^{A}_{\omega}(z),\phi\big>]=\sum_{n\geq 0}\lim_{A\to+\infty }\E[\big<\calX^{A,n}_{\omega}(z),\phi\big>]<\infty.\]
\end{lem}
\begin{proof} This result is just a consequence of the fact that
\[\sum_{n\geq 1}\sup_{A\geq1}\E[\| \calX^{A,n}_{\omega}(z) \| ^2_{L^2(\Rd)}]^{1/2}<+\infty.
\]
As in Lemma \ref{inversion1}, we have using \eqref{hypmomentm},
\[\begin{split}
\E\Big[\|  \calX^{A,n}_{\omega}(z)\|^2_{L^2(\Rd)} \Big] &\leq  k^{2n}_\omega \int d\kappa \int_{\calS^{n}\times \calS^n} \E[ \vert \mathbf{m}\vert (d\mathbf{q}^{(n)}_1) \vert \mathbf{m}\vert (d\mathbf{q}^{(n)}_2)]\int_{\Delta_n(z)}d\mathbf{u}^{(n)}_1 \int_{\Delta_n(z)}  d\mathbf{u}^{(n)}_2\\
& \qquad \times \Big\vert \E\Big[\prod_{m=1}^{n} \int_{-A}^{A}\frac{e^{ir^1_m u^1_m}}{\vert r^1_m \vert ^{H-1/2}}w(dr^1_{m},q^1_m)\overline{\int_{-A}^{A}\frac{e^{ir^2_m u^2_m}}{\vert r^2_m \vert ^{H-1/2}}w(dr^2_{m},q^2_m)}\Big]\\
&\qquad  \times \Big\vert \vert \hat{f}_0(\omega,Q_{1,n})\hat{f}_0(\omega,Q_{2,n})\vert \\
&\leq (nr_\calS+r_{\hat{f}_0})^2\frac{C^{2n} }{(n!)^2} (2n-1)!! \Big[\int_{(0,z)^2} |u_1-u_2|^{2H-2} du_1 du_2 \Big]^{n}\\
&\leq  (nr_\calS+r_{\hat{f}_0})^2 \frac{C^{n} }{n!}.
\end{split}\]
Above, we used \fref{maj}, the term $(2n-1)!!=(2n)!/(2^{n} n!)$ is the number of pairings of $\{1,\dots,2n\}$, and the term $(n!)^2$ is a consequence of \fref{fact}. 
\end{proof}

This  concludes the proof of Proposition \ref{moment1}.
\end{proof}

\subsubsection{Arbitrary Order Moments}

In the forthcoming computations, all indices with the subscript 2 correspond to the complex conjugate terms.  

\begin{prop}\label{momentmulti}
We have for all $z\in[0,L]$, frequencies $(\omega_{1,1},\dots,\omega_{1,M_1},\omega_{2,1},\dots, \omega_{2,M_2})$, and test functions $(\phi_{1,1},\dots,\phi_{1,M_1},\phi_{2,1},\dots,\phi_{2,M_2})$ in $ \calC^{\infty}_0(\Rd)$,
\[\begin{split}
\lim_{\e\to0}\mathbb{E}\Big[\prod_{j_1=1}^{M_1}\big<\calX^\e_{\omega_{1,j_1}}(z)&,\phi_{1,j_1}\big>_{L^2(\Rd)} \prod_{j_2=1}^{M_2}\overline{\big<\calX^\e_{\omega_{2,j_2}}(z),\phi_{2,j_2}\big>_{L^2(\Rd)}}\Big]\\
&=\lim_{A\to+\infty} \E\Big[\prod_{j_1=1}^{M_1} \big<\calX^A_{\omega_{1,j_1}}(z),\phi_{1,j_1}\big>_{L^2(\Rd)}\prod_{p_2=1}^{M_2}\overline{\big<\calX^A_{\omega_{2,j_2}}(z),\phi_{2,j_2}\big>_{L^2(\Rd)}}\Big]<+\infty.
\end{split}\]
\end{prop}
\begin{proof} Using \eqref{transferQ}, we have
\[
\prod_{j_1=1}^{M_1}\big<\calX^\e_{\omega_{1,j_1}}(z),\phi_{1,j_1}\big>_{L^2(\Rd)} \prod_{j_2=1}^{M_2}\overline{\big<\calX^\e_{\omega_{2,j_2}}(z),\phi_{2,j_2}\big>_{L^2(\Rd)}}=\sum_{j_1=1}^{M_1}\sum_{j_2=1}^{M_2}\sum_{n_{1,j_1}= 0}^{+\infty} \sum_{n_{2,j_2}= 0}^{+\infty}\mathbf{X}^{\e}_{\mathbf{n}}:=\sum_{\mathbf{J_n}} \mathbf{X}^{\e}_{\mathbf{n}},
\]
where  
\[ \begin{split}
\mathbf{X}^{\e}_{\mathbf{n}}&= \frac{i^{n_1-n_2}}{ \e^{n(s-1/2)}}  \int\cdots\int  \prod_{j_1=1}^{M_1}d\kappa_{1,j_1} k^{n_{1,j_1}}_{\omega_{1,j_1}} \overline{\phi_{1,j_1}(\kappa_{1,j_1})} \prod_{j_2=1}^{M_2}d\kappa_{2,j_2}k^{n_{2,j_2}}_{\omega_{2,j_2}}\phi_{2,j_2}(\kappa_{2,j_2}) \\
&\quad \times \prod_{j_1=1}^{M_1} \int_{\calS^{n_{1,j_1}}(\kappa_{1,j_1})} \mathbf{m}(d\mathbf{q}^{(n_{1,j_1})}_{1,j_1})\int_{\Delta_{n_{1,j_1}}(z)} d\mathbf{u}^{(n_{1,j_1})}_{1,j_1} \\
&\quad \times \prod_{j_2=1}^{M_2}  \int_{\calS^{n_{2,j_2}}(\kappa_{2,j_2})} \mathbf{m}(d\mathbf{q}^{(n_{2,j_2})}_{2,j_2})\int_{\Delta_{n_{2,j_2}}(z)} d\mathbf{u}^{(n_{2,j_2})}_{2,j_2} \\
&\quad \times\prod_{j_1=1}^{M_1} e^{iG^\e_n(\mathbf{u}^{(n_{1,j_1})}_{1,j_1},\mathbf{q}^{(n_{1,j_1})}_{1,j_1})} \phi^\e_{\omega_{1,j_1}}(Q_{1,j_1,n_{1,j_1}})\prod_{j_2=1}^{M_2}e^{-iG^\e_n(\mathbf{u}^{(n_{2,j_2})}_{2,j_2},\mathbf{q}^{(n_{2,j_2})}_{2,j_2})}\overline{\phi^\e_{\omega_{2,j_2}}(Q_{2,j_2,n_{2,j_2}})}\\
&\quad \times\prod_{j_1=1}^{M_1} \prod_{m_{1,j_1}=1}^{n_{1,j_1}}  \hat{V}(u_{1,j_1,m_{1,j_1}}/\e, q_{1,j_1,m_{1,j_1}})\prod_{j_2=1}^{M_2} \prod_{m_{2,j_2}=1}^{n_{2,j_2}} \overline{\hat{V}(u_{2,j_2,m_{2,j_2}}/\e, q_{2,j_2,m_{2,j_2}})},
\end{split}
\]
with
\[
n_1:=\sum_{j_1=1}^{M_1}n_{1,j_1},\qquad n_2:=\sum_{j_2=1}^{M_2}n_{2,j_2}, \qquad\text{and} \qquad n:=n_1+n_2.
\]
As before, we need to show that limit and expectation can be taken term by term.
\begin{lem}\label{inversion2}
The series $\sum_{\mathbf{J}_{\mathbf{n}}} \mathbf{X}^{\e}_{\mathbf{n}}$ is well-defined, and we have 
\[\mathbb{E}\Big[\prod_{j_1=1}^{M_1}\big<\calX^\e_{\omega_{1,j_1}}(z),\phi_{1,j_1}\big>_{L^2(\Rd)} \prod_{j_2=1}^{M_2}\overline{\big<\calX^\e_{\omega_{2,j_2}}(z),\phi_{2,j_2}\big>_{L^2(\Rd)}}\Big]=\sum_{\mathbf{J}_{\mathbf{n}}} \E[\mathbf{X}^{\e}_{\mathbf{n}}],\]
as well as
\[\lim_{\e\to0}\mathbb{E}\Big[\prod_{j_1=1}^{M_1}\big<\calX^\e_{\omega_{1,j_1}}(z),\phi_{1,j_1}\big> _{L^2(\Rd)}\prod_{j_2=1}^{M_2}\overline{\big<\calX^\e_{\omega_{2,j_2}}(z),\phi_{2,j_2}\big>_{L^2(\Rd)}}\Big]=\sum_{\mathbf{J}_{\mathbf{n}}}\lim_{\e\to0} \E[\mathbf{X}^{\e}_{\mathbf{n}}],\]
\end{lem}

\begin{proof} As usual, it suffices to show that
\begin{equation}\label{unifboundXe2}
\sum_{\mathbf{J}_{\mathbf{n}}} \sup_{\e\in(0,1)}\mathbb{E}[\vert \mathbf{X}^{\e}_{\mathbf{n}}\vert ^2]^{1/2}<+\infty.
\end{equation}
Adapting once more Lemma \ref{inversion1} and the first point of Proposition \ref{LRM}, we have for $s=2-\mathfrak{H}/2$,
\[\begin{split}
\mathbb{E}[\vert \mathbf{X}^{\e}_{\mathbf{n}}\vert ^2]&\leq  \int\cdots\int  \prod_{j_1=1}^{M_1}d\kappa_{1,j_1}\mathbf{1}_{\{\vert \kappa _{1,j_1}\vert < n_{1,j_1} r_\calS+r_{\hat{f}_0} \}}  \prod_{j_2=1}^{M_2}d\kappa_{2,j_2}\mathbf{1}_{\{\vert \kappa _{2,j_2}\vert < n_{2,j_2} r_\calS+r_{\hat{f}_0} \}} \\
&\quad \times \prod_{j_1=1}^{M_1}\frac{C^{n_{1,j_1}}}{n_{1,p_1}!}  \prod_{j_2=1}^{M_2} \frac{C^{n_{2,j_2}}}{n_{2,j_2}!} ,
\end{split}\]
which gives \eqref{unifboundXe2}.
\end{proof}
The latter lemma can be directly generalized to moments of $\calX^A_\omega$ using the fact that the measures are Gaussian. Using now the second point of Proposition \ref{LRM}, we obtain for $n=2n'$,
\[ \begin{split}
\lim_{\e\to0}\mathbb{E}[&\mathbf{X}^{\e}_{\mathbf{n}}]= i^{n_1-n_2}  \int\cdots\int  \prod_{j_1=1}^{M_1}d\kappa_{1,j_1}k^{n_{1,j_1}}_{\omega_{1,j_1}}\overline{\phi_{1,j_1}(\kappa_{1,j_1})} \prod_{j_2=1}^{M_2}d\kappa_{2,j_2}k^{n_{2,j_2}}_{\omega_{2,j_2}}\phi_{2,j_2}(\kappa_{2,j_2})\\
&\hspace{2cm}\times\E\Big[ \prod_{j_1=1}^{M_1}\prod_{j_2=1}^{M_2} \int_{\calS^{n_{1,j_1}}(\kappa_{1,j_1})\times\calS^{n_{2,j_2}}(\kappa_{2,j_2})} \mathbf{m}(d\mathbf{q}^{(n_{1,j_1})}_{1,j_1}) \mathbf{m}(d\mathbf{q}^{(n_{2,j_2})}_{2,j_2})\Big]\\
&\hspace{2cm}\times\int_{\Delta_{n_{1,j_1}}(z)} d\mathbf{u}^{(n_{1,j_1})}_{1,j_1} e^{iG_n(\mathbf{u}^{(n_{1,j_1})}_{1,j_1},\mathbf{q}^{(n_{1,j_1})}_{1,j_1})}  \phi^0_\omega(Q_{1,j_1,n_{1,j_1}})\\
&\hspace{2cm}\times   \int_{\Delta_{n_{2,j_2}}(z)} d\mathbf{u}^{(n_{2,j_2})}_{2,j_2}e^{-iG_n(\mathbf{u}^{(n_{2,j_2})}_{2,j_2},\mathbf{q}^{(n_{2,j_2})}_{2,j_2})}  \overline{\phi^0_\omega(Q_{2,j_2,n_{2,j_2}})} \\
&\hspace{2cm}\times C^{n'}_\mathfrak{H}\sum_{\mathcal{F}_{\mathbf{n}}}\prod_{(\alpha,\beta)\in\mathcal{F}_{\mathbf{n}}}\frac{\hat{R}(q_{\alpha},q_\beta)}{\vert u_{\alpha}-u_\beta\vert^{\mathfrak{H}}},
\end{split}
\]
where the sum runs over the pairings $\mathcal{F}_{\mathbf{n}}$ of 
\[
\mathbf{I}_{\mathbf{n}}:=\Big\{ (i,j_i,m_{i,j_i}) \in\{1,2\}\times\{1,\dots,M_i\}\times \{1,\dots,n_{i,j_i}\} \Big\}.
\]
Moreover, in the same way as in Proposition \ref{moment1},
\[\begin{split}
C^{n'}_\mathfrak{H}\sum_{\mathcal{F}_{\mathbf{n}}}&\prod_{(\alpha,\beta)\in\mathcal{F}_{\mathbf{n}}}\frac{\hat{R}(q_{\alpha},q_{q_\beta)}}{\vert u_{\alpha}-u_\beta\vert^{\mathfrak{H}}}\\
&=\lim_{A\to +\infty}\E\Big[\prod_{j_1=1}^{M_1}\prod_{m_{1,j_1}=1}^{n_{1,j_1}}   \int_{-A}^A \frac{e^{i r_{1,j_1,m_{1,j_1}} u_{1,j_1,m_{1,j_1}} }}{\vert r_{1,j_1,m_{1,j_1}} \vert^{H-1/2}}w(dr_{1,j_1,m_{1,j_1}},q_{1,j_1,m_{1,j_1}})\\
&\qquad \times \prod_{j_2=1}^{M_2}\prod_{m_{2,j_2}=1}^{n_{2,j_2}} \overline{  \int_{-A}^A \frac{e^{i r_{2,j_2,m_{2,j_2}} u_{2,j_2,m_{2,j_2}} }}{\vert r_{2,j_2,m_{2,j_2}} \vert^{H-1/2}}w(dr_{2,j_2,m_{2,j_2}},q_{2,j_2,m_{2,j_2}}) }\Big],
\end{split}\]
so that the proof is concluded by dominated convergence and the Fubini Theorem.
\end{proof}

\subsubsection{Conclusion}

We have now  everything needed to conclude the proof of convergence of $\calX_\omega^\e$. Consider first the limiting process $\calX_\omega$ solution to \fref{fracit}. Thanks to \fref{conservX}, the moment generating function of the random variable $Y(z)=\sum_{j=1}^M \langle \calX_{\omega_j}(z),\phi_j \rangle $ is perfectly defined for $z$ fixed in $[0,L]$, so that the law of $Y(z)$ is uniquely defined by its moments. Then, since $\calX_\omega^A$ converges in probability to $\calX_\omega$ according to Proposition \ref{cvXA}, and since all  moments of $Y^A(z):=\sum_{j=1}^M \langle \calX_{\omega_j}^A(z),\phi_j \rangle$ (test functions $\phi_j$ in $\calC^\infty_0(\Rm^2)$ are sufficient by density) converge according to Proposition \ref{momentmulti}, they necessarily converge to those of $Y(z)$. Furthermore, since the limits of the moments of $Y^\eps(z):=\sum_{j=1}^M \langle \calX_{\omega_j}^\eps(z),\phi_j \rangle$ are the same as those of $Y^A(z)$ according to Proposition \ref{momentmulti}, we conclude that the moments of $Y^\eps(z)$ converge to the moments of $Y(z)$. Proposition \ref{tightness} finally implies that $(\calX^\e_{\omega_1}(z),\dots,\calX^\e_{\omega_M}(z))_{\e}$ converges in law in $L^2(\Rm^2)$ equipped with the weak topology to $(\calX_{\omega_1}(z),\dots,\calX_{\omega_M}(z))$.

Finally,  convergence in law in $L^2(\Rd)$ for the strong topology is obtained thanks to Lemma \ref{consXe}, the Skorohod's representation theorem \cite[Theorem 6.7 pp.70]{billingsley}, and the following relation
 \[\lim_{\e\to0}\|\calX^\e_\omega (z)\|_{L^2(\Rd)}=\lim_{\e\to0}\|\phi^\e_\omega\|=\frac{1}{2}\|\hat{f}_0(\omega,\cdot )\|_{L^2(\Rd)}=\|\calX_\omega (z)\|_{L^2(\Rd)}.\]
This concludes the proof of Proposition \ref{ItoSc}.

\section{Proof of Theorem \ref{mainth1}}\label{proofpulse}

The proof is a direct application of Propositions \ref{evaprop} and \ref{noback} and Theorem \ref{mainth2}. As already mentioned in Section \ref{outline}, owing the convergence results of Propositions \ref{evaprop} and \ref{noback}, it is enough to check the convergence in law of 
\[p^\e_{2,L}(t,x):=\iint  e^{-i\omega t}e^{i\kappa\cdot x}\psi^\e_\omega(\kappa) \calX^{\e}_\omega(L,\kappa)d\omega d\kappa.
\]
Note that we cannot directly use here any continuity arguments of the map $(\calX^{\e}_\omega)_\omega \mapsto (p^\e_{2,L})_{t,x}$ for the appropriate topology since we only previously obtained pointwise information about $\calX^{\e}_\omega$ in the variable $\omega$. 

The proof is then done in three steps. First, we prove the tightness of $p^\e_{2,L}$ in the space $\mathcal{C}^0([-T,T],L^2_w(\Rm^2))$ for all $T>0$, where $L^2_w(\Rm^2)$ stands for the space $L^2(\Rm^2)$ equipped with the weak topology. Second, we characterize all the accumulation points using the convergence of the moments of $\calX^{\e}_\omega$. Finally, we obtain the convergence in $L^2((-\infty,+\infty)\times\Rd)$ and then in $\mathcal{C}^0([-T,T],L^2(\Rd))$, where $L^2(\Rd)$ is equipped with the strong topology.

We will use the notation
\[
p^\e_{2,L}(t,\phi):=\langle  p^\e_{2,L}(t,\cdot), \phi \rangle = (2 \pi)^2 \langle  \tilde{p}^\e_{2,L}(t,\cdot),\hat{\phi} \rangle,
\]
with 
\[   \tilde{p}^\e_{2,L}(t,\kappa)= \int   e^{-i\omega t} \psi^\e_\omega(\kappa) \calX^{\e}_\omega(L,\kappa)d\omega.\]
According to \cite[Theorem 7.3 pp. 70]{billingsley}, tightness is a consequence of the following lemma:    
\begin{lem}\label{tightnesspulse}
We have, for all $\phi\in L^2(\Rm^2)$,
\[
\lim_{M\to+\infty}\overline{\lim_{\e\to  0}}\Pro\Big(\sup_{t\in[-T,T]} \vert p^\e_{2,L}(t,\phi)\vert>M\Big)=0,
\]
and for all $\eta>0$,
\[
\lim_{\tau \to 0}\overline{\lim_{\e\to  0}}\Pro\big( \sup_{ \rvert t_1-t_2\lvert \leq \tau}\vert p^\e_{2,L}(t_1,\phi)-p^\e_{2,L}(t_2,\phi)\vert  >\eta\big)=0.
\]
\end{lem}
\begin{proof} This lemma is a  direct consequence of Lemma \ref{consXe}. Let indeed $\phi\in L^2(\Rd)$. For the first point, we have  
\[\overline{\lim_{\e\to0}}\E\big[\sup_{t\in[-T,T]}\big\vert p^\e_{2,L}(t,\phi) \big\vert\Big] \leq C \|\hat{\phi}\|_{L^2(\Rd)} \int d\omega \; \overline{\lim_{\e\to0}}\E\big[ \|\calX^{\e}_\omega(L)\|_{L^2(\Rd)}\big]  <+\infty,\]
and the conclusion follows from the Markov inequality. In the same way, we have for the second point, for all $\tau >0$,
\[
\overline{\lim_{\e\to0}}\E\left[\sup_{\rvert u_1-u_2\lvert \leq \tau}\vert  p^\e_{2,L}(t_1,\phi)- p^\e_{2,L}(t_2,\phi) \lvert \right]\leq C \tau \int d\omega \vert \omega \vert \overline{\lim_{\e\to0}}\E\left[ \|\calX^\e_\omega(L)\|_{L^2(\Rd)}\right],
 \]    
which concludes the proof of Lemma \ref{tightnesspulse}.
\end{proof}

In order to identify the accumulation points, we consider the finite-dimensional distributions of $p^\e_{2,L}$. We remark first that thanks to Lemma \ref{consXe},
\[\sup_{t\in[-T,T]}\|p_{2,L}^\e(t)\|_{L^2(\Rm^2)}\] 
is uniformly bounded in $\e$ by a deterministic constant, and therefore that the finite-dimensional distributions are uniquely characterized by their moments. Using once again Lemma \ref{consXe} in order to justify the use of the Fubini Theorem and dominated convergence, it follows from Proposition \ref{ItoSc} that, for all $M\in \mathbb{N}^\ast$, $(t_{m})_{m\in\{1,\dots,M\}}\in[-T,T]^{M}$, $( \phi_{m})_{m\in\{1,\dots,M\}}\in (L^2(\Rm^2))^{M}$, we have 
\[
\lim_{\e\to 0}\mathbb{E}\Big[\prod_{m=1}^{M}p^\e_{2,L}( t_{m}, \phi_{m}) \Big]= \mathbb{E}\Big[\prod_{m=1}^{M} p^0_{L}( t_{m}, \phi_{m})\Big].
\]
In order to go back to the original pulse $p_{L}^\e$ defined by \eqref{pulse22}, we remark first that \fref{CVpp1} holds for all $\phi \in L^2(\Rm^2)$ thanks to the bound \fref{L23} and the density of $\calC_0^\infty(\Rm^2)$ in $L^2(\Rm^2)$. Hence, using Proposition \ref{noback}, we obtain the convergence in law of $p_{L}^\e$  to $p^0_{L}$ in $\mathcal{C}^0([-T,T],L_w^2(\Rm^2))$. To conclude, we use the Skorohod's representation theorem \cite[Theorem 6.7 pp.70]{billingsley}: there exist a probability space $(\tilde{\Omega},\tilde{\calT},\tilde{\Pro})$ and random variables $\mathfrak{p}_{L}^\e$ and $\mathfrak{p}_{L}^0$, with the same laws as $p_{L}^\e$ and $p_{L}^0$, respectively, and such that  
\[\lim_{\e\to0}\sup_{t\in[-T,T]}\vert \mathfrak{p}_{L}^\e(t,\phi)- \mathfrak{p}_{L}^0(t,\phi)\vert =0\qquad\tilde{\Pro}-\text{a.s},\]
for all $\phi\in L^2(\Rm^2)$. A direct consequence is that $\mathfrak{p}_{L}^\e$ converges $\tilde{\Pro}$-a.s. to $\mathfrak{p}_{L}^0$ in $L^2_w((-\infty,+\infty)\times\Rd)$, since using \eqref{L24} one has 
\[ \overline{\lim_{\e\to0}}\|  \mathfrak{p}_{L}^\e \|_{L^2((\infty,+\infty)\times \Rd)} \leq \frac{1}{2}\| f_0 \|_{L^2((-\infty,+\infty)\times \Rd)},\]
and the unit ball of $L^2((-\infty,+\infty)\times \Rd)$ is weakly compact. Moreover, this convergence also holds in $L^2((-\infty,+\infty)\times \Rd)$ with the strong topology because of the conservation relation
\[\frac{1}{2}\| f_0 \|_{L^2((-\infty,+\infty)\times \Rd)}= \|  \mathfrak{p}_{L}^0 \|_{L^2((-\infty,+\infty)\times \Rd)}=\lim_{\e\to0}\|  \mathfrak{p}_{L}^\e \|_{L^2((\infty,+\infty)\times \Rd)}.\]
As a result, using the Plancherel theorem 
\[I_\e:=\int d\omega dx \vert \tilde{\mathfrak{p}}^\e_ L(\omega,x)-\tilde{\Psi}_\omega(L,x)\vert^2\underset{\e\to0}{\longrightarrow}0\qquad\tilde{\Pro}-\text{a.s}\]
where
\[ \tilde{\mathfrak{p}}^\e_\omega (L,x)=\frac{1}{2\pi}\int e^{i\omega t} \mathfrak{p}_{L}^\e(t,x)  dt \qquad\text{and}\qquad\tilde{\Psi}_\omega(L,x)=\frac{1}{2\pi}\int e^{i\omega t} \mathfrak{p}_{L}^0(t,x)  dt.\]
Since $\hat{f}_0(\omega,\kappa)$ has  a compact support with respect to $\omega$, so do  $\tilde{\mathfrak{p}}^\e_\omega$ and $\tilde{\Psi}_\omega$ according to \fref{L23} and \fref{consITO}. The Jensen's inequality then yields 
\[\sup_{t\in[-T,T]} \| \mathfrak{p}_{L}^\e(t,\cdot)- \mathfrak{p}_{L}^0(t,\cdot)\|_{L^2(\Rd)}\leq C I_\e. \]
This proves the convergence in $\calC^0([-T,T],L^2(\Rd))$ and concludes the proof of Theorem \ref{mainth1} since almost sure convergence implies convergence in law.

\section{Proof of estimate \fref{L2p}} \label{proofL2p}

We use here the notation of sections \ref{intro} and \ref{outline}. The core of the proof is the following lemma:
\begin{lem} \label{estest}We have the estimate,
\[ 
\| \hat{p}^\eps_\omega \|_{L^2(\Rm^3)} \leq \frac{C}{\eps \sqrt{\alpha_\eps}} \| \hat{f}_0(\omega,\cdot)\|_{L^2(\Rm^2)}.
\]
\end{lem}
\begin{proof}
We already have an estimate on $(0,L)$ according to Lemma \ref{borneL2}, which is
\[ 
\| \hat{p}^\eps_\omega \|_{L^2((0,L)\times \Rm^2)} \leq \frac{C}{\eps \sqrt{ \alpha_\eps}} \| \hat{f}_0(\omega,\cdot)\|_{L^2(\Rm^2)},
\]
so that it remains to treat the domain $z \notin (0,L)$. Consider first the case $z>L$. Plugging $z=L$ into \fref{formulea} and \fref{formuleb}, and using the fact that $\hat{b}_\omega^\eps(L,\kappa)=0$, we find that $\hat{p}^\eps_\omega$ reads
\[
\hat{p}^\eps_\omega(z,\kappa)= e^{i k_\omega \lambda_{\eps,\omega}(\kappa)(z-L)/\e^2}\hat{p}^\eps_\omega(L,\kappa), \qquad z>L.
\]
We need to split the domain of integration in $\kappa$ in order to obtain appropriate estimates. Suppose first that $\eps^2 |\kappa|^2/k_\omega^2 \geq 1$ (evanescent modes), then $\textrm{Im}(\lambda_{\eps,\omega}(\kappa)) \geq C \sqrt{\alpha_{\eps,\omega}}$ by \fref{expsqrt}. Together with \fref{L23}, this yields
\[
\int_{L}^\infty dz \int_{\{\eps^2 |\kappa|^2/k_\omega^2 \geq 1\}}  d\kappa| \hat{p}^\eps_\omega (z,\kappa)|^2  \leq \frac{C}{\sqrt{\alpha_\eps} }\| \hat{f}_0(\omega,\cdot)\|^2_{L^2(\Rm^2)}.
\]
When $\eps^2 |\kappa|^2/k_\omega^2 \leq 1$ (propagative modes), consider the strictly decreasing function $f_b(x)=\sqrt{x^2+b^2}-x$, for $x \geq 0$. It satisfies $f_b(x) \geq f_b(1) \geq C b^2$ for $x\in[0,1]$, $0<b\ll1$ and some constant $C$. This then yields $\textrm{Im}(\lambda_{\eps,\omega}(\kappa)) \geq C \alpha_{\eps,\omega}$ for $x=1-\eps^2 |\kappa|^2/k_\omega^2$. Together with  \fref{L23}, we find
\[
\int_{L}^\infty dz \int_{\{\eps^2 |\kappa|^2/k_\omega^2 \leq 1\}} d\kappa | \hat{p}^\eps_\omega (z,\kappa)|^2  \leq \frac{C}{\eps^2 \alpha_\eps} \| \hat{f}_0(\omega,\cdot)\|^2_{L^2(\Rm^2)}.
\]
We turn now to the case $z \in (L_S,0)$. We have
\[
 \hat{p}^\e_\omega(z,\kappa)=\frac{1}{\sqrt{\lambda_{\e,\omega}(\kappa)}}\Big( \hat{a}^\e_\omega(L_S^+,\kappa)e^{ik_{\omega}\lambda_{\e,\omega}(\kappa)z/\e^2}+ \hat{b}^\e_\omega(0,\kappa)e^{-ik_{\omega}\lambda_{\e,\omega}(\kappa)z/\e^2}\Big).
\]
Owing \fref{expa}, the first term of the r.h.s is direct and yields a control by $C \| \hat{f}_0(\omega,\cdot)\|_{L^2(\Rm^2)}$. For the second one, we write 
\be \label{relab}
\frac{\hat{b}^\e_\omega(0,\kappa)}{\sqrt{\lambda_{\e,\omega}(\kappa)}}=\hat{p}^\e_\omega(0,\kappa)-\frac{\hat{a}^\e_\omega(L_S^+,\kappa)}{\sqrt{\lambda_{\e,\omega}(\kappa)}},
\ee
and obtain, thanks to \fref{L23}, again a control by $C \| \hat{f}_0(\omega,\cdot)\|_{L^2(\Rm^2)}$. Consider finally the case $z<L_S$, for which
\[
\hat{p}^\eps_\omega(z,\kappa)= e^{i k_\omega \lambda_{\eps,\omega}(\kappa)(L_S-z)/\e^2}\hat{p}^\eps_\omega(L_S^-,\kappa), \qquad z<L_S.
\]
The jump condition \fref{jump} yields $\hat{p}^\eps_\omega(L_S^-,\kappa)=\hat{p}^\eps_\omega(L_S^+,\kappa)-\hat f_0(\omega,\kappa)$, which, together with \fref{expa} and \fref{relab}, gives the expression
\[
\hat{p}^\eps_\omega(L_S^-,\kappa)=-\frac{\hat f_0(\omega,\kappa)}{2} \left(1+e^{-2i k_\omega \lambda_{\eps,\omega}(\kappa)L_S/\e^2}\right)+\hat{p}^\eps_\omega(0,\kappa)
e^{-i k_\omega \lambda_{\eps,\omega}(\kappa)L_S/\e^2}.
\]
Using again \fref{L23}, we then proceed as in the case $z>L$ and obtain the same estimate. Putting together all previous estimates ends the proof of the lemma.
\end{proof}
Owing the previous lemma, we can proceed to the proof. Let $v=P-p$, which satisfies
\[
\Delta v-\frac{1}{c^2_0}\Big(1+\e^s V\Big(\frac{z}{\e},\frac{x}{\e}\Big)\mathbf{1}_{(0,L)}(z)\Big)\partial^2_t v = i \alpha_\eps p, \qquad (t,z,x)\in (0,+\infty)\times\mathbb{R}\times\Rd,
\]
equipped with $v(0,\cdot)=\partial_t v(0,\cdot)=0$. Since $V$ is uniformly bounded by a deterministic constant, and $p \in \calC^0((0,+\infty),L^2(\Rm^3))$ according to Lemma \ref{estest}, it is a classical problem to construct solutions to the above equation which satisfies the energy conservation relation
\begin{align*}
&\frac{1}{2} \frac{d}{dt} \left( \| \nabla v(t)\|^2_{L^2(\Rm^3)}+\frac{1}{c^2_0}\int_{\Rm^2} \Big(1+\e^s V\Big(\frac{z}{\e},\frac{x}{\e}\Big)\mathbf{1}_{(0,L)}(z)\Big)| \partial_t v(t,z,x)|^2 dzdx\right)\\
& \hspace{3cm}= \textrm{Re} \left(i \alpha_\eps \int_{\Rm^3} p(t,z,x) \partial_t \overline{v}(t,z,x) dx dz \right).
\end{align*}
After integration and the use of the Cauchy-Schwarz and Young inequalities, it follows that
\[
\| \nabla v(t)\|^2_{L^2(\Rm^3)}+\| \partial_t v(t)\|^2_{L^2(\Rm^3)} \leq C \alpha_\eps^2 \int_0^t \| p(s)\|^2_{L^2(\Rm^3)} ds +C\int_0^t \| \partial_t v(s)\|^2_{L^2(\Rm^3)}.
\]
Since $v(t=0,\cdot)=0$, we can use the Poincar\'e inequality (w.r.t. $t$), which together with the Gronwall lemma yield
\[
\| v(t)\|^2_{H^1(\Rm^3)}+\| \partial_t v(t)\|^2_{L^2(\Rm^3)} \leq C \alpha_\eps^2 \int_0^t \| p(s)\|^2_{L^2(\Rm^3)} ds. 
\]
In order to apply Lemma \ref{estest}, we notice that
\[
\| p(s)\|_{L^2(\Rm^3)}=\eps \| p(s,\eps \cdot)\|_{L^2(\Rm^3)} \leq \eps \int_{\Rm} \| \hat{p}^\eps_\omega \|_{L^2(\Rm^3)} d\omega.
\]
Above, we used the fact that $\hat{p}^\eps_\omega $ has compact support according to the aforementioned lemma since $f_0$ does.
Standard Sobolev embeddings then yield
\[
\sup_{(t,z) \in (0,T) \times \Rm}\| v(t,z, \cdot)\|_{L^2(\Rm^2)}\leq  C \alpha_\eps^{1/2},
\]
which concludes the proof after rescaling $x$ by $\eps x$.


\end{document}